\documentclass{amsart}
\usepackage[margin=1.5in]{geometry}
\usepackage[numeric]{amsrefs}
\usepackage[utf8]{inputenc}
\usepackage{tikz-cd}
\usetikzlibrary{decorations.pathmorphing}
\tikzstyle{printersafe}=[decoration={snake,amplitude=0pt}]
\usepackage{quiver}
\usepackage{amsmath}
\usepackage{mathtools} 
\usepackage{amsfonts}
\usepackage{amssymb} 
\usepackage{amsthm}
\usepackage{mathrsfs}
\usepackage{stmaryrd}
\usepackage[hidelinks]{hyperref}
\usepackage[normalem]{ulem}
\mathtoolsset{showonlyrefs}
\newtheorem{Def}[subsubsection]{Definition}
\newtheorem{definition}[subsubsection]{Definition}
\newtheorem{Thm}[subsubsection]{Theorem}
\newtheorem{Lem}[subsubsection]{Lemma}
\newtheorem{Prop}[subsubsection]{Proposition}
\newtheorem{Cor}[subsubsection]{Corollary}
\newtheorem{Conj}{Conjecture}
\newtheorem*{Thm*}{Theorem}
\newtheorem{mainThm}{Theorem}
\newcommand{\mbb}[1]{\mathbb{#1}}

\theoremstyle{remark}
\newtheorem{Eg}[subsubsection]{Example}
\newtheorem{Rem}[subsubsection]{Remark}
\newtheorem{Claim}[subsubsection]{Claim}

\newtheorem{Q}[subsubsection]{Question}
\numberwithin{equation}{subsection}
\usepackage{calligra}

\DeclareMathOperator{\Hom}{\mathscr{H}\text{\kern -3pt {\calligra\large om}}\,}
\DeclareMathOperator{\End}{\mathscr{E}\text{\kern -3pt {\calligra\large nd}}\,}

\newcommand{\Mod}{\operatorname{Mod}}
\newcommand{\spec}{\operatorname{Spec}}
\newcommand{\spf}{\operatorname{Spf}}
\newcommand{\spa}{\operatorname{Spa}}

\newcommand{\qp}{\mathbb{Q}_{p}}
\newcommand{\zp}{\mathbb{Z}_{p}}
\newcommand{\locan}{\mathrm{la}}
\newcommand{\Lie}{\operatorname{Lie}}

\newcommand{\et}{\text{\'et}}

\newcommand{\la}{\mathrm{la}}
\newcommand{\LA}{\mathcal{HT}}
\newcommand{\CD}{\mathcal{CD}}
\newcommand{\LAD}{\mathcal{HT}_{\mathrm{fh}}}
\newcommand{\LAE}{\mathcal{HT}_{\et}}

\newcommand{\coker}{\operatorname{Coker}}
\newcommand{\perf}{\operatorname{Perf}}
\newcommand{\prof}{\operatorname{Prof}}
\newcommand{\cond}{\mathrm{cond}}
\newcommand{\mv}{\mathrm{v}}

\newcommand\restr[2]{{% we make the whole thing an ordinary symbol
  \left.\kern-\nulldelimiterspace % automatically resize the bar with \right
  #1 % the function
  \vphantom{\big|} % pretend it's a little taller at normal size
  \right|_{#2} % this is the delimiter
  }}
\newcommand{\twopartdef}[4]
{
	\left\{
		\begin{array}{ll}
			#1 & \mbox{if } #2 \\
			#3 & \mbox{if } #4
		\end{array}
	\right.
}

\newcommand{\gmhateta}{\widehat{\mathbb{G}}_{m, \eta}}
\newcommand{\gmhat}{\widehat{\mathbb{G}}_{m}}
\newcommand{\gm}{\mathbb{G}_{m}}

\newcommand{\Sym}{\operatorname{Sym}}
\newcommand{\Cont}{C^{0}}

\newcommand{\Spf}{\operatorname{Spf}}

\newcommand{\Spa}{\operatorname{Spa}}
\newcommand{\Spd}{\operatorname{Spd}}
\newcommand{\spd}{\operatorname{Spd}}

\newcommand{\Gal}{\mathrm{Gal}}
\newcommand{\bdd}{\mathrm{bdd}}
\newcommand{\HT}{\mathrm{HT}}
\newcommand{\ul}[1]{\underline{#1}}

\newcommand{\Igkatz}{\mathfrak{Ig}_{\mathrm{Katz}}}

\title{$p$-adic Fourier theory in families}

\author{Andrew Graham}
\address{Mathematical Institute, University of Oxford, Woodstock Road, Oxford OX2 6GG, United Kingdom}
\email{andrew.graham@maths.ox.ac.uk}

\author{Pol van Hoften} 
\address{Department of Mathematics, Vrije Universiteit Amsterdam, De Boelelaan 1111, 1081 HV Amsterdam, The Netherlands}
\email{p.van.hoften@vu.nl}

\author{Sean Howe}
\address{Department of Mathematics, University of Utah, 155 S 1400 E, Salt Lake City, UT 84112}
\email{sean.howe@utah.edu}

\begin{document}
\begin{abstract}
We construct Fourier transforms relating functions and distributions on finite height $p$-divisible rigid analytic groups and objects in a dual category of $\mathbb{Z}_p$-local systems with analyticity conditions. Our Fourier transforms are formulated as isomorphisms of solid Hopf algebras over arbitrary small v-stacks, and generalize earlier constructions of Amice and Schneider--Teitelbaum. We also construct compatible integral Fourier transforms for $p$-divisible groups and their dual Tate modules. As an application, we use the Weierstrass $\wp$-function to construct a global Eisenstein measure over the $p$-adic modular curve, extending previous constructions of Katz over the ordinary locus and at CM points, and show its generic fiber, the global Eisenstein distribution, gives rise to new families of quaternionic modular forms that overconverge from profinite sets in the rigid analytic supersingular locus.  
\end{abstract}
\maketitle
\tableofcontents

\section{Introduction} For $A$ a locally compact abelian group with Pontryagin dual $A^\vee = \hom(A, S^1)$, there is a universal character 
\[ 
\kappa: A \times A^\vee \rightarrow S^1, \;
(x,\varphi) \mapsto \varphi(x). \]
The Fourier transforms of classical functional analysis are isomorphisms between spaces of measures or distributions on $A$ and spaces of functions on $A^\vee$ obtained by integrating $\kappa$ over fibers of the projection to $A$. Pontryagin duality is an involution on the category of locally compact abelian groups preserving universal characters, and thus these Fourier transforms are self-dual. They are moreover functorial, exchange multiplication and convolution, and exchange invariant derivations with multiplication by canonical coordinate functions. 

In this paper, we construct Fourier transforms with similar properties in $p$-adic geometry.  Unlike the classical setting, we do not know a natural category analogous to locally compact abelian groups that is preserved under Pontryagin duality. Instead, we develop a theory similar to the duality between compact groups and their discrete duals: In our case, we consider $\mathbb{Z}_p$-local systems with an analyticity condition, whose dual character groups are finite height $p$-divisible rigid analytic groups (in the sense of Fargues \cite{FarguesI, FarguesII}). 

We study these dual categories over an arbitrary small $v$-stack, and construct universal characters in this generality. We then show that integration over the universal character defines dual Fourier isomorphisms between solid Hopf algebras of distributions and functions, exchanging invariant derivations with multiplication by coordinate functions. When the base is the spectrum of a non-archimedean extension of $\mathbb{Q}_p$ and the $p$-divisible rigid analytic group is the character group of $\mathbb{Z}_p$, one of our Fourier transforms specializes to Amice's \cite{AmiceInterpolation} description of the dual of the space of locally analytic functions on $\mathbb{Z}_p$. Similarly, we also recover a previous extension of Amice's results due to Schneider--Teitelbaum \cite{SchneiderTeitelbaumFourier} --- in independent work, Kings--Sprang \cite{KingsSprang} also provide a generalization of \cite{SchneiderTeitelbaumFourier} over non-archimedean extensions of $\mathbb{Q}_p$. We also prove a compatibility between our Fourier theory and integral Cartier duality, generalizing the compatibility between the Amice transform and Mahler expansions. 

As  an application of our results, we construct an Eisenstein measure on the Tate module of the universal elliptic curve over the integral $p$-adic modular curve (viewed as a $p$-adic formal scheme). Our measure specializes to Katz's Eisenstein measure over the ordinary locus, and to a related construction of Katz at supersingular complex multiplication points. Passing to the associated distribution on the rigid analytic generic fiber yields a global Eisenstein distribution. Using this distribution, we recover the overconvergence from the ordinary locus of the classical weight family of $p$-depleted Eisenstein series parameterized by characters of $\mathbb{Z}_p^\times$. By a  similar argument we then construct, for $R$ any order in a degree two extension $L/\mathbb{Q}_p$, a new family of quaternionic Eisenstein series parameterized by $L$-analytic characters of $R^\times$, and show that they overconverge from the profinite locus of formal complex multiplication by $R$ to open subsets of the supersingular locus of the $p$-adic modular curve. 

\subsection{Fourier theory for rigid analytic character groups} We first state our main theorem when the base is the spectrum of a non-archimedean field $K/\mathbb{Q}_p$. To that end, let $K/\mathbb{Q}_p$ be a non-archimedean extension, let $\overline{K}$ be an algebraic closure, and let $C$ be the completion of $\overline{K}$. We use condensed $K$-vector spaces (see \cite{CondensedNotes}) in place of the usual theory of locally convex $K$-vector spaces to handle topologies on function and distribution spaces.

We fix a continuous representation of $\mathrm{Gal}(\overline{K}/K)$ on a finite free $\mathbb{Z}_p$ module $\Lambda$, a finite dimensional $K$-vector space $V$, and a $\mathrm{Gal}(\overline{K}/K)$-equivariant surjection $\gamma: \Lambda_C \twoheadrightarrow V_{C}$. One half of our Fourier transforms will concern functions and distributions on $\Lambda$ with an analyticity condition determined by $\gamma$. Explicitly, we define $\mathcal{O}^{\gamma-\locan}(\Lambda)$ to be the natural condensed\footnote{Sending a profinite set $S$ to the $\Gal(\overline{K}/K)$-equivariant $\gamma$-locally analytic functions from $\Lambda$ to $\Cont(S, C)$.} $K$-vector space of functions from $\Lambda$ to $C$ that are $\Gal(\overline{K}/K)$-equivariant and $\gamma$-locally analytic, i.e., locally on $\Lambda$, factor as the composition of $\gamma$ with an analytic function on an open ball in $V_C$. We define $\mathcal{D}^{\gamma-\locan}(\Lambda)$ to be its condensed dual. 

On the other side, we consider the rigid analytic group $H_{\gamma}^{\mathrm{rig}}$ over $K$  parametrizing $\gamma$-analytic characters of $\Lambda$. This is a $p$-divisible rigid analytic group in the sense of \cite{FarguesI}, and, by a theorem of Fargues, see \cite[Corollaire 17]{FarguesI}, every $p$-divisible rigid analytic group over $K$ is isomorphic to some $H_{\gamma}^{\mathrm{rig}}$. We denote the associated diamond over $\Spd K$ by $H_\gamma$, and we define $\mathcal{O}(H_{\gamma})$ to be the natural condensed\footnote{Sending a profinite set $S$ to $\mathcal{O}(H_{\gamma} \times S)$, which makes sense as $H_\gamma \times S$ is a diamond over $\Spd K$. In this case $\mathcal{O}(H_\gamma)$ is equal to the condensed set $\ul{\mathcal{O}(H_\gamma^{\mathrm{rig}})}$ associated to $\mathcal{O}(H_\gamma^{\mathrm{rig}})$ with its usual Fr\'echet topology.} $K$-vector space of functions on $H_\gamma$. We write $\mathcal{D}(H_{\gamma})$ for the condensed dual of $\mathcal{O}(H_{\gamma})$. \smallskip 

The condensed $K$-vector spaces $\mathcal{O}^{\gamma-\locan}(\Lambda),$ $\mathcal{D}^{\gamma-\locan}(\Lambda), $ $\mathcal{O}(H_{\gamma}),$ and $\mathcal{D}(H_{\gamma})$ are solid\footnote{We recall that the property of being solid is a robust type of completeness in this setting.}, and we show that the group structures on $\Lambda$ and $H_{\gamma}$ endow them with the structure of solid Hopf $K$-algebras, i.e., Hopf algebras with respect to the solid tensor product $-\otimes^{\blacksquare}_K-$. We also show that there is a universal $\gamma$-locally analytic character $\kappa: \Lambda \times H_{\gamma} \to \gmhateta^\lozenge$, where $\gmhateta$ is the rigid generic fiber of the formal multiplicative group $\gmhat$ over $\spf \mathcal{O}_K$ and $(-)^\lozenge$ denotes the associated diamond. Our first main theorem concerns the Fourier transforms for this data:
\begin{mainThm}\label{thm.rational-fourier-theory}
Integration against the universal character $\kappa$ defines isomorphisms of solid Hopf algebras over $K$
    \[ \mathcal{D}(H_\gamma) \xrightarrow{\sim} \mathcal{O}^{\gamma-\locan}(\Lambda) \textrm{ and } \mathcal{D}^{\gamma-\locan}(\Lambda) \xrightarrow{\sim} \mathcal{O}(H_{\gamma}). \]
These isomorphisms are functorial in the triple $(\Lambda, V, \gamma)$ and naturally dual to each other.
\end{mainThm}

That these transforms respect the Hopf algebra structures is an analog of the exchange of convolution and multiplication of functions for the classical Fourier transform on $\mathbb{R}$. 

\begin{Eg}
Suppose that $L$ is a finite extension of $\qp$ contained in $K$. Take $\Lambda=\mathcal{O}_L$ with the trivial Galois action, take $V=K$, and let $\gamma:\mathcal{O}_L \otimes_{\zp} C \to C$ be the surjection coming from the embedding of $L$ in $C$. Then $\mathcal{O}^{\gamma-\locan}(\Lambda)$ can be identified with the space of functions $\mathcal{O}_L \to K$ that are $L$-locally analytic in the sense that they can be written locally on $\mathcal{O}_L$ as a convergent power series in one variable with coefficients in $K$. In this setting, the second isomorphism in Theorem \ref{thm.rational-fourier-theory} recovers the $p$-adic Fourier theory of Schneider--Teitelbaum (\cite[Theorem 2.3]{SchneiderTeitelbaumFourier}), see Proposition \ref{Prop:Comparison}. As in \cite[\S 3]{SchneiderTeitelbaumFourier}, the character group $H_\gamma$ is a twisted form of the generic fiber of a one-dimensional Lubin--Tate formal group for $L$. 
\end{Eg}

Our Fourier transforms also satisfy equivariance properties with respect to actions of $\Sym^\bullet V$ and $\Sym^\bullet V^*$. To define these actions, we first note that we can identify $\Lie H_{\gamma}$ with $V^*$, and that there is a natural logarithm map $\log_{H_{\gamma}}: H_\gamma \rightarrow \Lie H_{\gamma}=V^*$. We equip $\mathcal{O}(H_\gamma)$ with the action of $\Sym^\bullet V^*$ by invariant differential operators, and we equip $\mathcal{O}^{\gamma-\locan}(\Lambda)$ with the action of $\mathrm{Sym}^\bullet V^*$ by multiplication by polynomial functions on $V$ pulled back along $\gamma$. Similarly, we equip $\mathcal{O}^{\gamma-\locan}(\Lambda)$ with the action $\Sym^\bullet V$ by invariant differential operators, and we equip $\mathcal{O}(H_\gamma)$ with the action of $\mathrm{Sym}^\bullet V$  by multiplication by polynomial functions on $V^*$ pulled back along $\log_{H_\gamma}$.

\begin{Prop} \label{Prop:Equivariance}
    The Fourier transforms of Theorem \ref{thm.rational-fourier-theory} are equivariant for the natural actions of $\mathrm{Sym}^\bullet V^*$ and $\mathrm{Sym}^\bullet V$ described above. 
\end{Prop}

Proposition \ref{Prop:Equivariance} is analogous to the statement that the classical Fourier transform on $\mathbb{R}$ exchanges multiplication by the coordinate $x$ with differentiation $\frac{d}{dx}$.

\subsubsection{}
Before explaining the ingredients that go into the proofs of Theorem \ref{thm.rational-fourier-theory} and Proposition \ref{Prop:Equivariance}, we make some remarks.
\begin{Rem}
 The space $\mathcal{O}^{\gamma-\locan}(\Lambda) \subseteq \Cont(\Lambda, C)^{\Gal(\overline{K}/K)}$ always contains the space of locally constant functions $\bigcup_n \Cont(\Lambda/p^n \Lambda, C)^{\Gal(\overline{K}/K)}$. Via pullback along $\gamma$, the space $\mathcal{O}^{\gamma-\locan}(\Lambda)$ also contains the polynomial algebra $\mathrm{Sym}^\bullet V^*$. The $K$-algebra spanned by products of these polynomials and locally constant functions is dense in $\mathcal{O}^{\gamma-\locan}(\Lambda)$. 
\end{Rem}

\begin{Rem}
   In Theorem \ref{thm.integral-fourier-theory} we establish an integral analog of Theorem \ref{thm.rational-fourier-theory} for $p$-divisible groups over $\Spf \mathcal{O}_K$ (more generally, $\Spf R$ for any $p$-adically complete ring $R$) as an immediate consequence of Cartier duality for finite flat groups schemes. The integral and rational results satisfy a natural compatibility, see Theorem \ref{thm.integral-rational-compatibility} for a precise statement. 
\end{Rem}

\begin{Rem}
    When $K$ is a $p$-adic field (i.e., discretely valued with perfect residue field), there is an initial/maximal choice of $\gamma$, given by the natural map $\Lambda_C \rightarrow \Lambda_C/\mathrm{im}\theta$ for $\theta$ the Sen operator. When $\mathcal{H}$ is a $p$-divisible group over $\mathcal{O}_K$ and $\Lambda = T_p \mathcal{H}^\vee(C)$, the Hodge--Tate map $T_p \mathcal{H}^\vee(C) \otimes C \rightarrow \omega_{\mathcal{H}} \otimes_{\mathcal{O}_K} C$ can be identified with this initial map and the rigid generic fiber $\mathcal{H}_\eta$ is the character variety associated with $\Lambda=T_p G^\vee(C)$ with its usual Galois action and $\gamma$ the Hodge--Tate map. 
\end{Rem}

\begin{Rem}\label{remark.locally-compact-analog}  We expect to have Fourier transforms for other related objects: For example, for certain spaces of functions on Banach--Colmez spaces of slopes between $0$ and $1$. These Banach--Colmez spaces can be realized (typically in many different ways) as the universal covers, in the sense of \cite{ScholzeWeinstein}, of $p$-divisible rigid analytic groups. The universal cover is an extension of the $p$-divisible rigid analytic group by its Tate module, and thus such a Fourier transform would be intimately related to Theorem \ref{thm.rational-fourier-theory}; we make a precise conjecture and detail this connection in \cite{GHHBC}. More generally, one hopes to find an analog of the full category of locally compact abelian groups where our Fourier theory can be situated; Juan Esteban Rodriguez Camargo has suggested to look for a natural category of group objects in analytic stacks closed under extension and Cartier duality.
\end{Rem}

\begin{Rem}
In \cite{HoweUnipotent}, the integral $p$-adic Fourier theory for $\gmhat$ is used to study the $p$-adic interpolation of Maass--Shimura operators on the space of $p$-adic modular forms. One of our motivations for proving Theorem \ref{thm.rational-fourier-theory} and Proposition \ref{Prop:Equivariance} is to similarly study the $p$-adic interpolation of Maass--Shimura operators on spaces of $p$-adic automorphic forms over the $\mu$-ordinary locus of more general Shimura varieties (as in e.g. \cite{EischenMantovan}). We hope that this will have applications to the construction of $p$-adic $L$-functions (as in e.g. \cite{HarrisSquareRoot}).
\end{Rem}

\subsection{Sketch of proof}\label{ss.sketch-of-proof} \label{Sub:Proof}

We sketch the proof of Theorem~\ref{thm.rational-fourier-theory}: First, the various compatibilities are essentially formal after the Fourier transforms have been constructed. In particular, using the duality, it suffices to prove that $\mathcal{D}^{\gamma-\locan}(\Lambda) \xrightarrow{} \mathcal{O}(H_{\gamma})$ is an isomorphism.  

To obtain this isomorphism we first treat the case where the Galois action is trivial, where the argument is essentially that of Schneider--Teiltelbaum \cite[\S2]{SchneiderTeitelbaumFourier}. More precisely, we begin by considering $V=\Lambda \otimes_{\mathbb{Z}_p} K$ and $\gamma$ the identity map $\Lambda \otimes_{\mathbb{Z}_p}  C \rightarrow \Lambda \otimes_{\mathbb{Z}_p} C$. In this case the $\gamma$-locally analytic functions are the locally analytic functions in the usual sense and $H_\gamma=\Lambda^* \otimes_{\mbb{Z}_p} \widehat{\mbb{G}}_{m,\eta}^\lozenge$. In particular, the Fourier transform in this setting is a multi-dimensional Amice transform, and thus the result is argued by reduction to the results of \cite{AmiceInterpolation}. To deduce the result for arbitrary $V$ but still when the Galois action is trivial, we observe that the $\gamma$-locally analytic functions can be cut out among all locally analytic functions as those annihilated by certain invariant vector fields. By the case of the theorem that has already been established, these are matched on the Fourier dual side with multiplication by coordinate functions composed with the logarithm. Thus the $\gamma$-locally analytic functions are a closed subspace and the associated character group $H_{\gamma}$ is a Zariski closed subvariety of $\Lambda^* \otimes_{\mbb{Z}_p} \widehat{\mbb{G}}_{m,\eta}$. Because we have already established the Fourier transform is an isomorphism on all locally analytic functions, we deduce that it is also an isomorphism on $\gamma$-locally analytic functions via passage to a closed subspace and its dual quotient. 

We obtain the general case of Theorem~\ref{thm.rational-fourier-theory} where the Galois action can be non-trivial by descent from the case of trivial Galois action when $K=C$. 

\subsection{Fourier transforms in families} One of the key ingredients in Theorem \ref{thm.rational-fourier-theory} is our \emph{construction} of the Fourier transforms. To make this construction, we need to work relatively to define the universal $\gamma$-locally analytic character of $\Lambda$, which lives over the base $H_{\gamma}$. Moreover, since we want to prove that Fourier transforms are isomorphisms of condensed sets, we also want to allow our base to be, e.g., $\Spd K \times S$ for $S$ a profinite set. Because the constructions and proofs already involve working over these kinds of general bases, it is natural to ask whether a version of Theorem \ref{thm.rational-fourier-theory} itself holds over a more general base. 

Our approach is to establish a maximally general version by working over arbitrary small v-stacks, and to then explain how it specializes to more concrete settings like in Theorem \ref{thm.rational-fourier-theory}. This relative theory is not just a matter of curiosity or a technical ingredient for the proof over a point --- it also plays an essential role in our construction of a global Eisenstein distribution over the $p$-adic modular curve. 

\subsubsection{Main results} Let $Y$ be a small v-sheaf (or small v-stack) over $\spd \qp$, i.e., a sheaf on the category of affinoid perfectoid spaces over $\qp$ equipped with the v-topology. Let $Y_{\mathrm{v}}$ be the category of affinoid perfectoid spaces over $Y$ with the v-topology together with its sheaf of rings $\mathcal{O}_Y$. Let $\Lambda$ be a locally free $\ul{\zp}$-module over $Y$, let $V$ be a locally free sheaf of $\mathcal{O}_Y$-modules and
let $\gamma:\Lambda \otimes_{\ul{\zp}} \mathcal{O}_Y \to V$ be a surjective morphism. We define the character group $H_{\gamma}$ by the following fiber product
\begin{equation} \label{Eq:FiberProductIntro}
    \begin{tikzcd}
       H_{\gamma} \arrow{r} \arrow{d} \arrow[dr, phantom, "\scalebox{1.5}{$\lrcorner$}" , very near start, color=black] & \Hom_{Y}(\Lambda, \widehat{\mathbb{G}}_{m,\eta,Y}^{\lozenge}) \arrow{d}{\operatorname{Log}} \\
        \Hom_{\mathcal{O}_{Y}}(V, \mathbb{G}_{a,Y}^{\lozenge}) \arrow{r}{\gamma^{\ast}} & \Hom_{Y}(\Lambda, \mathbb{G}_{a,Y}^{\lozenge}).
    \end{tikzcd}
\end{equation}
For $(A,A^+)$ affinoid perfectoid over $\qp$ and $\spd(A,A^+) \to Y$ a map, the base change $H_{\gamma} \times_{Y} \spd(A,A^+)$ is given by $(H_{\gamma}^{\mathrm{rig}})^\lozenge \to \spd(A,A^+)$ for a unique relative $p$-divisible rigid analytic group $H_{\gamma}^{\mathrm{rig}} \to \spa(A,A^+)$ in the sense of Fargues \cite{FarguesII}. This latter property defines the notion of a \emph{p-divisible v-group} over $Y$, see Definition \ref{Def:PdivisibleGroup}. 

\subsubsection{} We consider the sheaf $\mathcal{O}_{H_{\gamma}/Y}$ of $\mathcal{O}_Y$-algebras given by the pushforward of $\mathcal{O}_{H_{\gamma}}$ along the structure map $H_{\gamma} \to Y$. We let $\mathcal{D}_{H_{\gamma}/Y}=\Hom_{\mathcal{O}_Y}(\mathcal{O}_{H_{\gamma}/Y}, \mathcal{O}_Y)$ be the dual. Using the group structure of $H_{\gamma}$, we equip both $\mathcal{O}_{H_{\gamma}/Y}$ and $\mathcal{D}_{H_{\gamma}/Y}$ with the structure of a (commutative) Hopf algebra for the solid tensor product of $\mathcal{O}_Y$-modules using K\"{u}nneth isomorphisms (see Proposition \ref{Prop:GeometricFrechet}). On the dual side, we construct a sheaf $\mathcal{O}^{\gamma\mathrm{-la}}_{\Lambda/Y}$ of $\gamma$-locally analytic functions on $\Lambda$, and we also consider its dual $\mathcal{D}^{\gamma\mathrm{-la}}_{\Lambda/Y}$. We also equip these with the structure of (commutative) solid Hopf $\mathcal{O}_Y$-algebras using K\"unneth isomorphisms (see Proposition \ref{Prop:GeometricLB}), and show that the universal character $\Lambda \times_Y H_\gamma \rightarrow \widehat{\mathbb{G}}_{m,\eta}^\lozenge$ lies in $\mathcal{O}_{\Lambda/Y}^{\gamma-\locan}(H_\gamma)\subseteq \mathcal{O}(\Lambda \times_Y H_{\gamma})$. 

\begin{mainThm} \label{Thm:IntroFourierVsheaves}
Integration against the universal character defines isomorphisms of solid Hopf $\mathcal{O}_Y$-algebras
    \begin{align}
     \mathcal{D}_{H_{\gamma}/Y} \xrightarrow{\sim} \mathcal{O}^{\gamma\mathrm{-la}}_{\Lambda/Y} \text{ and } \mathcal{D}^{\gamma\mathrm{-la}}_{\Lambda/Y} \xrightarrow{\sim} \mathcal{O}_{H_{\gamma}/Y}
    \end{align}
    which are functorial in $(\Lambda, V ,\gamma)$, compatible with base change in $Y$, and naturally dual. 
\end{mainThm}

We also prove an equivariance for actions of $\Sym^\bullet V$ and $\Sym^\bullet V^{\ast}$ in this setting that specializes to Proposition \ref{Prop:Equivariance}, see Proposition \ref{Prop:FTequivSymV} and Proposition \ref{Prop:FTSymVstarEquiv}. We will say more about the proof of Theorem \ref{Thm:IntroFourierVsheaves} in 
\S\ref{Sub:IntroFA}. 

\subsubsection{} If $Y=\spd(R,R^+)$ comes from a seminormal rigid space over a non-archimedean field $K/\mathbb{Q}_p$ and $V$ is a trivial vector bundle, we use Theorem~\ref{Thm:IntroFourierVsheaves} to deduce an extension of Theorem~\ref{thm.rational-fourier-theory} where the objects are now solid Hopf $R$-algebras, see Corollary~\ref{Cor:Main}. This works also when $(R,R^+)$ is a diamantine Huber pair (e.g. affinoid perfectoid), see \S\ref{subsub:Fierce}. When the character variety $H_\gamma$ is the base change of a $p$-divisible group $\mathcal{H}$ over $R^+$, we also establish in Theorem \ref{thm.integral-rational-compatibility} a compatibility between this result and an integral Fourier theory for $\mathcal{H}$ described in Theorem \ref{thm.integral-fourier-theory} (this integral Fourier theory is a consequence of Cartier duality for finite flat groups schemes; see the introduction to \S\ref{s.integral-fourier} for more details). 

\begin{Rem}
The integral compatibility has some interesting consequences related to subtle problems in the comparison of norms on completed tensor products and explicit cases of Bhatt--Scholze's surjectivity result for perfectoidizations of semiperfectoid rings \cite[Theorem 7.4]{BhattScholzePrisms} due to Fresnel--de Mathan \cite{FresneldeMathan.AlgebresL1padiques}; see Remark \ref{remark.non-injectivity} and Corollary \ref{cor.la-section}. 
\end{Rem}

We give an application of these results below. 

\subsection{The global Eisenstein measure and distribution} Fix $p>3$ and $N \ge 3$ coprime to $p$. Let $\mathfrak{Y}$ be the modular curve over $\spf \zp$ of full level $N$, let $\mathcal{E}$ be the universal elliptic curve over $\mathfrak{Y}$ and let $\omega_{\mathcal{E}}$ be the sheaf of relative differentials on $\mathcal{E}/\mathfrak{Y}$. 

\newcommand{\Eis}{\mathrm{Eis}}
For each $n$ coprime to $p$, we adapt a construction of Katz \cite{Katz.FormalGroupsAndPAdicInterpolation} by using our integral Fourier transform to turn the Weierstrass $\wp$-function into a weight two modular form $\Eis^{(n)}$ on $\mathfrak{Y}$ valued in measures on $T_p \mathcal{E}^\vee$, i.e., a section 
\[ \Eis^{(n)} \in H^0\left(\mathfrak{Y},\omega_\mathcal{E}^2 \otimes_{\mathcal{O}_\mathfrak{Y}} \Hom_{\mathcal{O}_\mathfrak{Y}}(\mathcal{O}_{T_p \mathcal{E}^\vee/\mathfrak{Y}}, \mathcal{O}_\mathfrak{Y})\right).\]
This measure interpolates the classical Eisenstein series defined on lattices $\Lambda \subset \mathbb{C}$ by 
\[ G_k(\Lambda)=\frac{(-1)^k (k-1)!}{2} \sum_{\lambda \in \Lambda \backslash \{0\}}\frac{1}{\lambda^k}. \]

\begin{mainThm}[see Theorem \ref{thm.integral-eisenstein}, Theorem \ref{thm.integral-eisensteinII}, and Example \ref{e.g.p-equiv-3-mod-4}]\label{mainThm.EisensteinMeasure} \hfill
\begin{enumerate}
    \item Writing $\mathrm{HT}: T_p\mathcal{E}^\vee \rightarrow \omega_{\mathcal{E}}$ for the Hodge--Tate map, for $k \geq 3$ we have
    \[ \int_{T_p \mathcal{E}^\vee} \mathrm{HT}^{k-2} d\Eis^{(n)}= 2(1-n^{k})G_{k} \in H^0(\mathfrak{Y}, \omega^{k}_\mathcal{E}). \]
    \item The Eisenstein measure $\mu^{(n)}$ of \cites{Katz.p-adic-L-via-moduli, Katz.TheEisensteinMeasureAndpAdicInterpolation} valued in Katz--Serre $p$-adic modular forms is recovered, up to the constant term and a degree shift by $1$ in the moments, by evaluating $\Eis^{(n)}$ at the canonical differential over the Katz--Igusa formal scheme.
    \item The one-variable $p$-adic $L$-functions interpolating Bernoulli-Hurwitz numbers of \cites{Katz.p-adic-L-via-moduli, Katz.TheEisensteinMeasureAndpAdicInterpolation, Katz.FormalGroupsAndPAdicInterpolation, Katz.DivisibilitiesCongruencesAndCartierDuality} are recovered by evaluating $\Eis^{(n)}$ at $(y^2 = 4x^3 - 4x, \frac{dx}{y})$. 
\end{enumerate}
\end{mainThm}

It is well known that the members of Katz's Eisenstein family overconverge after passing to the rigid generic fiber $\mathfrak{Y}_\eta$. On the other hand, our Eisenstein measure $\Eis^{(n)}$ induces, after base change to the rigid generic fiber, a global Hodge--Tate analytic Eisenstein distribution
\[ 
\Eis^{(n)}_\eta \in H^0\left(\mathfrak{Y}_{\eta}, (\omega^2 \otimes \mathcal{D}^{\mathrm{HT}-\locan}_{T_p \mathcal{E}_\eta^\vee/\mathfrak{Y}_\eta})\right).
\] 
Using this Eisenstein distribution, we recover the classical overconvergence (Example \ref{eg.overconvergent-eisenstein-series}). 

By the same method, we also find that the $p$-adic $L$-functions in the $p \equiv 3 \mod 4$ case of Theorem \ref{mainThm.EisensteinMeasure}-(3) overconverge to open neighborhoods in the supersingular locus. More precisely, we obtain a new family of quaternionic Eisenstein series parameterized by $\mathbb{Q}_{p^2}$-analytic characters of $\mathbb{Z}_{p^2}^\times$, and overconverge its members from the locus of formal CM by $\mathbb{Z}_{p^2}$, a profinite quaternionic double coset in the supersingular locus. In fact, we make these constructions with $\mathbb{Z}_{p^2}$ replaced by any order in a quadratic extension of $\mathbb{Q}_p$ (Example \ref{eg.quaternionic-eisenstein-series}). 

\subsection{Functional analysis and sheaf theory} \label{Sub:IntroFA} We now make some comments on the proof of Theorem \ref{Thm:IntroFourierVsheaves}. The overall strategy of proof is the same as discussed in \S\ref{Sub:Proof}, and a significant portion of this paper is devoted to building a sheaf-theoretic formalism to handle the reduction to affinoid perfectoids where the $\mathbb{Z}_p$-local system $\Lambda$ and $\mathcal{O}_Y$-module $V$ are both free (generalizing the descent from $C$ in the proof of Theorem \ref{thm.rational-fourier-theory}). 

\begin{Rem}
    When discussing $p$-adic functional analysis, we have elected to use the language of condensed mathematics over the more classical language of locally convex spaces. This has several benefits: First, v-sheaves automatically give rise to sheaves of condensed modules, since covers by profinite sets are naturally covers in the v-topology. This means that the topologies on our sheaves are already encoded in the geometry of the situation. Second, the category of condensed (or solid) modules over a ring satisfies many favorable properties which one does not have in the usual category of locally convex topological spaces (e.g., it is a Grothendieck abelian category with a symmetric monoidal structure satisfying Grothendieck’s axioms (AB4), (AB5), (AB3*) and (AB4*)). Many of our proofs exploit this extra structure.

Finally, when working over a non-archimedean field, the standard condensed objects (Banach/Fr\'echet spaces etc.) which appear in this article often coincide with their classical analogue as considered in \cite{SchneiderFA} or \cite{PGS} for example. See \cite{RJRC} for the comparison when the field is a finite extension of $\mbb{Q}_p$, and Appendix \ref{App:A} for part of the comparison over an arbitrary uniform $\qp$-Banach algebra. Thus, the condensed theory can be seen as a natural extension of the classical theory. We note, however, that it is not at all clear whether the classical theory generalizes well to general uniform Banach algebras (although \cite{GaisinJacinto} develops some theory over affinoid algebras), whereas the condensed perspective is well-suited to handle this relative setting. 
\end{Rem}

\subsubsection{} We now highlight some important technical results in the paper. For example, we prove that the sheaves $\mathcal{O}_{H/Y}, \mathcal{D}_{H/Y}, \mathcal{O}^{\gamma\mathrm{-la}}_{\Lambda/Y}, \mathcal{D}^{\gamma\mathrm{-la}}_{\Lambda/Y}$
are reflexive and that they satisfy a K\"unneth formula with respect to the solid tensor product of sheaves of $\mathcal{O}_Y$-modules, see Proposition \ref{Prop:GeometricFrechet} and \ref{Prop:GeometricLB}. Proving these two propositions uses essentially all of the theory developed in Sections \ref{Sec:Condensed} and \ref{Sec:VSheaves}. A crucial ingredient here is the flatness for the solid tensor product of (strongly countably) Fr\'echet modules for arbitrary uniform $(R,R^+)$ over $(\qp,\zp)$, see Proposition \ref{Prop:FrechetFlat}. This proposition relies on input from the theory of solid functional analysis over $\ul{R}$ developed by Bosco in \cite[Appendix A]{Bosco} (following Clausen-Scholze).

Along the way, we also show that the internal hom on the v-site is related to the internal hom in condensed modules, see \S\ref{sss.solid-hopf-and-duality-discussion-fiercely}. This is what allows us to deduce Theorem \ref{thm.rational-fourier-theory} from Theorem \ref{Thm:IntroFourierVsheaves}. To prove this, we establish a ``quasi-coherence" property for the sheaves $\mathcal{O}_{H/Y}, \mathcal{D}_{H/Y}, \mathcal{O}^{\gamma\mathrm{-la}}_{\Lambda/Y}, \mathcal{D}^{\gamma\mathrm{-la}}_{\Lambda/Y}$ over fiercely v-complete bases (see Definition \ref{Def:Fierce}) when $V$, the codomain of $\gamma$, is free. See Proposition \ref{Prop:GeometricFrechet}-(1) and \ref{Prop:GeometricLB}-(1) for precise statements. Proving this quasicoherence for (locally strongly countably) Fr\'echet sheaves such as $\mathcal{O}_{H/Y}$ and $\mathcal{D}^{\gamma\mathrm{-la}}_{\Lambda/Y}$ makes crucial use of the flatness mentioned above to commute inverse limits and solid tensor products. 

\subsection{Leitfaden} In \S \ref{Sec:Condensed}, we develop the results in $p$-adic functional analysis needed for this article in the language of condensed mathematics. After this, we discuss the theory of sheaves on small v-stacks in \S \ref{Sec:VSheaves}, and explain how one can naturally view these objects as sheaves of condensed modules. In particular, we extend the theory developed in \S \ref{Sec:Condensed} to the setting of v-sheaves. With these preliminary results at hand, we then introduce and study (families of) $p$-divisible v-groups over small v-stacks in \S \ref{Sec:FamiliesPDiv}, and sheaves of functions associated with such v-groups in \S \ref{Sec:SheavesOfLaFunctions} (with the key results being Propositions \ref{Prop:GeometricFrechet} and \ref{Prop:GeometricLB}). In \S \ref{Sec:Fourier}, we construct the Fourier transforms, establish their equivariance properties, and show that they are isomorphisms, see Theorem \ref{Thm:Main} and Corollary \ref{Cor:Main}. We then discuss our main theorem over a field $K$ in classical language in \S\ref{Sub:Field}, and compare with \cite{SchneiderTeitelbaumFourier}. In \S \ref{s.integral-fourier}, we construct integral Fourier transforms for $p$-divisible groups in the sense of Tate over $p$-adically complete rings, and prove a compatibility with the rigid analytic Fourier transforms of Corollary \ref{Cor:Main}. In \S \ref{Sec:Eisenstein}, we apply our results to construct the global Eisenstein measure and distribution, then study its properties. In the appendix, we compare a small part of \S \ref{Sec:Condensed} with the theory of topological modules over uniform Banach $\qp$-algebras, and conclude with a discussion of the Amice transform.

\subsection{Acknowledgments} We would like to thank Guido Bosco, Ellen Eischen, Lukas Gerth, Arthur-César Le Bras, Vaughan McDonald, Gal Porat, Joaqu\'{i}n Rodrigues Jacinto, Juan Esteban Rodr\'{i}guez Camargo, Thomas Rot, and Arun Soor for helpful discussions. 

This research was supported through the program ``Oberwolfach Research Fellows" by the Mathematisches Forschungsinstitut Oberwolfach in 2025. AG was (partly) funded by UK Research and Innovation grant MR/V021931/1. PvH is partly funded by the Dutch Research Council (NWO) under the grant VI.Veni.232.127. SH was supported by the National Science Foundation through grant DMS-2201112 and as a member at the Institute for Advanced Study during the academic year 2023-24 by the Friends of the Institute for Advanced Study Membership. For the purpose of Open Access, the authors have applied a CC BY public copyright licence to any Author Accepted Manuscript (AAM) version arising from this submission.

\section{Condensed functional analysis} \label{Sec:Condensed} In this section we will develop the theory of solid condensed functional analysis over uniform Banach $\qp$-algebras, inspired by and mostly following \cite{CondensedNotes}, \cite{RJRC}, \cite[Appendix A]{Bosco}. In \S \ref{Sub:CondensedPrelim}, we discuss the basic theory of solid modules over $p$-adically complete rings and the relation to topological modules, and in \S \ref{Sub:CondensedStandard}, we introduce certain standard Banach and Smith modules, and compute their duals. In \S \ref{Sub:TraceClassCompact}, we discuss the theory of trace-class and compact type morphisms. In \S \ref{Sub:Frechet}, we introduce (strongly countable) Fr\'echet modules, prove that these are always flat (Proposition \ref{Prop:FrechetFlat}), and use this property to describe their interactions with tensor products and duals (Proposition \ref{Prop:TensorInverseLimit} and Corollary \ref{Cor:DualTensorCommute}).

\subsection{Preliminaries} \label{Sub:CondensedPrelim} Fix an uncountable strong limit cardinal $\kappa$ and let $\prof_{\kappa}$ denote the category of profinite sets of cardinality at most $\kappa$, see \cite[Remark 1.3]{CondensedNotes}. Recall that the category of $\kappa$-small condensed sets (abelian groups, rings,...) is defined to be the category of sheaves of sets (abelian groups, rings,...) on $\prof_{\kappa}$ with respect to the topology of finite jointly surjective families of maps, see \cite[Definition 2.1]{CondensedNotes}. The category of condensed sets (abelian groups, rings,...) is defined to be the filtered colimit over all $\kappa$ as above of the category of $\kappa$-small condensed sets (abelian groups, rings,...), see \cite[Definition 2.11]{CondensedNotes}. 

The category of condensed abelian groups is a Grothendieck abelian category satisfying Grothendieck’s axioms (AB4), (AB5), (AB3*) and (AB4*), see \cite[Theorem 1.10]{CondensedNotes}, which moreover admits  all (small) limits and colimits.

\subsubsection{} If $R$ is a $\kappa$-small condensed ring, then a $\kappa$-small $R$-module is a sheaf of $R$-modules on $\prof_{\kappa}$. Taking the colimit over $\kappa$, we obtain the abelian tensor category $\Mod_{R}^{\mathrm{cond}}$, which is simply the category of $R$-modules in condensed abelian groups. This category is an abelian category admitting all (small) limits and colimits, which moreover satisfies Grothendieck’s axioms (AB4), (AB5), (AB3*) and (AB4*). Indeed, this follows from the case of condensed abelian groups discussed above (since limits and colimits can be computed on the underlying condensed abelian group, and since exactness can be tested on the underlying condensed abelian group). If $R=\ul{\mathbb{Z}}$, then $\Mod_{\ul{\mathbb{Z}}}^{\mathrm{cond}}$ is simply the category of condensed abelian groups, and we will denote it by $\Mod_{\mathbb{Z}}^{\mathrm{cond}}$.

\subsubsection{} Recall from \cite[Definition 5.1]{CondensedNotes} the notion of a solid abelian group. The full subcategory $\Mod_{\mathbb{Z}}^{\operatorname{solid}} \subset \Mod_{\mathbb{Z}}^{\mathrm{cond}}$ consisting of solid abelian groups is closed under (small) colimits and limits, and extensions, see \cite[Theorem 5.8]{CondensedNotes}. Therefore, it also satisfies Grothendieck’s axioms (AB4), (AB5), (AB3*) and (AB4*). The category $\Mod_{\mathbb{Z}}^{\operatorname{solid}}$ has a symmetric monoidal structure given by the solid tensor product $\otimes^{\blacksquare}_{\ul{\mathbb{Z}}}$. \smallskip 

If $R$ is a condensed commutative ring whose underlying abelian group is solid, then we can consider the category $\Mod_{R}^{\operatorname{solid}} \subset \Mod_{R}^{\cond} $ of condensed $R$-modules whose underlying abelian group is solid. This subcategory is closed under (small) colimits and limits, and extensions, see \cite[Theorem 5.8]{CondensedNotes}. Therefore, it also satisfies Grothendieck’s axioms (AB4), (AB5), (AB3*) and (AB4*). The category $\Mod_{R}^{\operatorname{solid}}$ has a symmetric monoidal structure given by the solid tensor product $\otimes^{\blacksquare}_{R}$.

\subsubsection{} \label{SubSub:FullyFaithfulOnCGspaces}

There is a functor from topological spaces (abelian groups, rings,...) to condensed sets (abelian groups, rings,...) sending $V$ to the condensed set $\ul{V}$ whose value on a $\kappa$-small profinite set $S$ is the set $C^0(S,V)$ of continuous functions from $S$ to $V$, see \cite[Example 1.5]{CondensedNotes}. The functor $V \mapsto \ul{V}$ is fully faithful when restricted to $V$ whose underlying topological space is compactly generated (e.g., metrizable), see \cite[Proposition 1.7]{CondensedNotes}.

The functor $V \mapsto \ul{V}$ has a left adjoint given by $M \mapsto M(\ast)_{\mathrm{top}}$, which is a certain way to topologize the evaluation of the condensed set $M$ on the one-point profinite set $\ast$. By construction, the topological space $M(\ast)_{\mathrm{top}}$ is compactly generated.

\begin{Lem} \label{Lem:DiscreteDirectLimit}
    For $X=\varinjlim_i X_i$ a countably indexed filtered direct limit of discrete topological spaces, the natural map
    \begin{align}
        \varinjlim_i \ul{X_i} \to \ul{X}
    \end{align}
    is an isomorphism. 
\end{Lem}
\begin{proof}
It suffices to show that continuous maps $S \to X$, with $S$ a profinite set, factor through $X_i$ for some $i$. Since $S$ is profinite and $X$ is discrete, the image of a continuous map is compact hence finite; the lemma follows. 
\end{proof}

\subsubsection{} If $R$ is a topological ring, then we will write $\operatorname{Mod}_{R}^{\mathrm{cond}}=\operatorname{Mod}_{\ul{R}}^{\mathrm{cond}}$ for the category of condensed $\ul{R}$-modules. If $R$ is a topological ring and $M$ is a topological $R$-module, then $\ul{M}$ is naturally an $\ul{R}$-module. 
\begin{Lem} \label{Lem:Fullyfaithful}
Let $R$ be a topological ring. If $M,N$ are compactly generated\footnote{By this we always mean compactly generated as a topological space.}topological $R$-modules, then the natural map
    \begin{align}
        \hom_{C^0,R}(M,N) \to \hom_{\ul{R}}(\ul{M}, \ul{N})
    \end{align}
    is a bijection.
\end{Lem}
\begin{proof}
As mentioned in \S \ref{SubSub:FullyFaithfulOnCGspaces}, the fully faithfulness for compactly generated topological spaces is \cite[Proposition 1.7]{CondensedNotes}. The condition for a morphism $f:M \to N$ of topological spaces to be a homomorphism of $R$-modules comes down to the commutativity of the diagrams
\begin{equation}
    \begin{tikzcd}
        M \times M \arrow{r}{f \times f} \arrow{d} & N \times N \arrow{d} \\
        M \arrow{r}{f} & N
    \end{tikzcd}
    \begin{tikzcd}
        M \times A \arrow{r}{f \times \operatorname{Id}_{A}} \arrow{d} & N \times A \arrow{d} \\
        M \arrow{r}{f} & N.
    \end{tikzcd}
\end{equation}
Since the functor from topological spaces to condensed sets commutes with products and is faithful, the commutativity of these diagrams can be checked after passing to condensed sets. Thus a continuous map $M \to N$ is a homomorphism of $R$-modules if and only if the induced map $\ul{M} \to \ul{N}$ is a homomorphism of $\ul{R}$-modules.
\end{proof} 
\subsubsection{} For a condensed ring $R$ and condensed $R$-modules $M,N$, we will write $\Hom_{R}(M,N)$ for the internal hom in the category of condensed $R$-modules. If $R$ is a topological ring and $M$ a topological $R$-module, then for a profinite set $S$ we will write $M_{S}=C^0(S,M)$ for the set of continuous functions $f:S \to M$ equipped with the compact open topology. If $R$ and $M$ are compactly generated Hausdorff, then this is a topological $R$-module, because $C^0(S,M)$ is the internal hom in the category of Hausdorff compactly generated topological spaces, see \cite[Theorem 5.14]{Topology}. 
\begin{Lem} \label{Lem:PreEvaluationCondensed}
    Let $R$ be a topological ring and $M,N$ be topological $R$-modules. If $R,N,M$ are compactly generated and Hausdorff, then for a profinite set $S$ there is a natural bijection
    \begin{align}
        \Hom_{\ul{R}}(\ul{M},\ul{N})(S) \to \hom_{\Cont, R}(M, N_{S}). 
    \end{align}
\end{Lem}
\begin{proof}
Write $\pi^{-1}$ for the functor which takes sheaves on profinite sets and restricts them to the category of profinite sets over $S$. It has a right adjoint $\pi_{\ast}$ which takes a sheaf $F$ on profinite sets over $S$ and sends it to the sheaf $T \mapsto F(S \times T)$. Using the definition and the adjunction we find that
\begin{align}
    \Hom_{\ul{R}}(\ul{M}, \ul{N})(S)&:= \hom_{\pi^{-1} \ul{R}}(\pi^{-1} \ul{M}, \pi^{-1} \ul{N}) \\
    &=\hom_{\ul{R}}(\ul{M}, \pi_{\ast}\pi^{-1} \ul{N}).
\end{align}
The sheaf $\pi_{\ast}\pi^{-1} \ul{N}$ is given by 
\begin{align}
    T \mapsto C^0(T \times S, N)
\end{align}
which we claim is naturally isomorphic to $T \mapsto C^0(T, C^0(S,N))$. This claim follows from the fact that $C^0(S,N)$ is the internal hom in the category of weakly Hausdorff compactly generated topological spaces, see \cite[Theorem 5.14]{Topology} (here we are using the fact that $N$ is compactly generated Hausdorff); in particular $C^0(S,N)$ is itself compactly generated. Thus we have identified $\pi_{\ast}\pi^{-1} \ul{N} = \ul{C^0(S,N)}$ and by Lemma \ref{Lem:Fullyfaithful} we see that
\begin{align}
    \hom_{\ul{R}}(\ul{M}, \ul{C^0(S,N)})&= \hom_{C^0,R}(M, C^0(S,N)).
\end{align}
\end{proof}
\begin{Lem} \label{Lem:CompactOpen}
Let $N$ be a $p$-adically complete and separated topological abelian group with the $p$-adic topology and let $S$ be a profinite set. Then the $p$-adic topology on $C^0(S,N)$ agrees with the compact-open topology.
\end{Lem}
\begin{proof}
Since $S$ is compact Hausdorff, the compact-open topology on $C^0(S,N)$ is compactly generated. We now use the fact that $C^0(S,N)$, with the compact-open topology, is the internal hom in the category of compactly generated and weakly Hausdorff topological spaces, see \cite[Theorem 5.14]{Topology}. It follows from this that
\begin{align}
    C^0(S,N) &= C^0(S,\varprojlim_n N/p^n) \\
    &=\varprojlim_n C^0(S,N/p^n),
\end{align}
and thus it suffices to prove that the compact open topology on $C^0(S,N/p^n)$ is the discrete topology. Let $f:S \to N/p^n$ be a continuous map; we will show that $\{f\} \subset C^0(S,N/p^n)$ is open. Note that $f$ takes finitely many values $r_{1}, \cdots, r_{k}$ whose preimages give an open cover $K_{1}, \cdots, K_{k}$ of $S$. Per definition of the compact-open topology, the subsets $U_{i} \subset C^0(S,N/p^n)$ of continuous functions $g$ taking $K_i$ to $r_i$ are open. Now $\{f\}=\cap_{i=1}^{k} U_i$, so that $\{f\}$ is open.
\end{proof}

\subsubsection{} Let $R^+$ be a $p$-adically complete and separated ring, and let $M,N$ be $p$-adically complete and separated $R^+$-modules whose topology is the $p$-adic topology. 
\begin{Prop} \label{Prop:TensorProductsSolidVSComplete} 
If $\ul{M}$ is flat for the solid tensor product over $\ul{R^+}$, and $N$ and $R^+$ are $p$-torsion free, then the condensed $\ul{R^{+}}$-module $\ul{M \widehat{\otimes}_{R^+} N}$ is isomorphic to the solid tensor product $\ul{M} \otimes^{\blacksquare}_{\ul{R^+}} \ul{N}$.
\end{Prop}
\begin{proof}
For a profinite set $S=\varprojlim_i S_i$ we compute 
    \begin{align}
        C^0(S,M \widehat{\otimes}_{R^{+}} N) &= C^0(S, \varprojlim_n (M/p^n \otimes_{R^{+}/p^n} N/p^n)) \\
        &=\varprojlim_n \left( C^0(S, M/p^n \otimes_{R^{+}/p^n} N/p^n) \right) \\
        &=\varprojlim_n \left( \varinjlim_i C^0(S_i, M/p^n \otimes_{R^{+}/p^n} N/p^n) \right).
\end{align}
where the last equality follows from $M/p^n \otimes_{R^{+}/p^n} N/p^n$ having the discrete topology. Now we use the fact that tensor products commute with finite products and with colimits to see
\begin{align}
        \varprojlim_n \left( \varinjlim_i C^0(S_i, M/p^n \otimes_{R^{+}/p^n} N/p^n) \right) &=\varprojlim_n \left( \varinjlim_i C^0(S_i, M/p^n) \otimes_{C^0(S_i, R^{+}/p^n)} C^0(S_i, N/p^n) \right) \\
        &=\varprojlim_n \left( C^0(S, M/p^n) \otimes_{C^0(S, R^{+}/p^n)} C^0(S, N/p^n) \right) \\
        &=\varprojlim_n (\ul{M/p^n} \otimes_{\ul{R^{+}/p^n}} \ul{N/p^n})(S).
    \end{align}
So we see that the condensed $\ul{R^{+}}$-module $\ul{M \widehat{\otimes}_{R^{+}} N}$ is given by the $p$-adic completion of the (non-solid) condensed tensor product, i.e.,\[ \ul{M \widehat{\otimes}_{R^{+}} N} =\varprojlim_n (\ul{M/p^n} \otimes_{\ul{R^{+}/p^n}} \ul{N/p^n}) = \varprojlim_n \left((\ul{M} \otimes_{\ul{R^{+}}} \ul{N})/p^n\right) .\]
Now for the solid tensor product we use the flatness to write
\begin{align}
    \ul{M} \otimes^{\blacksquare}_{\ul{R^+}} \ul{N} = \ul{M} \otimes^{\blacksquare, \mathbb{L}}_{\ul{R^+}} \ul{N}.
\end{align}
It then follows from \cite[Lemma A.3]{BoscoII}, that the right hand side is derived $p$-adically complete. 
We can thus compute
\begin{align}
    \ul{M} \otimes^{\blacksquare, \mathbb{L}}_{\ul{R^+}} \ul{N} &= R \varprojlim_n \left( (\ul{M} \otimes^{\blacksquare, \mathbb{L}}_{\ul{R^+}} \ul{N}) \otimes_{\ul{R^+}}^{\blacksquare, \mathbb{L}} \ul{R^+/p^n} \right) \\
    &=R \varprojlim_n \left( \ul{M} \otimes^{\blacksquare, \mathbb{L}}_{\ul{R^+}}  \ul{R^+/p^n} \otimes_{\ul{R^+}}^{\mathbb{L},\blacksquare} \ul{N} \right) \\
    &=R \varprojlim_n \left( \ul{M} \otimes^{\blacksquare}_{\ul{R^+}}  \ul{R^+/p^n} \otimes_{\ul{R^+}}^{\mathbb{L},\blacksquare} \ul{N} \right) \\
    &=R \varprojlim_n \left( \ul{M/p^n} \otimes_{\ul{R^+}}^{\mathbb{L},\blacksquare} \ul{N} \right) \\
    &=R \varprojlim_n \left( \ul{M/p^n} \otimes_{\ul{R^+/p^n}}^{\mathbb{L},\blacksquare} \ul{R^+/p^n} \otimes_{\ul{R^+}}^{\mathbb{L},\blacksquare} \ul{N} \right).
\end{align}
where the third equality uses the flatness of $\ul{M}$. We now claim that $\ul{R^+/p^n} \otimes_{\ul{R^+}}^{\mathbb{L},\blacksquare} \ul{N}$ is isomorphic to $\ul{N/p^n}$ since $N$ is $p$-torsion free: For this we first note that $\ul{R^+/p^n} \otimes_{\ul{R^+}}^{\mathbb{L},\blacksquare} \ul{N} \simeq \ul{\mathbb{Z}/p^n \mathbb{Z}}\otimes_{\ul{\mathbb{Z}}}^{\mathbb{L},\blacksquare} \ul{N}$ (this uses the assumption that $R^+$ is $p$-torsion free). Now $\ul{N}(S)$ is torsion free for all $S$ because $N$ is torsion free, hence $N$ is flat for the (non-solid) condensed tensor product, see \cite[Lemma A.5]{BoscoII}. We see that
\begin{align}
\ul{\mathbb{Z}/p^n \mathbb{Z}}\otimes_{\ul{\mathbb{Z}}}^{\mathbb{L}} \ul{N} &\simeq \ul{\mathbb{Z}/p^n \mathbb{Z}}\otimes_{\ul{\mathbb{Z}}} \ul{N} \\
&=\ul{N/p^n}.
\end{align}
This is already solid, and so its derived solidification equals itself (see \cite[Theorem 6.2 and the paragraph following it]{CondensedNotes}), and so we get $\ul{R^+/p^n} \otimes_{\ul{R^+}}^{\mathbb{L},\blacksquare} \ul{N} \simeq \ul{N/p^n}$. We now continue with our computation: 
\begin{align}
    &=R \varprojlim_n \left( \ul{M/p^n} \otimes_{\ul{R^+/p^n}}^{\mathbb{L},\blacksquare} \ul{N/p^n} \right) \\
    &=R \varprojlim_n \left( \ul{M/p^n} \otimes_{\ul{R^+/p^n}}^{\blacksquare} \ul{N/p^n} \right) \\
    &=R\varprojlim_n \left( \ul{M/p^n \otimes_{R^+/p^n} N/p^n}\right) \\
    &=\varprojlim_n \left( \ul{M/p^n \otimes_{R^+/p^n} N/p^n}\right),
\end{align}
where, in the second equality we have used that $\ul{M/p^n}$ is a flat $\ul{R^+/p^n}$-module, and $p$-adic completeness of $M,N$ has been used in the last step. This finishes the proof of the lemma.
\end{proof}

\begin{Rem} \label{Rem:UncompletedTensor}
    The assumption that $M$ and $N$ have the $p$-adic topology is necessary. Indeed, let $R=\zp, N=\zp[[T]]$ with the $(p,T)$-adic topology and $M=\zp\langle X \rangle$ (the $p$-adic completion of $\zp[X]$). In this case $M \widehat{\otimes}_{\zp} N$ is the $(T,p)$-adic completion of $\zp[T,X]$ while $\ul{M} \otimes_{\ul{\zp}}^{\blacksquare} \ul{N}$ is the (condensed module associated with the) $p$-adic completion of $\zp[[T]][X]$ (which is not $T$-adically complete). 
\end{Rem}

\subsection{Standard modules} \label{Sub:CondensedStandard} Let $(R,R^+)$ be a uniform Huber pair over $(\qp,\zp)$, so that $R^+$ is a $p$-adically complete and separated ring. We define for an index set $I$ the solid $\ul{R^{+}}$ (resp. $\ul{R}$)-modules
\begin{align} \label{eq:CondensedStandard}
    \hat{\oplus}_I \ul{R^+}:=\varprojlim_n \oplus_I \ul{R^{+}/p^n} \\
    \hat{\oplus}_I \ul{R}:=(\hat{\oplus}_I \ul{R^{+}})[\tfrac{1}{p}].
\end{align}
As suggested by the notation, the module $\hat{\oplus}_I \ul{R}$ does not depend on the choice of $R^+$. 

The module $\hat{\oplus}_I \ul{R^+}$ is solid because it is an inverse limit of discrete $\ul{R^+}$-modules, which are automatically solid. Inverting $p$ is a filtered colimit, and therefore $\hat{\oplus}_I \ul{R}$ is also solid. We will refer to solid $\ul{R}$-modules isomorphic to $\hat{\oplus}_I \ul{R}$ as \emph{orthonormalizable Banach} $\ul{R}$-modules. For an index set $I$ we also consider the solid $\ul{R^{+}}$ (resp. $\ul{R}$)-modules
\begin{align}
    \textstyle\prod_I \ul{R^+}, \; \left(\textstyle\prod_I \ul{R^+}\right)[\tfrac{1}{p}],
\end{align}
where we note as before that the second module does not depend on the choice of $R^+$. We will refer to solid $\ul{R}$-modules isomorphic to $(\textstyle\prod_I \ul{R^+})[\tfrac{1}{p}]$ as \emph{free Smith} $\ul{R}$-modules. When $I$ is countable, then we will call these modules \emph{countably orthonormalizable Banach} and \emph{countably free Smith}. 

\subsubsection{} \label{Subsub:classicalStdModules} For an index set $I$, we can also form the topological $R^+$-modules 
\begin{align} \label{Eq:TopologicalStandard}
    \hat{\oplus}_I R^+&:=\varprojlim_n \oplus_I R^+/p^n, \qquad \textstyle\prod_I R^{+},
\end{align}
with the $p$-adic and product topologies respectively. Since $M \mapsto \ul{M}$ commutes with limits and takes direct sums of discrete modules to direct sums, see Lemma \ref{Lem:DiscreteDirectLimit}, it follows that the natural maps 
\begin{align}
    \hat{\oplus}_I \ul{R^{+}} &\to \ul{\hat{\oplus}_I R^+}  \\
    \ul{\textstyle\prod_I R^{+}} &\to \textstyle\prod_I \ul{R^+}
\end{align}
are isomorphisms. We also have the following lemma (which we do not generally expect to hold for the product when $I$ is not countable).
\begin{Lem} \label{Lem:CompactlyGenerated}
The topological spaces $\hat{\oplus}_I R^+$ and $\hat{\oplus}_I R$ are compactly generated and Hausdorff. If $I$ is countable, then the topological spaces $\textstyle\prod_I R^{+}$ and $(\textstyle\prod_I R^{+})[\tfrac{1}{p}]$ are compactly generated and Hausdorff.
\end{Lem}
\begin{proof}
The topological ring $R^+$ has the $p$-adic topology, which is metrisable (since $R^+$ is complete and separated for the $p$-adic topology) hence compactly generated and Hausdorff. The topological abelian group $\hat{\oplus}_I R^+$ has the $p$-adic topology, which is metrisable and therefore compactly generated; it is Hausdorff because $\hat{\oplus}_I R^+$ is $p$-adically separated. It follows that $\hat{\oplus}_I R$ is compactly generated and Hausdorff as it is a countably direct limit of compactly generated Hausdorff topological spaces under open and closed transition maps. The module $\textstyle\prod_I R^+$ is a countable product of metrisable (Hausdorff) spaces which is (Hausdorff) metrisable, see \cite[Satz 4.4.12 on p. 117]{HerrlichTopology}, and thus compactly generated Hausdorff. It follows similarly to before that $(\textstyle\prod_I R^{+})[\tfrac{1}{p}]$ is compactly generated and Hausdorff.
\end{proof}

\subsubsection{} Recall the completed tensor $\widehat{\otimes}$ product of linearly topologized modules over linearly topologized rings (in the sense of \cite[Tag 07E7]{stacks-project}), see \cite[Tag 0AMU]{stacks-project}. 

\begin{Lem} \label{Lem:StandardTensorTopological}
If $R^+ \to B^+$ is a morphism of $p$-complete rings, then the natural maps
    \begin{align}
        (\hat{\oplus}_I R^+) \widehat{\otimes}_{R^+} B^+ &\to  \hat{\oplus}_{I} B^+ \\
        (\hat{\oplus}_I R^+) \widehat{\otimes}_{R^+} (\hat{\oplus}_{J} R^+ )&\to  \hat{\oplus}_{I \times J} R^+
    \end{align}
    are isomorphisms.
\end{Lem}
\begin{proof}
This follows directly from the definition of the completed tensor product, and the fact that (usual) tensor products commute with direct sums. 
\end{proof}

\subsubsection{} We now study properties of these standard modules.

\begin{Lem} \label{Lem:BanachFlat}
If $(R,R^+)$ is a uniform Huber pair, then $\hat{\oplus}_I \ul{R^+}$ is flat for the solid tensor product.
\end{Lem}
\begin{proof}
First of all, $\textstyle\prod_I \ul{\mathbb{Z}}$ is flat for the solid tensor product over $\ul{\mathbb{Z}}$ because it is torsion free, see \cite[Lemma A.5]{BoscoII}. By \cite[Lemma A.19]{Bosco} we have $\textstyle\prod_I \ul{\mathbb{Z}} \otimes_{\ul{\mathbb{Z}}} \ul{\zp} = \textstyle\prod_I \ul{\mbb{Z}_p}$ which is therefore flat for the solid tensor product over $\ul{\zp}$. By \cite[Lemma A.4]{BoscoII}, we can write $\hat{\oplus}_I \ul{\zp}$ as a filtered colimit of objects isomorphic to $\textstyle\prod_I \ul{\zp}$ and it is thus flat for the solid tensor product over $\ul{\zp}$. It follows from Proposition \ref{Prop:TensorProductsSolidVSComplete} and Lemma \ref{Lem:StandardTensorTopological} that $\hat{\oplus}_I \ul{\zp} \otimes_{\ul{\zp}}^\blacksquare \ul{R^+} = \hat{\oplus}_I \ul{R^+}$, which is hence flat for the solid tensor product over $\ul{R^+}$.
\end{proof}

\begin{Prop} \label{Prop:StandardBCproperty}
Let $(R, R^+) \to (B, B^+)$ be a morphism of uniform Huber pairs over $(\mbb{Q}_p, \mbb{Z}_p)$. For any index set $I$, the natural maps
    \[
    \hat{\oplus}_I \ul{R}^+ \otimes^{\blacksquare}_{\ul{R}^+} \ul{B}^+ \xrightarrow{\sim} \hat{\oplus}_I \ul{B}^+ \quad \textrm{and} \quad \hat{\oplus}_I \ul{R} \otimes^{\blacksquare}_{\ul{R}} \ul{B} \xrightarrow{\sim} \hat{\oplus}_I \ul{B} 
    \]
    are isomorphisms. Moreover, for any two index sets $I$ and $J$, the natural maps
    \[ 
         \hat{\oplus}_I \ul{R^{+}} \otimes^{\blacksquare}_{\ul{R^+}}  \hat{\oplus}_J \ul{R^{+}} \to  \hat{\oplus}_{I \times J} \ul{R^+}\quad \textrm{and} \quad
         (\hat{\oplus}_I \ul{R}) \otimes^{\blacksquare}_{\ul{R}}  (\hat{\oplus}_J \ul{R}) \to  \hat{\oplus}_{I \times J} \ul{R} \]
are isomorphisms.
\end{Prop}
\begin{proof}
Because the solid tensor product commutes with inverting $p$, it suffices to prove the first and the third isomorphisms. By Proposition \ref{Prop:TensorProductsSolidVSComplete} and Lemma \ref{Lem:BanachFlat}, these reduce to a computation for the completed tensor product, where it is Lemma \ref{Lem:StandardTensorTopological}.
\end{proof}

\subsubsection{} Recall that we write $\Hom_{\ul{R^{+}}}$ and $\Hom_{\ul{R}}$ for the internal hom in condensed $\ul{R^{+}}$-modules or condensed $\ul{R}$-modules.

\begin{Prop} \label{Prop:DualityCondensed}
The natural maps
\begin{align}
\Hom_{\ul{R^{+}}}\left(\hat{\oplus}_{I}\ul{R^{+}}, \ul{R^{+}}\right) &\to  \textstyle\prod_{i \in I} \ul{R}^{+} \\
\Hom_{\ul{R}}\left(\hat{\oplus}_{I}\ul{R}, \ul{R}\right) &\to  \left(\textstyle\prod_{I} \ul{R}^{+}\right)[\tfrac{1}{p}]
\end{align}
are isomorphisms. If $I$ is countable, then the natural maps
\begin{align}
    \hat{\oplus}_{I}\ul{R^{+}} &\to \Hom_{\ul{R^{+}}}\left(\textstyle\prod_I \ul{R^{+}}, \ul{R^{+}}\right) \\
    \hat{\oplus}_{I}\ul{R} &\to \Hom_{\ul{R}}\left(\left(\textstyle\prod_I \ul{R^{+}}\right)[\tfrac{1}{p}], \ul{R}\right)
\end{align}
are isomorphisms.
\end{Prop}
\begin{proof}
For the first claim, we follow the proof of \cite[Lemma 3.10]{RJRC}: We note that
\begin{align}
     \Hom_{\ul{R^{+}}}\left(\hat{\oplus}_{I}\ul{R^{+}}, \ul{R^{+}}\right) &= \varprojlim_{n}  \Hom_{\ul{R^{+}}/p^n}\left(\oplus_{I}\ul{R^{+}}/p^n, \ul{R^{+}}/p^n\right) \\
     &=\varprojlim_{n}  \textstyle\prod_{I} \ul{R^{+}}/p^n \\
     &= \textstyle\prod_{I} \ul{R^{+}}.
\end{align}
The second claim follows from the first claim once we establish that the natural map
\begin{align}
    \Hom_{\ul{R^{+}}}\left(\hat{\oplus}_{I}\ul{R^{+}}, \ul{R^{+}}\right)[\tfrac{1}{p}] \to \Hom_{\ul{R}}\left(\hat{\oplus}_{I}\ul{R}, \ul{R}\right).
\end{align}
is an isomorphism. Using Lemma \ref{Lem:PreEvaluationCondensed} (whose assumptions are satisfied by Lemma \ref{Lem:CompactlyGenerated}), we can identify this map (on $S$ points for a profinite set $S$) with the natural map
\begin{align}
    \hom_{\Cont, R^+}(\hat{\oplus}_{I}R^+, R^+_S)[\tfrac{1}{p}] \to \hom_{\Cont, R}(\hat{\oplus}_{I}R, R_S). 
\end{align}
This map is a bijection since any continuous map of $\qp$-Banach spaces is bounded. For the third claim, we consider 
\begin{align}
    \Hom_{\ul{R^{+}}}\left(\textstyle\prod_I \ul{R^{+}}, \ul{R^{+}}\right) &= \varprojlim_n \Hom_{\ul{R^{+}}/p^n}\left(\textstyle\prod_I \ul{R^{+}}/p^n, \ul{R^{+}}/p^n\right).
\end{align}
This receives a natural map from $\hat{\oplus}_I \ul{R^+}$,
\begin{multline}  \label{eq:DirectSumToProduct}
    \hat{\oplus}_I \ul{R^+}=\varprojlim_n \oplus_I \Hom_{\ul{R^{+}}/p^n}\left(\ul{R^{+}}/p^n, \ul{R^{+}}/p^n\right) \to \varprojlim_n \Hom_{\ul{R^{+}}/p^n}\left(\textstyle\prod_I \ul{R^{+}}/p^n, \ul{R^{+}}/p^n\right),
\end{multline}
by taking the natural map termwise in the limit. We claim that this map is an isomorphism. We can check this on $S$-points for profinite sets $S$, and by Lemma \ref{Lem:PreEvaluationCondensed} applied to $R^+/p^n$ (whose assumptions hold because $\textstyle\prod_I R^+/p^n$ is compactly generated), we can identify the natural map on $S$-points with
\begin{align}
     \varprojlim_n \oplus_I \Hom_{\Cont, R^+/p^n}\left(R^+/p^n, R_S^+/p^n\right) \to \varprojlim_n \Hom_{\Cont,R^+/p^n}\left(\textstyle\prod_I R^+/p^n, R_S^+/p^n\right). 
\end{align}
Now $R_S^+/p^n$ has the discrete topology (see Lemma \ref{Lem:CompactOpen}), and so every continuous $R^+/p^n$-linear map $\textstyle\prod_I \ul{R^+}/p^n \to R_S^+/p^n$  has kernel containing $\textstyle\prod_{J} \{0\} \times \textstyle\prod_{I\backslash J} R^+/p^n$ for some finite $J \subseteq I$. This implies that \eqref{eq:DirectSumToProduct} is a bijection and the third claim follows. \smallskip 

For the final claim of the proposition, it suffices to show that the natural map
\begin{align}
    \Hom_{\ul{R^{+}}}\left(\textstyle\prod_I \ul{R^{+}}, \ul{R^{+}}\right)[\tfrac{1}{p}] \to \Hom_{\ul{R}}\left(\left(\textstyle\prod_I \ul{R^{+}}\right)[\tfrac{1}{p}], \ul{R}\right)
\end{align}
is an isomorphism. We can check this on $S$-points for profinite sets $S$, and then using Lemma \ref{Lem:PreEvaluationCondensed} (using Lemma \ref{Lem:CompactlyGenerated} to check its assumptions), we can identify this with
\begin{align}
    \hom_{C^0,R^+}\left(\textstyle\prod_I R,R^{+}_{S}\right)[\tfrac{1}{p}] \to  \hom_{C^0,R}\left(\left(\textstyle\prod_I R^{+}\right)[\tfrac{1}{p}],R_{S}\right).
\end{align}
This map is a bijection because the image of $\textstyle\prod_I R^+$ under any continuous morphism $(\textstyle\prod_I R^+)[\tfrac{1}{p}] \to R_{S}$ is bounded: Indeed, the preimage of $R^+_S$ intersected with $\textstyle\prod_I R^{+}$ is open and contains $0$, thus, by the definition of the product topology, contains a set of the form $(\textstyle\prod_{j\in J} p^{n_j} R^+) \times \textstyle\prod_{I \backslash J} R^+$ for some $J \subseteq I$ finite, but then, taking $n=\mathrm{max}_{j \in J}( n_j)$, we see $\textstyle\prod_I R^{+}$ maps into $p^{-n} R^+_S$. 
\end{proof}

\subsubsection{} We have the following complement to Propositions \ref{Prop:DualityCondensed} and \ref{Prop:StandardBCproperty}.
\begin{Prop} \label{Prop:TensorDualSmith}
Let $(R,R^+) \to (B,B^+)$ be a morphism of uniform Huber pairs over $(\qp,\zp)$. If $I,J$ are countable sets, then the natural maps
\begin{align}
    \hat{\oplus}_I B^+ &\to \Hom_{\ul{B^+}}\left(\textstyle\prod_I \ul{R^+} \otimes_{\ul{R^+}}^{\blacksquare} \ul{B^+}, \ul{B^+}\right) \\
    \hat{\oplus}_I B &\to \Hom_{\ul{B}}\left(\left(\textstyle\prod_I \ul{R^+}\right)[\tfrac{1}{p}] \otimes_{\ul{R}}^{\blacksquare} \ul{B}, \ul{B}\right) 
\end{align}
are isomorphisms. Moreover, the natural maps
\begin{align}
    \hat{\oplus}_{I \times J} R^+ &\to \Hom_{\ul{R^+}}\left(\textstyle\prod_I \ul{R^+} \otimes_{\ul{R^+}}^{\blacksquare} \textstyle\prod_J \ul{R^+}, \ul{R^+}\right) \\
    \hat{\oplus}_{I \times J} R &\to \Hom_{\ul{R^+}}\left(\left(\textstyle\prod_I \ul{R^+}\right)[\tfrac{1}{p}] \otimes_{\ul{R}}^{\blacksquare} \left(\textstyle\prod_J \ul{R^+}\right)[\tfrac{1}{p}], \ul{R}\right)
\end{align}
are isomorphisms.
\end{Prop}
\begin{proof}
For the first isomorphism, we notice that by the tensor-hom adjunction we have
\begin{align}
    \Hom_{\ul{B^+}}\left(\textstyle\prod_I \ul{R^+} \otimes_{\ul{R^+}}^{\blacksquare} \ul{B^+}, \ul{B^+}\right)=\Hom_{\ul{R^+}}\left(\textstyle\prod_I \ul{R^+}, \ul{B^+}\right).
\end{align}
Let $S$ be a profinite set. Then it follows from Lemma \ref{Lem:PreEvaluationCondensed} (whose assumptions hold by Lemma \ref{Lem:CompactlyGenerated}), that we may identify 
\begin{align}
    \Hom_{\ul{R^+}}\left(\textstyle\prod_I \ul{R^+}, \ul{B^+}\right)(S) &= \hom_{\Cont, R^+}\left(\textstyle\prod_I R^+, B^+_{S}\right) \\
    &=\varprojlim_{n} \hom_{\Cont, R^+}\left(\textstyle\prod_I R^+/p^n, B^+_{S}/p^n\right).
\end{align}
Arguing with the product topology as in the proof of Proposition \ref{Prop:DualityCondensed}, we further identify
\begin{align}
   \varprojlim_{n} \hom_{\Cont, R^+}\left(\textstyle\prod_I R^+/p^n, B^+_{S}/p^n\right) &=\varprojlim_{n} \oplus_I \hom_{R^+}\left( R^+/p^n, B^+_{S}/p^n\right) \\
    &=\varprojlim_{n} \oplus_I B^+_{S}/p^n \\
    &=(\hat{\oplus}_I B^+)(S).
\end{align}
For the second isomorphism, we consider the commutative diagram (for a profinite set $S$)
\begin{equation}
    \begin{tikzcd}
        \Hom_{\ul{R^+}}\left(\textstyle\prod_I \ul{R^+}, \ul{B^+}\right)(S)[\tfrac{1}{p}] \arrow{r} \arrow{d} & \hom_{\Cont, R^+}\left(\textstyle\prod_I R^+, B^+_{S}\right)[\tfrac{1}{p}] \arrow{d} \\
        \Hom_{\ul{R}}\left(\left(\textstyle\prod_I \ul{R^+}\right)[\tfrac{1}{p}], \ul{B}\right)(S) \arrow{r} & \hom_{\Cont, R}\left(\left(\textstyle\prod_I R^+\right)[\tfrac{1}{p}], B_S\right).
    \end{tikzcd}
\end{equation}
By the tensor-hom adjunction and the first claim of the Proposition, it suffices to show that the left vertical arrow is an isomorphism. The top horizontal arrow is an isomorphism by the case handled above, and it follows as in the proof of Proposition \ref{Prop:DualityCondensed} that the right vertical arrow is an isomorphism. The bottom horizontal arrow is an isomorphism by Lemma \ref{Lem:PreEvaluationCondensed} (using Lemma \ref{Lem:CompactlyGenerated} to verify the assumptions), and so we are done. 

For the third isomorphism we use Proposition \ref{Prop:DualityCondensed} and the tensor-hom adjunction to write
\begin{align}
    \Hom_{\ul{R^+}}\left(\textstyle\prod_I \ul{R^+} \otimes_{\ul{R^+}}^{\blacksquare}\textstyle\prod_J \ul{R^+}, \ul{R^+}\right)
    &=\Hom_{\ul{R^+}}\left(\textstyle\prod_I \ul{R^+}, \Hom_{\ul{R^+}}\left(\textstyle\prod_J \ul{R^+}, \ul{R^+}\right)\right) \\
    &=\Hom_{\ul{R^+}}\left(\textstyle\prod_I \ul{R^+}, \hat{\oplus}_J \ul{R^+}\right).
\end{align}
Let $S$ be a profinite set. Then it follows from Lemma \ref{Lem:PreEvaluationCondensed} (whose assumptions hold by Lemma \ref{Lem:CompactlyGenerated}), that 
\begin{align}
    \Hom_{\ul{R^+}}\left(\textstyle\prod_I \ul{R^+},  \hat{\oplus}_J \ul{R^+}\right)(S) &= \hom_{\Cont, R^+}\left(\textstyle\prod_I R^+, \left(\hat{\oplus}_J R^+\right)_{S}\right) \\
    &= \hom_{\Cont, R^+}\left(\textstyle\prod_I R^+, \hat{\oplus}_J R^+_{S}\right),
\end{align}
where we used Lemma \ref{Lem:CompactOpen} and Lemma \ref{Lem:DiscreteDirectLimit} in the last step. We further compute
\begin{align}
    \hom_{\Cont, R^+}\left(\textstyle\prod_I R^+, \hat{\oplus}_J R^+_{S}\right) &=\varprojlim_{n} \hom_{\Cont, R^+/p^n}\left(\textstyle\prod_I R^+/p^n, \oplus_{J} R^+_{S}/p^n\right) \\
    &=\varprojlim_{n} \oplus_I \hom_{R^+/p^n}\left( R^+/p^n, \oplus_{J} R^+_{S}/p^n\right),
\end{align}
where the last step follows by arguing with the product topology as in the proof of Proposition \ref{Prop:DualityCondensed}. Continuing, we get (using Lemma \ref{Lem:DiscreteDirectLimit} in the second step)
\begin{align}
    \varprojlim_{n} \oplus_I \hom_{R^+/p^n}\left( R^+/p^n, \oplus_{J} R^+_{S}/p^n\right)&=\varprojlim_{n} \oplus_I \oplus_{J} R^+_{S}/p^n \\
    &=\varprojlim_{n} \oplus_{I \times J} R^+_{S}/p^n \\
    &=\left(\hat{\oplus}_{I \times J} \underline{R^+}\right)(S).
\end{align}

 For the final isomorphism, we consider the commutative diagram (for a fixed profinite set $S$)
\begin{equation}
    \begin{tikzcd}
        \Hom_{\ul{R^+}}\left(\textstyle\prod_I \ul{R^+}, \hat{\oplus}_J \ul{R^+} \right)(S)[\tfrac{1}{p}] \arrow{r} \arrow{d} & \hom_{\Cont, R^+}(\textstyle\prod_I R^+, \hat{\oplus}_J R^+_{S})[\tfrac{1}{p}] \arrow{d} \\
        \Hom_{\ul{R}}((\textstyle\prod_I \ul{R^+})[\tfrac{1}{p}],  \hat{\oplus}_J \ul{R})(S) \arrow{r} & \hom_{\Cont, R}\left((\textstyle\prod_I R^+)[\tfrac{1}{p}], \hat{\oplus}_J R_{S}\right).
    \end{tikzcd}
\end{equation}
By the tensor-hom adjunction and the third claim of the proposition, it suffices to show that the left vertical arrow is an isomorphism. The top horizontal arrow is an isomorphism by the case handled above, and it follows as in the proof of Proposition \ref{Prop:DualityCondensed} that the right vertical arrow is an isomorphism. The bottom horizontal arrow is an isomorphism by Lemma \ref{Lem:PreEvaluationCondensed} (using Lemma \ref{Lem:CompactlyGenerated} to verify the assumptions), so we are done. \end{proof}
\begin{Rem}
    It is not true in general that the natural map $\left(\textstyle\prod_I \ul{R^+}\right) \otimes^{\blacksquare}_{\ul{R^+}} \ul{B^+} \to \textstyle\prod_I \ul{B^+}$ is an isomorphism, see Remark \ref{Rem:UncompletedTensor} and \cite[Theorem 3.27]{andreychev2021pseudocoherentperfectcomplexesvector}.
\end{Rem}

\subsection{Compact type and trace-class maps} \label{Sub:TraceClassCompact}  Let $(R,R^+)$ be a uniform Huber pair over $(\qp,\zp)$. Recall from \cite[Definition A.44]{Bosco} that a morphism $f: M \to N$ of solid $\ul{R}$-modules is said to be \emph{trace-class} if $f$ is in the image of the natural map
\begin{align}
  \left(\Hom_{\ul{R}}(M, \ul{R}) \otimes_{\ul{R}}^{\blacksquare} N\right)(\ast) \to \hom_{\ul{R}}(M,N).
\end{align}
More concretely, writing $M^{\ast}=\Hom_{\ul{R}}(M, \ul{R})$, there is a morphism $g:\ul{R} \to M^{\ast} \otimes_{\ul{R}}^{\blacksquare} N$ such that $f$ is equal to the composition
\begin{align}
    M = M \otimes_{\ul{R}}^{\blacksquare} \ul{R} \xrightarrow{1 \otimes g} M \otimes_{\ul{R}}^{\blacksquare} M^{\ast} \otimes_{\ul{R}}^{\blacksquare} N \xrightarrow{\operatorname{ev} \otimes 1} N
\end{align}
where $\mathrm{ev} \colon M \otimes_{\ul{R}}^{\blacksquare} M^{\ast} \to \ul{R}$ denotes the natural map (given by evaluation).
 The following lemma captures some of the basic properties of trace-class morphisms.
\begin{Lem} \label{Lem:TensorTraceClass}
    Let $(R,R^+) \to (B,B^+)$ be a morphism of uniform Huber pairs over $(\qp,\zp)$ and let $f:M \to N, f':M' \to N'$ be morphisms of solid $\ul{R}$-modules. 
\begin{enumerate}
    \item  If $f$ is trace-class, then so is $f \otimes 1_{\ul{B}}:M \otimes_{\ul{R}}^{\blacksquare} \ul{B} \to N \otimes_{\ul{R}}^{\blacksquare} \ul{B}$. 

    \item If $f$ and $f'$ are trace class, then so is $f \otimes f' : M \otimes_{\ul{R}}^{\blacksquare} M' \to N \otimes_{\ul{R}}^{\blacksquare} N'$. 

    \item If $f$ is trace class, then there exists a (not necessarily unique) dotted arrow in the commutative diagram 
    \begin{equation}
        \begin{tikzcd}
            M \arrow{r}{f}  \arrow{d} & N \arrow{d}\\
            M^{\ast \ast} \arrow{r}{f^{\ast \ast}} \arrow[ur, dashed, printersafe] & N^{\ast \ast},
        \end{tikzcd}
    \end{equation}
    making both triangles commute. Moreover, the map $f^{\ast \ast}$ is also trace-class.
\end{enumerate}    
\end{Lem}
\begin{proof}
The first statement follows directly from examining the commutative diagram
\begin{equation}
    \begin{tikzcd}
        \left(\Hom_{\ul{R}}(M, \ul{R}) \otimes_{\ul{R}}^{\blacksquare} N\right)(\ast) \arrow{r} \arrow{d} &\hom_{\ul{R}}(M,N) \arrow{d} \\
         \left(\Hom_{\ul{B}}(M \otimes_{\ul{R}}^{\blacksquare} \ul{B}, \ul{B}) \otimes_{\ul{B}}^{\blacksquare} N \otimes_{\ul{R}}^{\blacksquare} \ul{B} \right)(\ast) \arrow{r} & \hom_{\ul{B}}(M \otimes_{\ul{R}}^{\blacksquare} \ul{B}, N \otimes_{\ul{R}}^{\blacksquare} \ul{B}).
    \end{tikzcd}
\end{equation}
The second statement follows directly from the commutative diagram 
\begin{equation}
    \begin{tikzcd}
    & M \otimes_{\ul{R}}^{\blacksquare} M' \otimes_{\ul{R}}^{\blacksquare} \left(M \otimes_{\ul{R}}^{\blacksquare} M'\right)^{\ast}\otimes_{\ul{R}}^{\blacksquare} N \otimes_{\ul{R}} ^{\blacksquare} N' \arrow{dr} \\
   M \otimes_{\ul{R}}^{\blacksquare} M' \arrow{r} \arrow[ur, dashed, printersafe] & M \otimes_{\ul{R}}^{\blacksquare} M' \otimes_{\ul{R}}^{\blacksquare} M^{\ast} \otimes_{\ul{R}}^{\blacksquare} M^{'\ast} \otimes_{\ul{R}}^{\blacksquare} N \otimes_{\ul{R}} ^{\blacksquare} N' \arrow{r} \arrow{u} & N \otimes_{\ul{R}}^{\blacksquare} N'.
    \end{tikzcd}
\end{equation}
The third statement follows directly from the commutative diagram 
\begin{equation}
    \begin{tikzcd}
        & M^{\ast \ast} \otimes_{\ul{R}}^{\blacksquare} M^{\ast \ast \ast} \otimes_{\ul{R}}^{\blacksquare} N^{\ast \ast} \arrow{r} & N^{\ast \ast } \\
        M^{\ast \ast} \arrow{r} \arrow[dashed, printersafe, drr] \arrow[dashed, printersafe, ur] & M^{\ast \ast} \otimes_{\ul{R}}^{\blacksquare} M^{\ast} \otimes_{\ul{R}}^{\blacksquare} N \arrow{dr} \arrow{u} \\
        M \arrow{r} \arrow{u} & M \otimes_{\ul{R}}^{\blacksquare} M^{\ast} \otimes_{\ul{R}}^{\blacksquare} N \arrow{r} \arrow{u} & N \arrow{uu}.
    \end{tikzcd}
\end{equation}
\end{proof}

\subsubsection{} \label{subsub:compact} Let $(R,R^+)$ be a uniform Huber pair over $(\qp,\zp)$. Following \cite[Definition 3.33]{RJRC}, we define a morphism $f:M \to N$ of solid $\ul{R}$-modules to be \emph{compact} if the induced map
\begin{align}
    N^{\ast} \to M^{\ast}
\end{align}
is trace-class. Compact morphisms are not preserved under base change or tensor products in general: Consider $R=\mathbb{Q}_p\langle T \rangle$, $M=N=\hat{\oplus}_{\mathbb{Z}} \mathbb{Q}_p$, viewed as a module supported at $T=0$, and $f$ the identity map --- since $M^*=N^*=0$, the map $f$ is trivially compact, but its base change to $\mathbb{Q}_p$ is not compact. However, we do have the following result.
\begin{Lem} \label{Lem:CompactTensor}
    Let $(R,R^+) \to (B,B^+)$ be a morphism of uniform Huber pairs over $(\qp,\zp)$ and let $f:M \to N, f':M' \to N'$ be morphisms of solid $\ul{R}$-modules. Suppose that $M,N,M',N'$ are direct summands of countably orthonormalizable $\ul{R}$-modules.
    \begin{enumerate}
    \item  If $f$ is compact, then $f \otimes 1_{\ul{B}}:M \otimes_{\ul{R}}^{\blacksquare} \ul{B} \to N \otimes_{\ul{R}}^{\blacksquare} \ul{B}$ is compact (as a morphism of solid $\ul{B}$-modules). 

    \item If $f$ and $f'$ are both compact, then $f \otimes f' \colon M \otimes_{\ul{R}}^{\blacksquare} M' \to N \otimes_{\ul{R}}^{\blacksquare} N'$ is compact. 
    \end{enumerate}
\end{Lem}
\begin{proof}
For the first statement, we note that it follows from Propositions \ref{Prop:DualityCondensed} and \ref{Prop:TensorDualSmith} and the exactness of idempotent projection, that the natural map
\begin{align}
    M^{\ast} \otimes_{\ul{R}}^{\blacksquare} \ul{B} \to (M \otimes_{\ul{R}}^{\blacksquare} \ul{B})^{\ast}
\end{align}
identifies the target with the double dual of the source. It follows from Lemma \ref{Lem:TensorTraceClass}-(1) that $N^{\ast} \otimes_{\ul{R}}^{\blacksquare} \ul{B} \to M^{\ast} \otimes_{\ul{R}}^{\blacksquare} \ul{B}$ is trace class, and then from Lemma \ref{Lem:TensorTraceClass}-(3) that $(N \otimes_{\ul{R}}^{\blacksquare} \ul{B})^{\ast} \to (M \otimes_{\ul{R}}^{\blacksquare} \ul{B})^{\ast}$ is trace-class. For the second statement, a similar argument works once we note that
\begin{align}
    M^{\ast} \otimes_{\ul{R}}^{\blacksquare} N^{\ast} \to (M \otimes_{\ul{R}}^{\blacksquare} N)^{\ast}
\end{align}
identifies the target with the double dual of the source; this follows from Propositions \ref{Prop:DualityCondensed} and \ref{Prop:TensorDualSmith} and the exactness of idempotent projection. 
\end{proof}

\subsection{Fr\'echet modules and flatness} \label{Sub:Frechet}  Let $R$ be a $\qp$-Banach algebra. Recall that $\operatorname{Mod}_R^{\mathrm{cond}}$ denotes the category of condensed $\ul{R}$-modules, and that $\operatorname{Mod}_R^{\operatorname{solid}}$ denotes the subcategory of solid modules (i.e., modules over the analytic ring $(\ul{R}, \mbb{Z})_{\blacksquare}$ from \cite[Proposition A.29]{Bosco}, noting that $\ul{R}$ is solid). Let $\operatorname{Mod}_{R}^{\operatorname{nuc}}$ denote the full subcategory of $\operatorname{Mod}^{\operatorname{solid}}_R$ consisting of nuclear $\ul{R}$-modules (as in Definition A.40 of \emph{op.cit.}). Note that $\ul{R}$ is a nuclear solid $\mbb{Q}_p$-algebra (Corollary A.50 in \emph{op.cit.}).

Here are some properties of nuclear $\ul{R}$-modules. 

\begin{enumerate}
    \item $\operatorname{Mod}^{\operatorname{nuc}}_R$ is an abelian category stable under the solid tensor product $\otimes^{\blacksquare}_{\ul{R}}$, finite limits, countable products, and all colimits (see \cite[Theorem A.43]{Bosco}). 
    \item Orthonormalizable Banach $\ul{R}$-modules are nuclear and flat for the solid tensor product (\cite[Proposition A.55.(iii)]{Bosco}). Also, $\operatorname{Mod}^{\operatorname{nuc}}_R$ is generated under colimits by orthonormalizable Banach $\ul{R}$-modules (\cite[Proposition A.55.(ii)]{Bosco}).\footnote{We note that the objects denoted by $\Hom_{A}(A[S],A)$ in the statement of \cite[Proposition A.55.(ii)]{Bosco} are orthonormalizable Banach, and conversely any orthonormalizable Banach module is of this form. Indeed, this follows from \cite[Lemma A.52]{Bosco} and \cite[Proposition A.55.(i)]{Bosco}.}
    \item The subcategory of $\operatorname{Mod}_R^{\operatorname{nuc}}$ generated under filtered colimits by orthonormalizable Banach $R$-modules is stable under countable products (\cite[Proposition A.57]{Bosco}).
\end{enumerate}

\subsubsection{} Consider the following definition.
\begin{definition} \label{Def:Frechet}
We say a condensed $\ul{R}$-module $M$ is \emph{strongly Fr\'echet} if it is isomorphic to the limit of a countable inverse system $\{ M_n \}$ where each $M_n$ is a direct summand of an orthonormalizable Banach $\ul{R}$-module, and where the transition maps $M_{n+1}(\ast)_{\mathrm{top}} \to M_{n}(\ast)_{\mathrm{top}}$ have dense image. We say $M$ is \emph{strongly countably Fr\'echet} if each $M_n$ is a direct summand of a countably orthonormalizable Banach $\ul{R}$-module. We say that $M$ is of \emph{compact type} if the transition maps $M_{n+1} \to M_{n}$ are compact.
\end{definition}

We have the following crucial technical result about strongly Fr\'echet $\ul{R}$-modules.
\begin{Prop} \label{Prop:FrechetFlat}
If $M$ is a strongly Fr\'echet $\ul{R}$-module, then $M$ is flat for the solid tensor product $\otimes_{\ul{R}}^{\blacksquare}$.
\end{Prop}
\begin{proof}
    Let $M \simeq \varprojlim M_n$ as above. Since the forgetful functor $\operatorname{Mod}_R^{\mathrm{cond}} \to \operatorname{Mod}_{\mbb{Q}_p}^{\mathrm{cond}}$ has a left-adjoint, and exactness of sequences can be checked after composing with this forgetful functor, one has a short exact sequence 
    \[
    0 \to M \to \textstyle\prod_n M_n \xrightarrow{s} \textstyle\prod_n M_n \to 0
    \]
    where the map $s$ is the identity minus the transition morphisms. Here we are using \cite[Lemma A.37]{Bosco} with $F = \mbb{Q}_p$ to deduce that the sequence is exact on the right. Now each $M_n$ is a direct summand of an orthonormalizable Banach $\ul{R}$-module $N_n$, and it follows that
    \begin{align}
        \textstyle\prod_{n } M_n \subset \textstyle\prod_n N_n
    \end{align}
    is also the inclusion of a direct summand (take the product of the idempotents). By properties (1) and (2) above (second paragraph of \S\ref{Sub:Frechet}), the module $\textstyle\prod_n N_n$ is nuclear, and by property (3), it can be written as a filtered colimit of orthonormalizable Banach $\ul{R}$-modules. Since the solid tensor product commutes with filtered colimits, the category $\operatorname{Mod}^{\operatorname{solid}}_R$ satisfies (AB5), and orthonormalizable Banach $\ul{R}$-modules are flat (Lemma \ref{Lem:BanachFlat}), we see that $\textstyle\prod_n N_n$ is flat for the solid tensor product. This implies that $\textstyle\prod_{n } M_n$ is flat for the solid tensor product since direct summands of (solidly) $\ul{R}$-flat modules are again (solidly) $\ul{R}$-flat.

    Let $N$ be a solid $\ul{R}$-module, and let $\operatorname{Tor}_i^{R, \blacksquare}(-, N)$ denote the $i$-th left derived functor of $- \otimes_{\ul{R}}^{\blacksquare} N$. Passing to the long exact sequence for the above exact sequence, we see that $\operatorname{Tor}_i^{R, \blacksquare}(M, N)$ ($i \geq 1$) sits in the exact sequence
    \[
    \operatorname{Tor}_{i+1}^{R, \blacksquare}(\textstyle\prod_n M_n, N) \to \operatorname{Tor}_i^{R, \blacksquare}(M, N) \to \operatorname{Tor}_i^{R, \blacksquare}(\textstyle\prod_n M_n, N) .
    \]
    Now, since $\textstyle\prod_n M_n$ is flat for the solid tensor product, both of the outer terms vanish (since $i \geq 1$). This implies that $\operatorname{Tor}_i^{R, \blacksquare}(M, N) = 0$ for all $i \geq 1$ and all solid $N$, and hence $M$ is flat for the solid tensor product.
\end{proof}

\subsubsection{} We collect the following lemmas. We call an $\ul{R}$-module $M$ a \emph{Banach} module if $M=\ul{M^+}[\tfrac{1}{p}]$ where $M^+$ is a $p$-torsion free $p$-adically complete and separated topological $R^+$-module with the $p$-adic topology. Note that Banach $R^+$-modules are automatically metrizable, hence compactly generated. 

\begin{Lem} \label{Lem:DenseImageTensor}
    Let $(R,R^+)$ be a uniform Huber pair over $(\qp,\zp)$, let $M_{1}$ and $M_2$ be direct summands of orthonormalizable Banach $\ul{R}$-modules and let $N$ be a Banach $\ul{R}$-module. If $f:M_{1} \to M_{2}$ is a morphism of $\ul{R}$-modules whose induced morphism on topological $R$-modules has dense image, then the same is true for
    \begin{align}
        f \otimes_{\ul{R}}^{\blacksquare} \operatorname{Id}_{N}: M_1 \otimes_{\ul{R}}^\blacksquare N \to M_2 \otimes_{\ul{R}}^\blacksquare N.
    \end{align}
\end{Lem}
\begin{proof}
Using Lemma \ref{Lem:DenseImageTensorTopological}, it suffices to prove that for $i=1,2$ we have 
\begin{align}
    \left(M_i \otimes_{\ul{R}}^\blacksquare N\right)(\ast)_{\mathrm{top}} \simeq M_i(\ast)_{\mathrm{top}} \widehat{\otimes}_{R} N(\ast)_{\mathrm{top}}.
\end{align}
This is true when $M_i$ is orthonormalizable by Proposition \ref{Prop:TensorProductsSolidVSComplete} (because then $M_i=M_i^+[\tfrac{1}{p}]$ with $M_i^+$ flat over $\ul{R^+}$ for the solid tensor product, see Lemma \ref{Lem:BanachFlat}). The case of direct summands follows from the additivity of the completed projective tensor product. 
\end{proof}

\begin{Lem} \label{Lem:DualOfFrechet}
    Let $M=\varprojlim_n M_n$ be a strongly Fr\'echet $R$-module with presentation as in Definition \ref{Def:Frechet}. Then the natural map
    \begin{align}
        \varinjlim_n \Hom_{\ul{R}}(M_n, \ul{R}) \to \Hom_{\ul{R}}(\varprojlim_n M_n, \ul{R})
    \end{align}
    is an isomorphism. 
\end{Lem}
\begin{proof}
It suffices to check this on $S$-points for profinite sets $S$. By the adjunction argument of Lemma \ref{Lem:PreEvaluationCondensed}, this comes down to showing that 
\begin{align}
     \varinjlim_n \hom_{\ul{R}}(M_n, \ul{R_{S}}) \to \hom_{\ul{R}}(M, \ul{R_{S}})
\end{align}
is a bijection. Now $M_n(\ast)_{\mathrm{top}}$ is compactly generated and thus $\varprojlim_n M_n(\ast)_{\mathrm{top}}$ is compactly generated.\footnote{Indeed, each $M_n$ has the $p$-adic topology and is thus metrizable, which shows that $\textstyle\prod_n M_n$ is also metrizable. Now use that $\varprojlim_n M_n \subset \textstyle\prod_n M_n$ is closed and thus metrizable and hence compactly generated. } Since $\ul{\varprojlim_n M_n(\ast)_{\mathrm{top}}}=M$, we can use Lemma \ref{Lem:Fullyfaithful} to reduce to showing that the natural map
\begin{align}
     \varinjlim_n \hom_{C^0, R}(M_n(\ast)_{\mathrm{top}}, R_{S}) \to \hom_{C^0, R}(M(\ast)_{\mathrm{top}}, R_{S})
\end{align}
is a bijection. Note that $M$ is a $\qp$-Fr\'echet space and that $R_S$ is a $\qp$-Banach space, so that the result follows from the proof of \cite[Lemma 3.31]{RJRC} (here we are using the property that $M(\ast)_{\mathrm{top}} \to M_n(\ast)_{\mathrm{top}}$ has dense image to conclude that any $f \in \hom_{C^0, R}(M(\ast)_{\mathrm{top}}, R_{S})$ factors through a continuous \emph{$R$-linear} morphism $M_n(\ast)_{\mathrm{top}} \to R_S$, for some $n$).
\end{proof}
\subsubsection{} Let $V = \varprojlim_m V_m$ be a strongly countably Fr\'echet $\ul{R}$-module written as an inverse limit as in Definition \ref{Def:Frechet}. Let $\{W_n\}$ be a countable family of direct summands of orthonormalizable Banach $\ul{R}$-modules. We have the following proposition.
\begin{Prop} \label{Prop:TensorCommuteProducts}
The natural map
 \begin{align}
      V \otimes_{\ul{R}}^{\blacksquare} \textstyle\prod_n W_n  \to \textstyle\prod_{n}\left( V  \otimes_{\ul{R}}^{\blacksquare} W_n\right)
    \end{align}
    is an isomorphism.
\end{Prop}
\begin{proof}
We follow the proof of \cite[Proposition A.66]{Bosco}: As in the proof of Proposition \ref{Prop:FrechetFlat}, we can fit $V$ in a short exact sequence
\begin{align}
    0 \to V \to \textstyle\prod_m V_m \to \textstyle\prod_m V_m \to 0.
\end{align}
We now consider the commutative diagram
\begin{equation}
    \begin{tikzcd}
        0 \arrow{r} & V \otimes_{\ul{R}}^{\blacksquare} \textstyle\prod_n W_n \arrow{d} \arrow{r} &  \left(\textstyle\prod_m V_m\right) \otimes_{\ul{R}}^{\blacksquare} \left(\textstyle\prod_n W_n \right)\arrow{d} \arrow{r} &  \left(\textstyle\prod_m V_m\right) \otimes_{\ul{R}}^{\blacksquare} \left(\textstyle\prod_n W_n \right)\arrow{d} \arrow{r} & 0 \\
        0 \arrow{r} & \textstyle\prod_{n}\left( V  \otimes_{\ul{R}}^{\blacksquare} W_n\right) \arrow{r} & \arrow{r} \textstyle\prod_{n}\left( \textstyle\prod_{m} V_m  \otimes_{\ul{R}}^{\blacksquare} W_n\right) & \textstyle\prod_{n}\left( \textstyle\prod_{m} V_m  \otimes_{\ul{R}}^{\blacksquare} W_n\right) \arrow{r} & 0.
    \end{tikzcd}
\end{equation}
Note that the top row is exact because $\textstyle\prod_n W_n$ is flat by Proposition \ref{Prop:FrechetFlat}, and the bottom row is exact by the exactness of countable products (AB4*). By the five-lemma, the proposition reduces to the case that $V=\textstyle\prod_{m} V_m$. By an argument with idempotents as in the proof of Proposition \ref{Prop:FrechetFlat}, this reduces to showing that for countable index sets $I$ and $J$, the natural map
\begin{align} \label{Eq:ProductBanachMap}
    \left(\textstyle\prod_m \hat{\oplus}_{I} \ul{R} \right) \otimes^{\blacksquare}_{\ul{R}} \left(\textstyle\prod_n \hat{\oplus}_{J} \ul{R}\right) \xrightarrow{} \textstyle\prod_{n} \left(\textstyle\prod_{m} \hat{\oplus}_{I} \ul{R} \otimes^{\blacksquare}_{\ul{R}} \hat{\oplus}_{J} \ul{R} \right)
\end{align}
is an isomorphism.\footnote{More precisely, we can assume $V_m \simeq \hat{\oplus}_{I_m} \ul{R}$ and $W_n \simeq \hat{\oplus}_{J_n} \ul{R}$ for countable index sets $I_m$ and $J_n$, and without loss of generality, we may assume that $I_m = J_m = \mbb{Z}$ (since finite-free $\ul{R}$-modules are direct summands of $\hat{\oplus}_{\mbb{Z}} \ul{R}$).} By the discussion in the last paragraph of the proof of \cite[Proposition A.57]{Bosco} (they write $A$ for what we call $\ul{R}$ and $M_A$ for our $\hat{\oplus}_{I} \ul{R}$ and we are taking $F=\ul{\qp}$) we can identify
\begin{align}
    \textstyle\prod_m \hat{\oplus}_{I} \ul{R} = \left(\hat{\oplus}_{I} \ul{\qp}\right) \otimes^{\blacksquare}_{\ul{\qp}} \left(\textstyle\prod_{m} \ul{\qp}\right) \otimes^{\blacksquare}_{\ul{\qp}} \ul{R},
\end{align}
since $\hat{\oplus}_{I} \ul{R} = \hat{\oplus}_{I} \ul{\qp} \otimes^{\blacksquare}_{\ul{\qp}} \ul{R}$, see Proposition \ref{Prop:StandardBCproperty}. Thus we can rewrite the left hand side of \eqref{Eq:ProductBanachMap} as
\begin{align}
    \left(\hat{\oplus}_{I} \ul{\qp}\right) \otimes^{\blacksquare}_{\ul{\qp}}  \left(\hat{\oplus}_{J} \ul{\qp}\right) \otimes^{\blacksquare}_{\ul{\qp}} \left(\textstyle\prod_{m} \ul{\qp}\right) \otimes^{\blacksquare}_{\ul{\qp}} \left(\textstyle\prod_{n} \ul{\qp}\right)\otimes^{\blacksquare}_{\ul{\qp}} \ul{R}.
\end{align}
By \cite[Corollary A.59]{Bosco}, the proposition reduces to proving that the natural map
\begin{align}
    \left(\textstyle\prod_{m} \ul{\qp}\right) \otimes^{\blacksquare}_{\ul{\qp}} \left(\textstyle\prod_{n} \ul{\qp}\right) \to \textstyle\prod_{n,m} \ul{\qp}
\end{align}
is an isomorphism; this is explained at the end of \cite[Proposition A.66]{Bosco}, see \cite[Equation A.22]{Bosco}. 
\end{proof}
\subsubsection{} Let $V = \varprojlim_m V_m$ and $W=\varprojlim_n W_n$ be strongly countably Fr\'echet $R$-modules written as inverse limits as in Definition \ref{Def:Frechet}. We have the following proposition.
\begin{Prop} \label{Prop:TensorInverseLimit}
The natural map
 \begin{align}
      V \otimes_{\ul{R}}^{\blacksquare} W \to \varprojlim_{m,n} \left(V_{m} \otimes_{\ul{R}}^{\blacksquare} W_n\right)
    \end{align}
    is an isomorphism.
\end{Prop}
\begin{proof}
We follow the proof of \cite[Corollary A.67]{Bosco}. As in the proof of Proposition \ref{Prop:FrechetFlat}, there is an exact sequence
\begin{align}
    0 \to W \to \textstyle\prod_n W_n\to \textstyle\prod_n  W_n \to 0.
\end{align}
By the flatness proved in Proposition \ref{Prop:FrechetFlat}, it remains exact after taking the tensor product with $V$, giving 
\begin{align}
    0 \to V \otimes^{\blacksquare}_{\ul{R}} W \to V \otimes^{\blacksquare}_{\ul{R}} \textstyle\prod_n W_n\to V \otimes^{\blacksquare}_{\ul{R}} \textstyle\prod_n W_n \to 0.
\end{align}
Applying Proposition \ref{Prop:TensorCommuteProducts} gives us the exact sequence
\begin{align}
    0 \to V \otimes^{\blacksquare}_{\ul{R}} W \to \textstyle\prod_n \left(V \otimes^{\blacksquare}_{\ul{R}}  W_n\right) \to \textstyle\prod_n \left(V \otimes^{\blacksquare}_{\ul{R}}  W_n\right) \to 0,
\end{align}
which exhibits $V \otimes^{\blacksquare}_{\ul{R}} W$ as the projective limit $\varprojlim_n \left(V \otimes^{\blacksquare}_{\ul{R}} W_n\right)$. 
The same argument after tensoring $W_n$ with the exact sequence
\begin{align}
    0 \to V \to \textstyle\prod_m V_m\to \textstyle\prod_m  V_m \to 0.
\end{align}
shows that the natural map
\begin{align}
    V \otimes^{\blacksquare}_{\ul{R}} W_n\to \varprojlim_m (V_m \otimes^{\blacksquare}_{\ul{R}}  W_n)
\end{align}
is an isomorphism, completing the proof.
\end{proof}

\begin{Cor} \label{Cor:DualTensorCommute}
    For strongly countably Fr\'echet $\ul{R}$-modules $V$ and $W$ of compact type, the natural map
\begin{align}
    \Hom_{\ul{R}}(V, \ul{R}) \otimes_{\ul{R}}^{\blacksquare} \Hom_{\ul{R}}(W, \ul{R}) \to \Hom_{\ul{R}}(V \otimes_{\ul{R}}^{\blacksquare} W, \ul{R})
\end{align}
is an isomorphism.
\end{Cor}
\begin{proof}
Write $V=\varprojlim V_n, W=\varprojlim_m V_m$ as before and define $H_{n,m}=V_n \otimes^{\blacksquare}_{\ul{R}} W_m$. It follows from Lemma \ref{Lem:DenseImageTensor}, Lemma \ref{Lem:BanachFlat} that the transition maps $H_{n+1, m}(\ast)_{\mathrm{top}} \to H_{n,m}(\ast)_{\mathrm{top}}$ and $H_{n,m+1}(\ast)_{\mathrm{top}} \to H_{n,m}(\ast)_{\mathrm{top}}$ have dense image. Moreover, it is clear that each $H_{n,m}$ is a direct summand of $\hat{\oplus}_{I} \ul{R}$ for some countable set $I$. If we combine this analysis with Lemma \ref{Lem:CompactTensor}, we see that $H=\varprojlim_{n,m} H_{n,m}$ is strongly countably Fr\'echet of compact type. \smallskip

It now follows from Lemma \ref{Lem:DualOfFrechet} and Proposition \ref{Prop:TensorInverseLimit} that we may identify the map of the corollary with
\begin{align}
    \varinjlim_{n,m} V_n^{\ast} \otimes_{\ul{R}}^{\blacksquare} W_m^{\ast} \to \varinjlim_{n,m} H_{n,m}^{\ast}. 
\end{align}
We first note that it follows from Lemma \ref{Lem:TensorTraceClass}-(2) and Proposition \ref{Prop:DualityCondensed} that the transition maps in the direct system $\{ V_n^{\ast} \otimes_{\ul{R}}^{\blacksquare} W_m^{\ast}\}$ are trace class. It follows as in the proof of Lemma \ref{Lem:CompactTensor} that the natural maps
\begin{align}
   V_n^{\ast} \otimes_{\ul{R}}^{\blacksquare} W_m^{\ast} \to H_{n,m}^{\ast} 
\end{align}
identify the target with the double dual of the source. Consider the commutative diagram
\begin{equation}
    \begin{tikzcd}
        V_n^{\ast} \otimes_{\ul{R}}^{\blacksquare} W_m^{\ast} \arrow{d} \arrow{r} & V_{n+1}^{\ast} \otimes_{\ul{R}}^{\blacksquare} W_{m+1}^{\ast} \arrow{d} \\
        H_{n,m}^{\ast} \arrow{r} \arrow[ur, dashed, printersafe] & H_{n+1,m+1}^{\ast},
    \end{tikzcd}
\end{equation}
where we note that the bottom arrow is the double dual of the top arrow. It follows from Lemma \ref{Lem:TensorTraceClass}-(3) that the dashed arrow exists in a way that makes both triangles commute. This shows that the natural map
\begin{align}
      \varinjlim_{n,m} V_n^{\ast} \otimes_{\ul{R}}^{\blacksquare} W_m^{\ast} \to \varinjlim_{n,m} H_{n,m}^{\ast}
\end{align}
is an isomorphism. 
\end{proof}
\begin{Cor} \label{Cor:DualTensorFrechetBaseChange}
Let $V = \varprojlim_m V_m$ be a strongly countably Fr\'echet $\ul{R}$-module of compact type written as an inverse limit as in Definition \ref{Def:Frechet}. If $(R,R^+) \to (B,B^+)$ is a morphism of uniform Huber pairs over $(\qp,\zp)$, then the natural map
    \begin{align}
        V^{\ast} \otimes_{\ul{R}}^{\blacksquare} \ul{B} \to (V \otimes_{\ul{R}}^{\blacksquare} \ul{B})^{\ast}
    \end{align}
    is an isomorphism. 
\end{Cor}
\begin{proof}
Using Proposition \ref{Prop:TensorCommuteProducts} and Lemma \ref{Lem:DenseImageTensor} we find that $V \otimes_{\ul{R}}^{\blacksquare} \ul{B} \xrightarrow{\sim} \varprojlim_n (V_n \otimes_{\ul{R}}^{\blacksquare} \ul{B})$ is a strongly countably Fr\'echet $\ul{B}$-module. By Lemma \ref{Lem:DualOfFrechet} we may moreover identify the map of the corollary with
    \begin{align}
        \varinjlim_{n} V_n^{\ast} \otimes_{\ul{R}}^{\blacksquare} \ul{B} \to \varinjlim_{n} (V_n  \otimes_{\ul{R}}^{\blacksquare} \ul{B})^{\ast}.
    \end{align}
It follows as in the proof of Lemma \ref{Lem:CompactTensor} that the natural maps
\begin{align}
  V_n^{\ast} \otimes_{\ul{R}}^{\blacksquare} \ul{B} \to (V_n  \otimes_{\ul{R}}^{\blacksquare} \ul{B})^{\ast}.
\end{align}
identify the target with the double dual of the source. By Proposition \ref{Prop:DualityCondensed} (and the fact that $V$ is of compact type) the transition maps in the direct system $\{V_n^{\ast}\}$ are trace class. The corollary now follows as in the proof of Corollary \ref{Cor:DualTensorCommute} from Lemma \ref{Lem:TensorTraceClass}-(3).
\end{proof}

\section{Sheaves on v-stacks} \label{Sec:VSheaves} In this section we introduce a way to think about sheaves of $\mathcal{O}$-modules on nice analytic adic spaces over $\spa(\qp,\zp)$. In fact, we develop a theory that works for arbitrary small v-stacks on the site of affinoid perfectoid spaces over $\spa(\qp,\zp)$. In \S \ref{Sub:Prelim} we collect preliminaries about small v-stacks and sheaves of $\mathcal{O}$-modules on them, and introduce the notion of (fiercely) v-complete affinoid analytic adic spaces. In \S \ref{Sub:Standard}, we introduce certain standard sheaves of $\mathcal{O}$-modules and check that they (and their duals) take the expected values on v-complete affinoid analytic adic spaces. We also introduce a notion of (locally strongly countably) Fr\'echet sheaves of $\mathcal{O}$-modules in this setting, and similarly prove that they take the expected values on v-complete affinoid analytic adic spaces. In \S \ref{Sub:GeometricOrigin} we prove a criterion for checking that a sheaf ``of geometric origin" is Fr\'echet, and also prove the K\"unneth formula for sheaves ``of geometric origin". 

\newcommand{\perfd}{{\mathrm{Perfd}}}
\subsection{Preliminaries} \label{Sub:Prelim} We write $\perfd$ for the opposite of the category of perfectoid Huber pairs $(A,A^{+})$ over $(\qp,\zp)$, equipped with the v-topology, see \cite[Definition 8.1]{EtCohDiam}. We will consider stacks and sheaves in the v-topology, with an important example being the representable sheaf $\spd(A,A^+)$ for $(A,A^+) \in \perfd$. A sheaf (or stack) $X$ on $\perfd$ in the v-topology is called \emph{small} (see \cite[Section 4]{EtCohDiam}) if there is an uncountable strong limit cardinal $\kappa$, an index set $I$ of cardinality bounded by $\kappa$, a collection of $\kappa$-small perfectoid Huber pairs $(A_i,A_{i}^+)$ over $\spa(\qp,\zp)$ for $i \in I$, and a surjection 
\begin{align}
    \coprod_{i \in I} \spd(A_i, A_i^+) \to X. 
\end{align}
An important example of a small v-sheaf is $X^{\lozenge}$, where $X$ is an analytic adic space over $\spa(\qp,\zp)$, see \cite[Definition 15.5]{EtCohDiam}. In the affinoid case, we will write $\spd(R,R^+)$ for $\spa(R,R^+)^{\lozenge}$.

For a small v-stack we will write $X_{\mv}$ for the full subcategory of (small) v-sheaves $Y \to X$ where $Y=\spd(A,A^+)$ for $(A,A^+)$ a perfectoid Huber pair over $(\qp,\zp)$. This definition is different from \cite[Definition 14.1]{EtCohDiam}, which takes the category of all (small) v-sheaves $Y \to X$. However, this does not affect the category of v-sheaves on $X_v$ since every small v-sheaf $Y \to X$ can be covered by a disjoint union of $\spd(A,A^+)$'s.

\subsubsection{} Let us write $\prof$ for the category of profinite sets as before. There is a natural functor from $\prof$ to small sheaves on $\perfd$ sending $T \to \ul{T}=\spd(C^0(T,\qp), C^0(T,\zp))$, where for a topological ring $Y$ the ring $C^0(T,Y)$ is the ring of continuous functions from $T$ to $Y$, equipped with the compact-open topology. We note that by Lemma \ref{Lem:CompactOpen}, if $(R,R^+)$ is a Huber pair over $(\qp,\zp)$, then so is $(C^0(T,R), C^0(T,R^+))$. If $(R,R^+)$ is moreover affinoid perfectoid, then so is $(C^0(T,R), C^0(T,R^+))$, see \cite[Example 11.12]{EtCohDiam}. We note that $\ul{T}$ represents the functor sending $\spa(R,R^+)$ to the set of continuous functions from $|\spa(R,R^+)|$ to $T$; we will often use this in the case of $T=\zp$. \smallskip

A sheaf $\mathcal{F}$ on $X_{\mv}$ determines a sheaf of condensed sets $\mathcal{F}_{\cond}$ on $X_{\mv}$ which takes $(A,A^{+}) \in X_{\mv}$ to the condensed set\footnote{That this is a condensed set follows from the fact that if $S \to T$ is a surjective map of profinite sets, then $\spd(C^{0}(S,A), C^{0}(S,A^{+})) \to \spd(C^{0}(T,A), C^{0}(T,A^{+}))$ is a v-cover.} sending $T \in \prof$ to the set $\mathcal{F}((C^{0}(T,A), C^{0}(T,A^{+}))$. We will sometimes write $\mathcal{F}_{\cond}(-)=H^0_{\cond}(-,\mathcal{F})$, and we note that the functor sending $(A,A^+) \in X_{\mv}$ to $\mathcal{F}_{\cond}(A,A^+)$ commutes with limits and filtered colimits. 

\subsubsection{} There is a presheaf of rings $\mathcal{O}$ on $\perfd$ sending $(A,A^{+}) \mapsto A$ and similarly a presheaf $\mathcal{O}^{+}$ of rings sending $(A,A^{+}) \mapsto A^{+}$; these are sheaves by \cite[Theorem 8.7]{EtCohDiam}. Note that $\mathcal{O}^+[\frac{1}{p}]=\mathcal{O}$. If $X$ is a small v-stack, then we will write $\mathcal{O}_X$ and $\mathcal{O}_{X}^+$ for the inverse image of $\mathcal{O}$ and $\mathcal{O}^+$ along the structure map $X \to \spd \qp$. The following observation will be useful.
\begin{Lem} \label{Lem:PullBackO}
    Let $f:X \to Y$ be a morphism of small v-sheaves, then $f^{-1} \mathcal{O}_X = \mathcal{O}_Y$ and $f^{-1} \mathcal{O}_X^{+}=\mathcal{O}_Y^{+}$.
\end{Lem}
\begin{proof}
This follows directly from the definitions, since we are pulling back to slice categories. More precisely, for $\spd(A,A^+) \to X$ we have that $\left(f^{-1}\mathcal{O}_{Y}\right)(\spd(A,A^+) \to X)$ will be given by $\mathcal{O}_Y(\spd(A,A^+) \to X \to Y)=A$ which equals $\mathcal{O}_X(\spd(A,A^+) \to X)$.
\end{proof}
From now on we will write $f^{\ast}$ for $f^{-1}$. Because of Lemma \ref{Lem:PullBackO}, this will cause no ambiguity.

\begin{Cor} \label{Cor:CohomologyAndBaseChange}
Given a Cartesian diagram of small v-sheaves
\begin{equation}
\begin{tikzcd}
X' \arrow{r}{g'} \arrow{d}{f'} & X \arrow{d}{f} \\
Y' \arrow{r}{g} & Y, 
\end{tikzcd}
\end{equation}
there is a natural base change isomorphism
\begin{align}
g^{\ast} f_{\ast} \mathcal{O}_X \to f'_{\ast} \mathcal{O}_{X'}.
\end{align}
\end{Cor}
\begin{proof}
There is a natural base change map $g^{\ast} f_{\ast} \mathcal{O}_X \to f'_{\ast} (g')^{\ast} \mathcal{O}_X$ for abstract reasons, and we can identify $\mathcal{O}_{X'}= (g')^{\ast} \mathcal{O}_X$ using Lemma \ref{Lem:PullBackO}. The natural base change map is moreover an isomorphism because base change always holds when restricting to slice categories, cf. \cite[Proposition 17.6]{EtCohDiam}.
\end{proof}

\subsubsection{} Let $(R,R^+)$ be a uniform Huber pair over $\spa(\qp,\zp)$ and consider $U=\spa(R,R^+)$. Then we can consider $U^{\lozenge}=\spd(R,R^+)$ as a small v-sheaf on $\perfd$ as above. We can thus consider the condensed ring $H^0_{\cond}(U^{\lozenge}, \mathcal{O})$. Recall that $(R,R^+)$ as above is called v-complete, see \cite[Definition 9.6]{HansenKedlaya}, if the natural map $(R,R^+) \to (H^0(U^{\lozenge}, \mathcal{O}),H^0(U^{\lozenge}, \mathcal{O}^+)) $ is a bijection. The following lemma is implicit in \cite[Section 9]{HansenKedlaya}. 
\begin{Lem} \label{Lem:CondensedRingStructure}
If $(R,R^+)$ is v-complete, then the condensed ring $H^0_{\cond}(U^{\lozenge}, \mathcal{O})$ is isomorphic to the condensed ring $\underline{R}$ associated to the topological ring $R$.
\end{Lem}

\subsubsection{} \label{subsub:HuberFiberProduct} Let $(R,R^+) \to (P,P^+)$ and $(R,R^+) \to (Q,Q^+)$ be morphisms of uniform Huber pairs over $\spa(\qp,\zp)$. Recall that the fiber product $\spd(P,P^+) \times_{\spd(R,R^+)} \spd(Q,Q^+)$ is given by $\spd(E,E^+)$, where
\begin{align}
    E = \left(P^+ \widehat{\otimes}_{R^+} Q^+ \right)[\tfrac{1}{p}]
\end{align}
and where $E^+$ is the integral closure of the image of $P^+ \hat{\otimes}_{R^+} Q^+$ in $E$. In this generality, we caution that $(E,E^+)$ may not have nice properties such as uniformity or $v$-completeness, even if all three terms in the fiber product have these properties. However, it follows from the construction and  Proposition \ref{Prop:TensorProductsSolidVSComplete} and its proof that, for a profinite set $S$ and a uniform Huber pair $(R,R^+)$ over $(\qp,\zp)$, we have
\begin{align} \label{Eq:ProductProfiniteSet}
    \spd(R,R^+) \times_{\spd \qp} \ul{S} = \spd(C^{0}(S,R), C^{0}(S, R^{+})).
\end{align}

\begin{proof}[Proof of Lemma \ref{Lem:CondensedRingStructure}]
It follows from \cite[Lemma 3.6.26]{KedlayaLiu} that we can find a faithfully profinite \'etale Galois morphism $(R,R^+) \to (A,A^{+})$ with $(A,A^{+})$ affinoid perfectoid (note that v-complete implies uniform). Let $\Gamma$ be the Galois group of $\spa(A,A^{+}) \to \spa(R,R^+)$. \smallskip

The condensed ring $H^0_{\cond}(U^{\lozenge}, \mathcal{O})$ is naturally identified with $H^0_{\cond}(\spd(A,A^{+}), \mathcal{O})^{\underline{\Gamma}}$ using v-descent. On the other hand since $R$ is v-complete we deduce from \cite[Lemma 9.7]{HansenKedlaya} that the natural map $R \to A^{\Gamma}=H^0_{\cond}(\spd(A,A^{+}), \mathcal{O})^{\underline{\Gamma}}(\ast)$ is a bijection. Since $R$ has the subspace topology inside $A$, this means that $R=A^{\Gamma}$ as topological rings. Since the functor $R \mapsto \underline{R}$ is a right adjoint, it commutes with limits such as taking fixed points. It thus suffices to show that $H^0_{\cond}(\spd(A,A^{+}), \mathcal{O})=\underline{A}$ for all $(A,A^{+}) \in \perfd$. \smallskip

By definition, this means that we need to show that for $T \in \prof$ we have a natural identification $H^0_{\cond}(\spd(A,A^{+}), \mathcal{O})(T) = C^{0}(T,A)$. But this is a tautology since
\begin{align}
    H^0_{\cond}(\spd(A,A^{+}), \mathcal{O})(T) &= H^0((\spd(A,A^{+}) \times \underline{T}), \mathcal{O}) \\ &= H^0(\spd(C^{0}(T,A), C^{0}(T, A^{+})), \mathcal{O}) \\
    &=C^{0}(T,A)
\end{align}
by equation \eqref{Eq:ProductProfiniteSet} (and using that $C^0(T,A)$ is perfectoid). 
\end{proof}

\subsubsection{} \label{subsub:Fierce} The class of v-complete Huber pairs over $\spa(\qp,\zp)$ is not known to have good permanence properties, for example under rational localisations, finite \'etale maps, or the formation of Tate-algebras. For our purposes, we need a class of v-complete Huber pairs that does have these permanence properties. 
\begin{Def} \label{Def:Fierce}
Consider the largest full subcategory $\mathcal{C}_{\mathrm{frc}}$ of the category of all v-complete Huber pairs $(R,R^+)$ over $(\qp,\zp)$ such that:
\begin{itemize}
    \item If $(R,R^+) \in \mathcal{C}_{\mathrm{frc}}$, then for all rational localisations $(R,R^+) \to (B,B^+)$, we have that $(B,B^+) \in \mathcal{C}_{\mathrm{frc}}$. 

    \item If $(R,R^+) \in \mathcal{C}_{\mathrm{frc}}$, then for all finite \'etale morphisms $(R,R^+) \to (B,B^+)$, we have that $(B,B^+) \in \mathcal{C}_{\mathrm{frc}}$.

    \item If $(R,R^+) \in \mathcal{C}_{\mathrm{frc}}$, then $(R\langle T \rangle, \overline{R^+\langle T \rangle}) \in \mathcal{C}_{\mathrm{frc}}$.
\end{itemize}
Objects of $ \mathcal{C}_{\mathrm{frc}}$ are called \emph{fiercely v-complete Huber pairs}.
\end{Def}
\begin{Lem} \label{Lem:Fierce}
The following classes of Huber pairs over $(\qp,\zp)$ are fiercely v-complete: 
\begin{itemize}
        \item Affinoid perfectoid Huber pairs.

        \item Diamantine Huber pairs in the sense of \cite[Definition 11.1]{HansenKedlaya}.

        \item Huber pairs $(R,R^+)$ over $(K, \mathcal{O}_K)$ with $K$ a nonarchimedean field, such that $R$ is a seminormal affinoid algebra over $K$.\footnote{Recall that an affinoid algebra over $K$ is a quotient of $K\langle X_1, \cdots, X_n \rangle$ for some $n$.} 
    \end{itemize}
\end{Lem}
\begin{proof}
    Affinoid perfectoid Huber pairs are diamantine by \cite[Proposition 11.3]{HansenKedlaya}, so the second bullet point implies the first. The second bullet point follows from the fact that diamantine Huber pairs are themselves stable under rational localisations (\cite[Theorem 11.10]{HansenKedlaya}), finite \'etale morphisms (\cite[Lemma 11.13]{HansenKedlaya}), and forming 
    Tate algebras (\cite[Corollary 11.5]{HansenKedlaya}), and that diamantine Huber pairs are v-complete by definition. \smallskip

    To establish the third bullet point, we note that the affinoid algebras $(R,R^+)$ over $K$ with $R$ seminormal are v-complete (e.g. by \cite[Theorem 10.3]{HansenKedlaya}). Moreover, being seminormal is preserved under rational localization and finite \'etale morphisms (\cite[Lemma 6.4]{HansenKedlaya}), and under the formation of Tate algebras by \cite[Theorem 10.3, Lemma 11.4]{HansenKedlaya}. 
\end{proof}
\begin{Def} \label{Def:LocallyFiercelyVComplete}
    An analytic adic space $Y$ over $\spa(\qp,\zp)$ is said to be \emph{locally fiercely v-complete} if $Y$ is locally of the form $\spa(R,R^+)$ with $(R,R^+)$ fiercely v-complete.
\end{Def}
\begin{Lem} \label{Lem:FiercelyFullyFaithful}
    Let $X,Y$ be locally fiercely v-complete analytic adic spaces over $\spa(\qp,\zp)$. Then the natural map
    \begin{align}
        \hom(X,Y) \to \hom_{\spd \qp}(X^{\lozenge}, Y^{\lozenge})
    \end{align}
   is a bijection.
\end{Lem}
\begin{proof}
    It follows from Definition \ref{Def:Fierce} that $X$ and $Y$ are stably v-complete in the sense of \cite[Definition 9.5]{Kim}. The lemma is now a consequence of \cite[Lemma 9.8]{Kim}.
\end{proof}

\begin{Def} \label{Def:StronglySmooth}
    Let $(R,R^+)$ be a uniform Huber pair over $(\qp,\zp)$. We say that $(R,R^+) \to (P,P^+)$ is \emph{strongly smooth} if $(R,R^+) \to (P,P^+)$ factors as
    \begin{align}
        (R,R^+) \to (R\langle X_1, \cdots, X_k \rangle, \overline{R^+\langle X_1, \cdots, X_k \rangle}) \xrightarrow{f} (P,P^+),
    \end{align}
    with $f$ finite \'etale. 
\end{Def}
\begin{Lem} \label{Lem:SmoothOverFierce}
Let $(R,R^+) \to (B,B^+)$ be a morphism of Huber pairs over $(\qp, \zp)$ which 
is strongly smooth. If $(R,R^+)$ is fiercely v-complete, then so is $(B,B^+)$. 
\end{Lem}
\begin{proof}
This is a direct consequence of Definition \ref{Def:Fierce} and Definition \ref{Def:StronglySmooth}. 
\end{proof}

\subsubsection{} For any small v-stack $X$ on $\perfd$ recall that we write $X_{\mv}$ for the category of $(A,A^{+}) \in \perfd$ together with a map $\spd(A,A^{+}) \to X$, equipped with the v-topology. We will work in the category $\operatorname{Ab}(X_{\mv})$ of (small) sheaves of abelian groups on $X_{\mv}$; this is a Grothendieck abelian category closed under small limits and colimits, see \cite[03CO]{stacks-project}. We note that by definition, every cover in $X_{\mv}$ has a finite subcover; in other words, the objects $\spd(A,A^+)$ are qcqs in the topos-theoretic sense. This has the consequence that filtered colimits of sheaves can be computed pointwise, see \cite[Tag 0738]{stacks-project}.

\subsubsection{} For a small v-stack $X$ we can now consider the abelian category of (sheaves of) $\mathcal{O}_X$-modules in $X_{\mv}$, see \cite[Tag 03DA]{stacks-project}. By \cite[Tag 03DB]{stacks-project}, the forgetful functor from $\mathcal{O}_X$-modules to sheaves of abelian groups commutes with limits and colimits. As before every sheaf $F$ of $\mathcal{O}_X$-modules can be considered as a sheaf of condensed $\mathcal{O}_X$-modules over the sheaf of condensed rings $\mathcal{O}_{X,\cond}$. In other words, for $(A,A^+) \in X_{\mv}$ the condensed abelian group $H^0_{\cond}(\spd(A,A^+), F)$ is a condensed $\ul{A}$-module. 

\subsection{Standard sheaves and Fr\'echet sheaves} \label{Sub:Standard} 
For $I$ an index set we consider the following sheaves of $\mathcal{O}^+$ or $\mathcal{O}$-modules on $\perfd$
\begin{align} \label{Eq:Standard}
    \hat{\oplus}_{I} \mathcal{O}^{+}&:= \varprojlim_{n}\oplus_{I} \mathcal{O}^+/p^n \\
    \hat{\oplus}_{I} \mathcal{O}&:= \left(\hat{\oplus}_{I} \mathcal{O}^{+}\right)[\tfrac{1}{p}] \\
    \textstyle\prod_{I} \mathcal{O}^{+}&, \qquad
    (\textstyle\prod_{I} \mathcal{O}^{+})[\tfrac{1}{p}].
\end{align}
We have the following lemma.
\begin{Lem} \label{Lem:Sheaves}
For $(R,R^+)$ a v-complete Huber pair and $Y=\spd(R,R^+)$, there are natural isomorphisms
\begin{align}
H^0_{\cond}(Y, \hat{\oplus}_{I} \mathcal{O}^{+}) &\simeq \hat{\oplus}_{I} \ul{R^+} \\
H^0_{\cond}(Y, \hat{\oplus}_{I} \mathcal{O}) &\simeq \hat{\oplus}_{I} \ul{R} \\
H^0_{\cond}(Y, \textstyle\prod_I \mathcal{O}^+) &\simeq \textstyle\prod_I \ul{R^+} \\
H^0_{\cond}(Y, (\textstyle\prod_I \mathcal{O}^+)[\tfrac{1}{p}]) &\simeq (\textstyle\prod_{I} \ul{R^+})[\tfrac{1}{p}]. 
\end{align}
\end{Lem}
\begin{proof}
The natural isomorphism $\ul{R^+} \to H^0_{\cond}(Y, \mathcal{O}^{+})$ of Lemma \ref{Lem:CondensedRingStructure} induces the third natural isomorphism upon taking the product, since products commute with global sections. The fourth isomorphism is induced from the third by taking a filtered colimit (which also commutes with global sections). The second isomorphism is similarly induced from the first, which we now explain how to construct. \smallskip 

For $n > 0$, by taking the long exact sequence in condensed cohomology for $0 \rightarrow \mathcal{O}^+ \xrightarrow{p^n} \mathcal{O}^+ \rightarrow \mathcal{O}^+/p^n \rightarrow 0$, we obtain an exact sequence 
\[ 0 \rightarrow \ul{R^+/p^n} \rightarrow H^0_{\cond}(Y,\mathcal{O}^+/p^n) \rightarrow H^1_{\cond}(Y,\mathcal{O}^+)[p^n] \rightarrow 0. \]
By varying $n$ and using that the first system is Mittag--Leffler, taking the inverse limit gives an exact sequence
\[ 0 \rightarrow \ul{R^+} \rightarrow H^0_{\cond}(Y,\mathcal{O}^+) \rightarrow \varprojlim_n H^1_{\cond}(Y,\mathcal{O}^+)[p^n] \rightarrow 0. \]
The first arrow is an isomorphism by assumption, so $\varprojlim_n H^1_{\cond}(Y,\mathcal{O}^+)[p^n] =0$. Next we consider the long exact sequence in cohomology for 
\begin{align}
   0 \to \oplus_I \mathcal{O}^+ \xrightarrow{p^n} \oplus_I \mathcal{O}^+ \to \oplus_I \mathcal{O}^+/p^n \to 0,
\end{align}
giving
\[ 0 \rightarrow \oplus_I \ul{R^+/p^n} \rightarrow H^0_{\cond}(Y, \oplus_I \mathcal{O}^+/p^n) \rightarrow H^1_{\cond}(Y, \oplus_I \mathcal{O}^+) [p^n] \rightarrow 0. \] Taking the inverse limit and using the fact that the first system is Mittag--Leffler, we get an exact sequence
\begin{align}
    0 \to \hat{\oplus}_{I} \ul{R^+} \to H^0_{\cond}(Y, \hat{\oplus}_I \mathcal{O}^+) \to \varprojlim_n (H^1_{\cond}(Y, \oplus_I \mathcal{O}^+) [p^n]) \to 0.
\end{align}
But now we observe that (using the fact that cohomology commutes with filtered colimits and thus with direct sums in a coherent site)
\begin{align}
    \varprojlim_n (H^1_{\cond}(Y, \oplus_I \mathcal{O}^+) [p^n]) &= \varprojlim_n \oplus_I (H^1_{\cond}(Y, \mathcal{O}^+) [p^n]) \\
    & \subset \varprojlim_n \textstyle\prod_I (H^1_{\cond}(Y, \mathcal{O}^+) [p^n]) \\
    & \subset \textstyle\prod_I \varprojlim_n (H^1_{\cond}(Y, \mathcal{O}^+) [p^n])=0
\end{align}
to conclude. 
\end{proof}
We let $\Hom_{\mathcal{O}^+}$ denote the internal hom in sheaves of $\mathcal{O}^+$-modules, and we let $\Hom_{\mathcal{O}}$ denote the internal hom in sheaves of $\mathcal{O}$-modules. The following result should be compared to Proposition \ref{Prop:DualityCondensed}.
\begin{Prop} \label{Prop:DualityvSheaves}
The natural maps
\begin{align}
\Hom_{\mathcal{O}^+}\left(\hat{\oplus}_{I}\mathcal{O}^+, \mathcal{O}^+\right) &\to \textstyle\prod_{I} \mathcal{O}^{+} \\
\Hom_{\mathcal{O}}\left(\hat{\oplus}_{I}\mathcal{O}, \mathcal{O}\right) &\to (\textstyle\prod_{I} \mathcal{O}^{+})[\tfrac{1}{p}]
\end{align}
are isomorphisms. If $I$ is countable, then the natural maps
\begin{align}
     \hat{\oplus}_{I}\mathcal{O}^+ &\to \Hom_{\mathcal{O}^+}\left(\textstyle\prod_I \mathcal{O}^+, \mathcal{O}^+\right) \\
    \hat{\oplus}_{I}\mathcal{O} &\to \Hom_{\mathcal{O}}\left((\textstyle\prod_I \mathcal{O}^+)[\tfrac{1}{p}], \mathcal{O}\right)
\end{align}
are isomorphisms.
\end{Prop}
Our proof of the proposition is very similar to our proof of Proposition \ref{Prop:DualityCondensed}. Let $(R,R^+)$ denote a Huber pair over $(\qp,\zp)$ and let $\mathcal{O}^{?}$ denote either $\mathcal{O}^+$ or $\mathcal{O}$ (we will use similar notation for Huber pairs later) and let $F$ be a sheaf of $\mathcal{O}^{?}$-modules on $\spd(R,R^+)$. There is an evaluation map
\begin{align} \label{eq:evaluationsheaf}
    H^0_{\mathrm{cond}}(\spd(R,R^+), \Hom_{\mathcal{O}^{?}}(F, \mathcal{O}^{?})) \to \Hom_{\ul{R^{?}}}(H^0_{\mathrm{cond}}(\spd(R,R^+),F), \ul{R^{?}}).
\end{align}
\begin{Lem} \label{Lem:EvaluationAbstract}
    Let $(R,R^+)$ be a v-complete Huber pair over $(\qp,\zp)$ and let $F$ be a sheaf of $\mathcal{O}^{?}$-modules on $Y=\spd(R,R^+)$. If for all $(R,R^+) \to (A,A^+) \in \spd(R,R^+)_{\mathrm{v}}$ the natural map
    \begin{align}
        H^0_{\mathrm{cond}}(\spd(R,R^+),F) \otimes^{\blacksquare}_{\ul{R^{?}}} \ul{A^{?}} \to H^0_{\mathrm{cond}}(\spd(A,A^+),F)
    \end{align}
    is an isomorphism, then the evaluation map \eqref{eq:evaluationsheaf} is an isomorphism.
\end{Lem}
\begin{proof}
    We follow the proof of Lemma \ref{Lem:PreEvaluationCondensed}: Fix a profinite set $S$, write $\ul{S}_{(R,R^+)}$ for $\ul{S} \times_{\spd \qp} \spd(R,R^+)$ and consider the evaluation of \eqref{eq:evaluationsheaf} on $S$
    \begin{multline} \label{eq:evaluationsheafevaluation}
        H^0_{\mathrm{cond}}(\spd(R,R^+), \Hom_{\mathcal{O}^{?}}(F, \mathcal{O}^{?}))(S) \\\to \Hom_{\ul{R^{?}}}(H^0_{\mathrm{cond}}(\spd(R,R^+),F), \ul{R^{?}})(S).
    \end{multline}
    We can identify the left hand side with the set
    \begin{align}
        H^0(\spd(C^{0}(S,R),C^{0}(S,R^+)), \Hom_{\mathcal{O}^{?}}(F, \mathcal{O}^{?})),
    \end{align}
    so a section gives a map $f \in \hom_{\mathcal{O}^{?}}(F, \mathcal{O}^{?})$ of sheaves on $(\ul{S}_{(R,R^+)})_{\mathrm{v}}$. The map $f$ is uniquely determined by the induced map of condensed sets (notation as in Lemma \ref{Lem:PreEvaluationCondensed})
    \begin{align}
        H^0_{\mathrm{cond}}(\ul{S}_{(R,R^+)},F) \to \ul{R^{?}_{S}}
    \end{align}
    since for any $\spd(A,A^+) \to \ul{S}_{(R,R^+)}$, there is a unique induced map
    \begin{align}
       H^0_{\mathrm{cond}}(\spd(A,A^+),F) \to \ul{A^{?}}
    \end{align}
    given by $f \otimes^{\blacksquare}_{\mathcal{O}^{?}(\ul{S}_{(R,R^+)})} \ul{A^{?}}$; this uses the assumption of the lemma. \smallskip
    
    The right hand side of \eqref{eq:evaluationsheafevaluation} is the set of $ \pi^{-1}\ul{R^{?}}$-linear homomorphisms (notation as in the proof of Lemma \ref{Lem:PreEvaluationCondensed})
    \begin{align}
        \pi^{-1}H^0_{\mathrm{cond}}(\spd(R,R^+),F) \to \pi^{-1}\ul{R^{?}}
    \end{align}
    on the site of profinite sets over $S$. Using the adjunction argument as in Lemma \ref{Lem:PreEvaluationCondensed}, we identify this with the set of $\ul{R^{?}}$-linear homomorphisms
    \begin{align}
        H^0_{\mathrm{cond}}(\spd(R,R^+),F) \to \ul{R^{?}_{S}},
    \end{align}
    which is identified with the left hand side of \eqref{eq:evaluationsheafevaluation} as above. 
\end{proof}

\begin{Cor} \label{Cor:CondensedEvaluationMaps}
For a v-complete Huber pair $(R,R^+)$ over $\spa(\qp,\zp)$, the natural condensed evaluation maps
\begin{align}
    H^0_{\cond}(\spd(R,R^{+}), \Hom_{\mathcal{\mathcal{O}^{+}}}(\hat{\oplus}_{I} \mathcal{O}^+, \mathcal{O}^+)) &\to\Hom_{\underline{R}^+}(\hat{\oplus}_{I} \underline{R^+} , \underline{R^+}) \\
    H^0_{\cond}(\spd(R,R^{+}),\Hom_{\mathcal{\mathcal{O}}}(\hat{\oplus}_{I} \mathcal{O}, \mathcal{O})) &\to\Hom_{\underline{R}}(\hat{\oplus}_{I} \underline{R} , \underline{R}) 
\end{align}
are isomorphisms. If $I$ is moreover countable, then the natural condensed evaluation maps
\begin{align}
    H^0_{\cond}(\spd(R,R^{+}),\Hom_{\mathcal{\mathcal{O}^+}}(\textstyle\prod_I \mathcal{O}^+, \mathcal{O}^+)) &\to\Hom_{\underline{R}^+}(\textstyle\prod_I \underline{R^+} , \underline{R^+}) \\
    H^0_{\cond}(\spd(R,R^{+}),\Hom_{\mathcal{\mathcal{O}}}((\textstyle\prod_I  \mathcal{O}^+)[\tfrac{1}{p}], \mathcal{O})) &\to\Hom_{\underline{R}}((\textstyle\prod_I \underline{R})[\tfrac{1}{p}], \underline{R})
\end{align}
are isomorphisms.
\end{Cor}
\begin{proof}
The first two statements are a direct consequence of Lemma \ref{Lem:EvaluationAbstract} and the fact that the formation of $\hat{\oplus}_I R^+$ and $\hat{\oplus}_I R$ commute with the solid tensor product, see Proposition \ref{Prop:StandardBCproperty}. For the third and fourth claim, we slightly modify the argument of Lemma \ref{Lem:EvaluationAbstract}. Write
\begin{align}
    F^? &= \twopartdef{\textstyle\prod_I \mathcal{O}^+}{?=+}{(\textstyle\prod_I \mathcal{O}^+)[\tfrac{1}{p}]}{?=\emptyset.} \\
    F^?_R &= \twopartdef{\textstyle\prod_I \ul{R^+}}{?=+}{(\textstyle\prod_I \ul{R^+})[\tfrac{1}{p}]}{?=\emptyset.}
\end{align}
It follows from Proposition \ref{Prop:DualityCondensed}, Proposition \ref{Prop:TensorDualSmith} and Lemma \ref{Lem:Sheaves} that for all morphisms $(A,A^+) \to (B,B^+)$ of v-complete Huber pairs over $(\qp,\zp)$, the natural map
\begin{align}
    H^0_{\cond}(\spd(A,A^+), F^?) \otimes_{\ul{A^?}} \ul{B^?} \to H^0_{\cond}(\spd(B,B^+), F^?)
\end{align}
identifies the target with the double dual of the source. Using this, we can write down an explicit inverse of the natural map of the corollary: For a fixed profinite set $S$ we can identify the natural map as in the proof of Lemma \ref{Lem:EvaluationAbstract} with
\begin{align} \label{Eq:BigEquationReferenceFiveHundred}
     &H^0_{\cond}(\spd(R,R^{+}),\Hom_{\mathcal{\mathcal{O}^?}}(F^?, \mathcal{O}^?))(S) = H^0(\spd(R_S,R^+_S), \Hom_{\mathcal{\mathcal{O}^?}}(F^?, \mathcal{O}^+)) \to \\
     &\Hom_{\underline{R}^?}(F^?_{R}, \underline{R^?})(S) = \hom_{\underline{R}^?}(F^?_{R}, \underline{R^?_{S}}).
\end{align}
An inverse is now given by taking $f: F^?_{R} \to \underline{R^?_{S}}$ to the morphism of sheaves over $\spd(R_S,R_S^+)$ whose evaluation on $(R_S, R_S^+) \to (A,A^+)$ is: We first extend $f$ to a morphism $F^?_{R_S} \to \underline{R^?_{S}}$ by taking (solid) extension of scalars along $\ul{R^?} \to \ul{R_S^?}$ and then taking double dual, we then take solid extension of scalars of $f$ along $\ul{R^?_S} \to \ul{A^?}$ and take double dual again.
\end{proof}
\begin{proof}[Proof of Proposition \ref{Prop:DualityvSheaves}]
All statements are a direct consequence of Proposition \ref{Prop:DualityCondensed} and Corollary \ref{Cor:CondensedEvaluationMaps}.
\end{proof}

\subsubsection{} Let $Y$ be a small v-sheaf and let $F$ be a sheaf of $\mathcal{O}$-modules on $Y_{\mv}$. We say that $F$ is \emph{countably orthonormalizable Banach} if $F \simeq \hat{\oplus}_{I} \mathcal{O}_{Y}$ for some countable set $I$. We say that $F$ is \emph{locally countably orthonormalizable Banach} if there is a v-cover $Y' \to Y$ such that the pullback of $F$ to $Y'$ is isomorphic to $\hat{\oplus}_{I} \mathcal{O}_{Y'}$ for some countable set $I$. We say that $F$ is \emph{reflexive} if the natural map $F \rightarrow \Hom_{\mathcal{O}}(\Hom_\mathcal{O}(F, \mathcal{O}))$ is an isomorphism. 
\begin{Cor} \label{Cor:BanachReflexive}
    If $F$ is a locally countably orthonormalizable Banach sheaf of $\mathcal{O}$-modules on $Y_{\mv}$, then the sheaf $F$ is reflexive.
\end{Cor}
\begin{proof}
    This can be checked locally in the topology, where it follows from Proposition \ref{Prop:DualityvSheaves}.
\end{proof}
We need the following refinement.
\begin{Cor} \label{Cor:SummandReflexive}
    Suppose that $F$ is v-locally a direct summand of a locally countably orthonormalizable Banach sheaf of $\mathcal{O}_Y$-modules. Then $F$ is reflexive. 
\end{Cor}
\begin{proof}
The statement is local, so we may suppose that $F=F_1 \xrightarrow{\oplus} F_2$ where $F_2$ is a locally free Banach sheaf of $\mathcal{O}_Y$-modules; let $\Theta$ be an idempotent of $F_2$ with kernel $F_1$. The point is now that splitting an idempotent is both a limit and a colimit, hence using Corollary \ref{Cor:BanachReflexive} we get
\begin{align}
    \Hom_{\mathcal{O}_Y}(\Hom_{\mathcal{O}_Y}(F_1, \mathcal{O}_Y), \mathcal{O}_Y) &= \Hom_{\mathcal{O}_Y}(\Hom_{\mathcal{O}_Y}(F_2, \mathcal{O}_Y)^{\Theta=0}, \mathcal{O}_Y)  \\
    &= \Hom_{\mathcal{O}_Y}(\Hom_{\mathcal{O}_Y}(F_2, \mathcal{O}_Y), \mathcal{O}_Y)^{\Theta=0} \\
    &=F_1.
\end{align}
\end{proof}

\begin{Def} \label{Def:LocallyFrechet}
A sheaf $F$ of $\mathcal{O}_Y$-modules is called \emph{locally strongly countably Fr\'echet} if there is a v-cover $Y'=\coprod_{j \in J} \spd(A_{j},A_{j}^{+}) \to Y$ with $(A_j, A_j^+) \in \perfd$ such that the pullback of $f_j^{\ast}F$ to $\spd(A_{j},A_{j}^{+})$ can be written as an inverse limit $f_j^{\ast}F=\varprojlim_n F_n$ with the property that: T he sheaf $F_n$ is isomorphic to a direct summand of $\hat{\oplus}_{I} \mathcal{O}_{Y'}$ for some countable set $I$ and the transition maps $H^0_{\cond}(\spd(A_{j},A_{j}^{+}), F_{n+1})(\ast)_{\mathrm{top}} \to H^0_{\cond}(\spd(A_{j},A_{j}^{+}), F_{n})(\ast)_{\mathrm{top}}$ have dense image. If the transition maps $H^0_{\cond}(\spd(A_{j},A_{j}^{+}), F_{n+1}) \to H^0_{\cond}(\spd(A_{j},A_{j}^{+}), F_{n})$ are moreover compact, then $F$ is called \emph{locally strongly countably Fr\'echet of compact type}.  
\end{Def} 
\begin{Lem} \label{Lem:EvaluationFrechet}
Let $Y=\spd(R,R^+)$ with $(R,R^+)$ a v-complete Huber pair over $(\qp,\zp)$ and let $F=\varprojlim_n F_n$ be a sheaf of $\mathcal{O}_Y$-modules such that for all $n$ the sheaf $F_n$ is isomorphic to a direct summand of $\hat{\oplus}_{I} \mathcal{O}_{Y}$ for some countable set $I$, and such that the transition maps $H^0_{\cond}(\spd(R,R^+), F_{n+1})(\ast)_{\mathrm{top}} \to H^0_{\cond}(\spd(R,R^+), F_{n})(\ast)_{\mathrm{top}}$ have dense image. Then for $(R,R^+) \to (B,B^+)$ a morphism of v-complete Huber pairs over $(\qp,\zp)$, the natural map
\begin{align}
    H^0_{\mathrm{cond}}(\spd(R,R^+),F) \otimes^{\blacksquare}_{\ul{R}} \ul{B} \to H^0_{\mathrm{cond}}(\spd(B,B^+),F)
\end{align}
is an isomorphism. 
\end{Lem}
\begin{proof}
The formation of $\hat{\oplus}_{I} \ul{R}$ commutes with the solid tensor product (Proposition \ref{Prop:StandardBCproperty}), so the lemma holds for $F_n$ for all $n$ since tensor products and global sections commute with idempotent projections. It follows from the definitions that $H^0_{\mathrm{cond}}(\spd(R,R^+),F)$ is a strongly countably Fr\'echet $\ul{R}$-module, and thus it follows from Proposition \ref{Prop:TensorInverseLimit} that
\begin{align}
    H^0_{\mathrm{cond}}(\spd(R,R^+),F) \otimes^{\blacksquare}_{\ul{R}} \ul{B} &= \left(\varprojlim_n H^0_{\mathrm{cond}}(\spd(R,R^+),F_n) \right)\otimes^{\blacksquare}_{\ul{R}} \ul{B} \\
    &\simeq \varprojlim_n\left(H^0_{\mathrm{cond}}(\spd(R,R^+),F_n)\otimes^{\blacksquare}_{\ul{R}} \ul{B} \right).
\end{align}
Since the lemma holds for each $F_n$, we can identify this with
\begin{align}
    \varprojlim_n\left(H^0_{\mathrm{cond}}(\spd(A,A^+),F_n)\otimes^{\blacksquare}_{\ul{R}} \ul{B} \right) &= \varprojlim_n\left(H^0_{\mathrm{cond}}(\spd(B,B^+),F_n)\right) \\
    &=H^0_{\mathrm{cond}}(\spd(B,B^+),F).
\end{align}
\end{proof}
\begin{Lem} \label{Lem:FrechetDualSheaf}
Let $Y=\spd(R,R^+)$ with $(R,R^+)$ a v-complete Huber pair over $(\qp,\zp)$ and let $F=\varprojlim_n F_n$ be a sheaf of $\mathcal{O}_Y$-modules such that for all $n$ the sheaf $F_n$ is isomorphic to a direct summand of $\hat{\oplus}_{I} \mathcal{O}_{Y}$ for some countable set $I$, and such that the transition maps $H^0_{\cond}(\spd(R,R^+), F_{n+1})(\ast)_{\mathrm{top}} \to H^0_{\cond}(\spd(R,R^+), F_{n})(\ast)_{\mathrm{top}}$ have dense image. Then the natural map
\begin{align}
    \varinjlim_n \Hom_{\mathcal{O}_Y}(F_n, \mathcal{O}_Y) \to \Hom_{\mathcal{O}_Y}(\varprojlim_n F_n, \mathcal{O}_Y) 
\end{align}
is an isomorphism.
\end{Lem}
\begin{proof}
By Lemma \ref{Lem:EvaluationFrechet}, the assumptions of Lemma \ref{Lem:EvaluationAbstract} hold. The result then follows from Lemma \ref{Lem:EvaluationAbstract} and Lemma \ref{Lem:DualOfFrechet}. 
\end{proof}
The statement of Lemma \ref{Lem:FrechetDualSheaf} can be thought of as saying that $\Hom_{\mathcal{O}_Y}(F, \mathcal{O}_Y)$ is the dual ``equipped with the weak topology". Indeed, the filtered colimit of the weak topologies does indeed give the weak topology on the dual of the inverse limit (given that every functional factors through some stage of the inverse limit).
\begin{Cor} \label{Cor:FrechetReflexive}
    A locally strongly countably Fr\'echet $\mathcal{O}_Y$-module $F$ is reflexive. 
\end{Cor}
\begin{proof}
This can be proved locally in the v-topology, so we may assume that we are in the setting of Lemma \ref{Lem:FrechetDualSheaf}. Then it follows from Lemma \ref{Lem:FrechetDualSheaf} that
    \begin{align}
        \Hom_{\mathcal{O}_{Y}}(\Hom_{\mathcal{O}_{Y}}(\varprojlim_n F_n, \mathcal{O}_Y), \mathcal{O}_Y) &= \Hom_{\mathcal{O}_{Y}}(\varinjlim_n \Hom_{\mathcal{O}_{Y}}(F_n, \mathcal{O}_Y), \mathcal{O}_Y) \\
        &=\varprojlim_n \Hom_{\mathcal{O}_{Y}}(\Hom_{\mathcal{O}_{Y}}(F_n, \mathcal{O}_Y), \mathcal{O}_Y) \\
        &=\varprojlim_n F_n,
    \end{align}
    where we used Corollary \ref{Cor:SummandReflexive} in the last step.
\end{proof}

\subsubsection{Solid tensor products} \label{Subsubb:SolidTensor} Let $Y$ be a small v-stack and let $F,G$ be $\mathcal{O}_Y$-modules. Then we write
\begin{align}
    F \otimes^{\blacksquare}_{\mathcal{O}_Y} G
\end{align}
for the sheafification of the presheaf sending $(A, A^+)$ to
\begin{align}
    \left(H^0_{\cond}(\spd(A,A^+), F) \otimes^{\blacksquare}_{\ul{A}} H^0_{\cond}(\spd(A,A^+), G)\right)(\ast),
\end{align}
where $\otimes^{\blacksquare}_{\ul{A}}$ is the solid tensor product.

\begin{Rem}
As a warning, we note that $F \otimes_{\mathcal{O}_Y}^{\blacksquare} \mathcal{O}_Y$ might not be equal to $F$, because $H^0_{\cond}(\spd(A,A^+), F)$ might not be solid. Even when starting with sheaves satisfying this solidity condition, it is not clear that the tensor product is associative or has any other good formal properties (we thank Peter Scholze for pointing this out to us). 
\end{Rem}

\subsection{Sheaves of geometric origin} \label{Sub:GeometricOrigin} Let $f:X \to Y$ be a morphism of small v-sheaves and consider the sheaf of $\mathcal{O}_Y$-algebras $\mathcal{O}_{X/Y}:=f_{\ast} \mathcal{O}_{X}$. We are going to study this sheaf under some assumptions.
\begin{Lem} \label{Lem:SheafBanach}
Let $(R,R^+)$ be a fiercely v-complete Huber pair and let $(R,R^+) \to (P,P^+)$ be strongly smooth. For $X=\spd(P,P^+) \to \spd(R,R^+)=Y$, the sheaf $\mathcal{O}_{X/Y}$ is a direct summand of a countably orthonormalizable Banach sheaf of $\mathcal{O}_Y$-modules. 
\end{Lem}
\begin{proof}
Factor $(R,R^+) \to (P,P^+)$ as $(R,R^+) \to (R\langle X_1, \cdots, X_k \rangle, \overline{R^+\langle X_1, \cdots, X_k \rangle}) \to (P,P^+)$ with the second map finite \'etale as in the definition. This exhibits $(P,P^+) $ as a finite locally free module over $R\langle X_1, \cdots, X_k \rangle$, and thus an $R\langle X_1, \cdots, X_k \rangle$-linear direct summand of $R\langle X_1, \cdots, X_k \rangle^{\oplus N}$ for some $N$, see \cite[Tag 00NX]{stacks-project}. \smallskip 

Write $Z=\spd(R\langle X_1, \cdots, X_k \rangle, \overline{R^+\langle X_1, \cdots, X_k \rangle})$. Then we claim that
\begin{align}
    \mathcal{O}_{Z/Y} \simeq \hat{\oplus}_I \mathcal{O}_Y
\end{align}
with $I=\mathbb{Z}_{\ge 0}^{k}$. The claim follows from Corollary \ref{Cor:CohomologyAndBaseChange} together with the fact that the formation of Tate algebras commutes with fiber products of adic spaces (see Section \ref{subsub:HuberFiberProduct}) and Lemma \ref{Lem:SmoothOverFierce}, which gives the v-completeness of Tate algebras over affinoid perfectoid $(A,A^+)$. Since $X \to Z$ is finite \'etale and $(R,R^+)$ is fiercely v-complete, it follows in a similar way that (here we are taking the algebraic tensor product)
\begin{align}
    \mathcal{O}_{X/Y} \simeq \mathcal{O}_{Z/Y} \otimes_{H^0(Y, \mathcal{O}_{Z/Y})} P.
\end{align}
Now since $P$ is an $H^0(Y, \mathcal{O}_{Z/Y})=R\langle X_1, \cdots, X_k \rangle$-linear direct summand of $H^0(Y, \mathcal{O}_{Z/Y})^{\oplus N}$ for some $N$, it follows that $\mathcal{O}_{X/Y}$ is a $\mathcal{O}_{Z/Y}$-linear direct summand of $\mathcal{O}_{Z/Y}^{\oplus N}$ for some $N$, proving the lemma. 
\end{proof}

\begin{Lem} \label{Lem:SheafFrechet}
Let $Y=\spd(R,R^+)$ for $(R,R^+)$ a fiercely v-complete Huber pair over $(\qp,\zp)$. Let $f:X \to Y$ be a morphism of small v-sheaves such that $X=\varinjlim_n U_n$, where $U_n \to \spd(R,R^+)$ is given by $\spd(P_n,P_n^+) \to \spd(R,R^+)$ with $(R,R^+) \to (P_n,P_n^+)$ strongly smooth. If the transition maps $P_{n+1} \to P_{n}$ have dense image,  then $\mathcal{O}_{X/Y}$ is locally strongly countably Fr\'echet. If the transition maps $\ul{P_{n+1}} \to \ul{P_{n}}$ are moreover compact, then $\mathcal{O}_{X/Y}$ is locally strongly countably Fr\'echet of compact type. 
\end{Lem}
\begin{proof}
First of all $\mathcal{O}_{U_n/Y}$ is a direct summand of a countably orthonormalizable Banach sheaf of $\mathcal{O}_Y$-modules by Lemma \ref{Lem:SheafBanach}. It follows from Lemma \ref{Lem:SmoothOverFierce} and Lemma \ref{Lem:CondensedRingStructure} that $H^0_{\cond}(Y, \mathcal{O}_{U_n/Y}) \simeq P_n$, thus the transition maps $H^0_{\cond}(Y, \mathcal{O}_{U_n/Y})(\ast)_{\mathrm{top}} \to H^0_{\cond}(Y, \mathcal{O}_{U_n/Y})(\ast)_{\mathrm{top}}$ have dense image. Moreover, the transition maps $H^0_{\cond}(Y,\mathcal{O}_{U_n/Y})\to H^0_{\cond}(Y, \mathcal{O}_{U_n/Y})$ are compact if the maps $\ul{P_{n+1}} \to \ul{P_{n}}$ are compact. The lemma is thus proved if we show that the natural map 
\begin{align}
    \mathcal{O}_{X/Y} \to \varprojlim_n \mathcal{O}_{U_n/Y}
\end{align}
is an isomorphism. We can check this on $(A,A^+) \in \spd(R,R^+)_{\mv}$, and then
\begin{align}
    H^0(X_{\spd(A,A^+)}, \mathcal{O}) &= \hom\left(X_{\spd(A,A^+)}, \mathbb{A}^{1,\lozenge}_{\spd(A,A^+)}\right) \\
    &=\hom\left(\varinjlim_n U_{n,\spd(A,A^+)} , \mathbb{A}^{1,\lozenge}_{\spd(A,A^+)}\right) 
\end{align}
which identifies with $H^0\left(\spd(A,A^+), \varinjlim_n \mathcal{O}_{U_n/Y}\right) = \varprojlim_n H^0\left(\spd(A,A^+), \mathcal{O}_{U_n/Y}\right)$, proving the lemma. 
\end{proof}

\subsubsection{} Let $X \to Y, Z \to Y$ be morphisms of small v-stacks. Then there is a natural cup product morphism
\begin{align} \label{eq:TensorProductMap}
\mathcal{O}_{X/Y} \otimes^{\blacksquare}_{\mathcal{O}_Y}\mathcal{O}_{Z/Y} \to \mathcal{O}_{\left(X \times_{Y} Z\right)/ Y}.
\end{align}
given by the solidification (and then sheafification) of the usual cup product map (see e.g. \cite[Tag 0B6C]{stacks-project}). 
\begin{Lem} \label{Lem:Kunneth}
    Let $Y=\spd(R,R^+)$ for a fiercely v-complete Huber pair $(R,R^+)$ over $(\qp,\zp)$. Let $(R,R^+) \to (P,P^+)$ and $(R,R^+) \to (Q,Q^+)$ be strongly smooth morphisms of Huber pairs over $(\qp,\zp)$. If we write $X=\spd(P,P^+)$ and $Z=\spd(Q,Q^+)$, then the natural map
    \begin{align}
        H^0_{\cond}(Y,\mathcal{O}_{X/Y}) \otimes_{\ul{R}}^{\blacksquare} H^0_{\cond}(Y,\mathcal{O}_{Z/Y}) \to H^0_{\cond}(Y,\mathcal{O}_{(X \times_{Y} Z)/Y})
    \end{align}
    is an isomorphism. 
\end{Lem}
\begin{proof}
Factor $(R,R^+) \to (P,P^+)$ as  $$(R\langle X_1, \cdots, X_k \rangle, \overline{R^+\langle X_1, \cdots, X_k \rangle}) \xrightarrow{f} (P,P^+)$$ with $f$ finite \'etale and similarly factor $(R,R^+) \to (Q,Q^+)$ as $$(R\langle X_{k+1}, \cdots, X_{k+j} \rangle, \overline{R^+\langle X_{k+1}, \cdots, X_{k+j} \rangle}) \xrightarrow{g} (Q,Q^+)$$ with $g$ finite \'etale. It follows from the discussion in \S\ref{subsub:HuberFiberProduct} and the preservation of Tate rings under fiber products of Huber pairs over $(\qp,\zp)$ that we can describe the fiber product $\spa(Q,Q^+) \times_{\spa (R,R^+)} \spa(P,P^+)$ as $\spa(E,E^+)$ where 
\begin{align}
   E=Q \langle X_1, \cdots, X_k \rangle \otimes_{R\langle X_1, \cdots, X_{j+k} \rangle} P \langle X_{k+1}, \cdots, X_{k+j} \rangle,
\end{align}
In particular, the natural map $(R,R^+) \to (E,E^+)$ is strongly smooth, and so $(E,E^+)$ is v-complete by Lemma \ref{Lem:SmoothOverFierce}. Moreover, since
\begin{align}
    \ul{R\langle X_1, \cdots, X_k \rangle} \otimes^{\blacksquare}_{\ul{R}}\ul{R\langle X_{k+1}, \cdots, X_{k+j} \rangle} \simeq \ul{R\langle X_1, \cdots, X_{k+j} \rangle}
\end{align}
by Proposition \ref{Prop:StandardBCproperty}, it follows that $\ul{E}=\ul{P} \otimes^{\blacksquare}_{\ul{R}} \ul{Q}$. Thus by Lemma \ref{Lem:CondensedRingStructure} the natural maps
\begin{align}
     \ul{P} &\to H^0_{\mathrm{cond}}(X, \mathcal{O}) \\
     \ul{Q} &\to H^0_{\mathrm{cond}}(Z, \mathcal{O}) \\
     \ul{P} \otimes^{\blacksquare}_{\ul{R}} \ul{Q} &\to H^0_{\cond}(X \times_{Y} Z, \mathcal{O})
\end{align}
are isomorphisms, proving the lemma.   
\end{proof}
\begin{Lem} \label{Lem:Kunneth1}
Let $X \to Y, Z \to Y$ be morphisms of small v-stacks. If there is a v-cover $Y'=\coprod_{j \in J} \spd(A_j,A^+_{j}) \to Y$ with $(A_j,A_j^+) \in \perfd$ such that: For each $j \in J$ we have $Z_{\spd(A_j,A^+_{j})} \simeq \spd(R,R^+)$ and $X_{\spd(A_j,A^+_{j})} \simeq \spd(P,P^+)$ for $(A_j,A_j^{+}) \to (R,R^+)$ and $(A_j,A_j^{+}) \to (P,P^+)$ strongly smooth, then the natural map of \eqref{eq:TensorProductMap} is an isomorphism. 
\end{Lem}
\begin{proof}
It suffices to prove this after basechanging to $Y'$, so we will assume $Y=Y'=\spd(A,A^+)$. We can then check the statement of the proposition on $\spd(B,B^+)$ points for $(B,B^+) \in \spd(A,A^+)_{\mathrm{v}}$ and we may assume without loss of generality that $(B,B^+)=(A,A^+)$, noting that the assumptions of the lemma are stable under basechange. The lemma now follows from Lemma \ref{Lem:Kunneth}.
\end{proof}

\begin{Prop} \label{Prop:Kunneth3}
Let $Y=\spd(R,R^+)$ for a fiercely v-complete Huber pair $(R,R^+)$ over $(\qp,\zp)$. Let $X=\varinjlim_n U_n$, where $U_n \to \spd(R,R^+)$ is given by $\spd(P_n,P_n^{+}) \to \spd(R,R^+)$ for some strongly smooth $(R,R^+) \to (P_n,P_n^+)$. If the transition maps $P_{n+1} \to P_{n}$ have dense image, then the natural map
 \begin{align}
        H^0_{\cond}(Y,\mathcal{O}_{X/Y}) \otimes_{\ul{R}}^{\blacksquare} H^0_{\cond}(Y,\mathcal{O}_{X/Y}) \to H^0_{\cond}(Y,\mathcal{O}_{(X \times_{Y} X)/Y})
    \end{align}
    is an isomorphism.
\end{Prop}
\begin{proof}
It follows from Lemma \ref{Lem:Kunneth} that the natural map
\begin{align}
    H^0_{\cond}(Y,\mathcal{O}_{U_{n}/Y}) \otimes_{\ul{R}}^{\blacksquare} H^0_{\cond}(Y,\mathcal{O}_{U_{n}/Y}) \to H^0_{\cond}(Y,\mathcal{O}_{(U_{n} \times_{Y} U_{n})/Y})
\end{align}
is an isomorphism. The lemma now follows by observing that
\begin{align}
    H^0_{\cond}(Y,\mathcal{O}_{X/Y}) \simeq \varprojlim_n H^0_{\cond}(Y,\mathcal{O}_{U_{n}/Y}),
\end{align}
see the proof of Lemma \ref{Lem:SheafFrechet}, and applying Proposition \ref{Prop:TensorInverseLimit} to commute the solid tensor product with the inverse limit in both factors. 
\end{proof}
\begin{Prop} \label{Prop:Kunneth2}
    Let $f:X \to Y$ be a morphism of small v-sheaves. Assume that there is a v-cover $Y'=\coprod_{j \in J} \spd(A_j,A^+_{j}) \to Y$ with $(A_j,A_j^+) \in \perfd$ such that: For each $j \in J$ we have $X_{\spd(A_j,A^+_{j})}=\varinjlim_n U_n$ where $U_n \to \spd(A_j,A_j^+)$ is given by $\spd(A_j,A_j^+) \to \spd(P_{n,j},P_{n,j}^+)$ for some strongly smooth $(A_j,A_j^+) \to (P_{n,j},P_{n,j}^+)$. If the transition maps $P_{n,j} \to P_{n+1,j}$ have dense image, then the natural map of \eqref{eq:TensorProductMap} (with $Z=X$) is an isomorphism. 
\end{Prop}
\begin{proof}
It suffices to prove this after basechanging to $Y'$, so we will assume $Y=Y'=\spd(A,A^+)$. We can then check the statement of the proposition on $\spd(B,B^+)$ points for $(B,B^+) \in \spd(A,A^+)_{\mathrm{v}}$, and we can as in the proof of Lemma \ref{Lem:Kunneth1} assume that $(B,B^+)=(A,A^+)$ (using Lemma \ref{Lem:DenseImageTensor}). The statement now follows from Proposition \ref{Prop:Kunneth3}. 
\end{proof}

\subsubsection{} Let $Y$ be a small v-stack and let $f:X \to Y$ be a morphism of small v-stacks. We will write $\mathcal{O}^{+}_{X/Y}=(f_\ast \mathcal{O}_X^+)$ and $\mathcal{O}^{\mathrm{bdd}}_{X/Y}=\mathcal{O}^+_{X/Y}[\tfrac{1}{p}]$.
\begin{Lem} \label{Lem:BoundedFunctions}
    Let $Y=\spd(R,R^+)$ for a fiercely v-complete Huber pair $(R,R^+)$ over $(\qp,\zp)$. Let $X=\varinjlim_n U_n$, where $U_n \to \spd(R,R^+)$ is given by $\spd(P_n,P_n^{+}) \to \spd(R,R^+)$ for some strongly smooth $(R,R^+) \to (P_n,P_n^+)$. For $X^\mathrm{rig}:=\varinjlim_n \spa(P_n, P_n^+)$, there are isomorphisms 
    \begin{align}
        H^0_{\cond}(Y,\mathcal{O}^{+}_{X/Y}) \xrightarrow{\sim} \ul{H^0(X^{\mathrm{rig}}, \mathcal{O}_X^+)} \\
        H^0_{\cond}(Y,\mathcal{O}^{\mathrm{bdd}}_{X/Y}) \xrightarrow{\sim} \ul{H^0(X^{\mathrm{rig}}, \mathcal{O}_X^+)[\tfrac{1}{p}]}.
    \end{align}
\end{Lem}
\begin{proof}
    This is an immediate consequence of the v-completeness of $(P_n, P_n^+)$, see Lemma \ref{Lem:SmoothOverFierce}, and the fact  (which follows as in the proof of Lemma \ref{Lem:SheafFrechet}) that
    \begin{align}
         H^0_{\cond}(Y,\mathcal{O}^{+}_{X/Y}) = \varprojlim_n  H^0_{\cond}(Y,\mathcal{O}^{+}_{U_n/Y}) \\
         H^0(X^{\mathrm{rig}}, \mathcal{O}_X^+) = \varprojlim_n H^0(\spa(P_n, P_n^+), \mathcal{O}_X^+).
    \end{align}
\end{proof}

\section{Families of \texorpdfstring{$p$}{p}-divisible rigid analytic groups} \label{Sec:FamiliesPDiv} In this section we discuss the theory of families of $p$-divisible rigid analytic groups. In \S \ref{Sub:RigidPDiv}, we discuss families of $p$-divisible rigid analytic groups over sousperfectoid adic spaces and rigid spaces over non-archimedean extensions of $\qp$. When the base is the adic spectrum of a non-archimedean extension of $\qp$, we discuss the classification of $p$-divisible rigid analytic groups in terms of Hodge--Tate triples, following Fargues \cite{FarguesI}. In \S \ref{Sub:RigidPDivII}, we define $p$-divisible v-groups over general small v-stacks and construct a functor from Hodge--Tate triples to $p$-divisible v-groups in this generality. In \S \ref{Sub:FiniteHeight}, we introduce the category of finite height $p$-divisible v-groups and relate them to locally analytic character data. 

\subsection{Families of $p$-divisible rigid analytic groups} \label{Sub:RigidPDiv}
In this section, we discuss families of $p$-divisible rigid analytic groups, following Fargues \cites{FarguesI,FarguesII}.

\subsubsection{} Let $Y$ be a sousperfectoid adic space over $\spa \qp$ or a rigid space over a non-archimedean extension of $\qp$. We will consider relative commutative groups $H \to Y$ in the category of smooth adic spaces over $Y$, see \cite[Section IV.4.1]{FarguesScholze}. These have a Lie algebra $\operatorname{Lie} H$, see \cite[p. 10]{HeuerXu}, which is functorial in $H$. To define a logarithm map, we need to restrict to groups which are $p$-divisible in a precise sense, see \cite[Definition 4.1]{FarguesII}. We will present this definition following \cite[Section 3.2]{HeuerXu}.

\subsubsection{} \label{subsub:HeuerXu}  Recall from \cite[Definition 3.2.3]{HeuerXu} that a subspace $T \subset H$ is called \emph{topologically $p$-torsion} if for any open subspace $U \subset H$ containing the image of the unit section, and any quasicompact open subspace $T_0 \subset T$ there is an $n \in \mathbb{Z}_{\ge 1}$ such that $[p^n](T_0) \subset U$, where $[p^n]:H \to H$ is the multiplication by $p^n$ map. 

In the rest of \S \ref{subsub:HeuerXu}, we explain the statement of  \cite[Proposition 3.2.4]{HeuerXu}: There is an open $Y$-subgroup $H \langle p^{\infty} \rangle \subset H$ which is topologically $p$-torsion and is maximal among topologically $p$-torsion open $Y$-subgroups of $H$.\footnote{Note that \cite[Proposition 3.2.4]{HeuerXu} writes $\widehat{H}$ for what we call $H \langle p^{\infty} \rangle$; we instead follow \cite[Definition 2.6]{Heuer}.} Moreover, there is a functorial map of $Y$-groups
\begin{align}
    \log_{H \langle p^{\infty} \rangle}:H \langle p^{\infty} \rangle \to \operatorname{Lie} H,
\end{align}
which induces the identity map on Lie algebras. This fits in an exact sequence
\begin{align}
    0 \to H[p^{\infty}] \to H \langle p^{\infty} \rangle \to \operatorname{Lie} H,
\end{align}
where $H[p^{\infty}]=\bigcup_{n} H[p^n]$. If $[p]:H \to H$ is \'etale, then so is $\log_{H \langle p^{\infty} \rangle}$, and if $[p]:H \to H$ is surjective, then so is $\log_{H \langle p^{\infty} \rangle}$. Finally, the morphism of v-sheaves in groups $H \langle p^{\infty} \rangle^{\lozenge} \to H^{\lozenge}$ over $Y^{\lozenge}$ can be identified with the evaluation at $1$ map $\Hom_{Y^{\lozenge}}(\ul{\zp}_{Y}, H^{\lozenge}) \to H^{\lozenge}$. In particular, this gives $H \langle p^{\infty} \rangle^{\lozenge}$ the structure of a $\ul{\zp}$-module. 

\subsubsection{} Let $Y$ be a sousperfectoid adic space over $\spa \qp$ or a rigid space over a non-archimedean extension of $\qp$ and let $H \to Y$ be a commutative smooth relative group in adic spaces. Recall that $H \to Y$ is a \emph{$p$-divisible rigid analytic group} if $H \langle p^{\infty} \rangle=H$ and if $[p]$ is finite \'etale surjective, see \cite[Definition 4.1]{FarguesII}. Note that \cite{Gerth} uses the terminology analytic $p$-divisible group. 

\begin{Eg} \label{Eg:RigidGM}
Let $Y=\spa \qp$ and consider the rigid analytic fiber $\gmhateta$ over $Y$ of the formal multiplicative group $\gmhat$. This is a $p$-divisible rigid analytic group (e.g. by Lemma \ref{Lem:RigidFiberFormalGroup}). The logarith map $\operatorname{log} \colon \gmhateta \to \mathbb{G}_{a}$ is defined by the usual power series expansion 
\[ \log(1+t)=\sum_{n \geq 1}(-1)^{n+1}\frac{t^n}{n}.\] 
For $\varpi=2p$, the map $\operatorname{log}^{\lozenge}$ is an isomorphism from $\mathbb{B}(1, |\varpi|)$, the ball of radius $|\varpi|$ around $1$ in $\widehat{\mathbb{G}}_{m, \eta}^\lozenge$, to $\mathbb{B}(0, |\varpi|)$, the ball of radius $|\varpi|$ around $0$ in $\mathbb{G}_{a}^\lozenge$, with inverse defined by the usual exponential series 
\[ 
\exp(t)=\sum_{n \geq 0} \frac{1}{n!} t^n.
\]
\end{Eg}
\subsubsection{} Recall that a $p$-divisible group over a $p$-adic formal scheme $X$ (in the sense of \cite[Section 2.2]{ScholzeWeinstein}) is a relative commutative group functor $\mathcal{H} \to X$  such that $[p]:\mathcal{H} \to \mathcal{H}$ is an epimorphism in the fpqc topology, such that $\varinjlim_n \mathcal{H}[p^n] \to \mathcal{H}$ is an isomorphism, and such that $\mathcal{H}[p^n] \to X$ is representable in finite locally free group schemes for all $n$. Note that $\mathcal{H}$ might not be representable by a formal scheme (although it often is, see \cite[Lemma 3.1.1]{ScholzeWeinstein}). We will write $T_p \mathcal{H}$ for the $p$-adic Tate module $\varprojlim_n \mathcal{H}[p^n]$, where the transition maps are induced by multiplication by $p$. The functor $T_p \mathcal{H} \to X$ is representable in flat affine morphisms over $R$, see \cite[Proposition 3.3.1]{ScholzeWeinstein}.
\begin{Lem} \label{Lem:RigidFiberFormalGroup}
Let $Y=\spa(R,R^+)$ be a sousperfectoid adic space over $\spa \qp$ or a rigid space over a non-archimedean extension of $\qp$, and let $\mathcal{H} \to \spf R^+$ be a $p$-divisible group. If $(R,R^+)$ is uniform, then the rigid generic fiber\footnote{Which is defined because $R^+$ is a ring of definition, since $R$ is uniform.} in the sense of \cite[Section 2.2]{ScholzeWeinstein}, $\mathcal{H}^{\mathrm{ad}}_{\eta}$, is a $p$-divisible rigid analytic group over $Y$. 
\end{Lem}
\begin{proof}

This follows from \cite[Proposition 3.4.2.(ii)]{ScholzeWeinstein} and its proof, as we will now explain. Let $H:=\mathcal{H}^{\mathrm{ad}}_{\eta}$. First of all, it follows from that proof that $H \to Y$ is representable by a relative group in adic spaces. It is moreover established that $H$ comes equipped with an exact sequence (where $\operatorname{Lie} H = \operatorname{Lie} \mathcal{H} \otimes_{R^+} \mathbb{G}_a$)
\begin{align}
    0 \to H[p^\infty] \to H \xrightarrow{\log_H} \operatorname{Lie} H.
\end{align}
Furthermore, there is an open subset $U \subset H$ containing the identity section such that the restriction of $\log_{H}$ to $U$ induces an isomorphism between $U$ and $\log_{H}(U)$. On top of that, $\log_{H}(U)$ is an open subset of $\operatorname{Lie} H$ containing the identity section, and it follows that $U$ is smooth over $Y$. Finally it is shown in the proof of \cite[Proposition 3.4.2.(ii)]{ScholzeWeinstein} that $H$ is covered by the inverse images under $[p^n]$ of $U$, and that $[p]:H \to H$ is relatively representable and finite locally free. \smallskip 

It remains to show that $[p]:H \to H$ is finite \'etale and surjective. By the discussion above, for an affinoid open $W=\spa(R,R^+) \subset H$, the inverse image $V=[p]^{-1}W = \spa(B,B^+)$ for some finite locally free $R$-algebra $B$. Since $[p]$ is clearly a quasi-torsor for $H[p]$, we see that $V \times_{W} V \simeq V \times H[p]$ over $W$. But $V \times_{W} V$ is given by $\spa(D,D^+)$ where $D=B \otimes_{R} B$, and $H[p]$ is finite \'etale over $\spa(R,R^+)$ since $p$ is invertible in $R$. Thus it follows from faithfully flat descent that $B$ is a finite \'etale $R$-algebra, showing that $[p]$ is finite \'etale and surjective.
\end{proof}

\subsubsection{} Let $\pi:H \to Y$ be a $p$-divisible rigid analytic group as above. Then the $p^n$-torsion $H[p^n]$ of $H$ is a locally free of finite rank sheaf of $\mathbb{Z}/p^n \mathbb{Z}$-modules on $Y_{\et}$. The multiplication by $p$ map is surjective and thus defines a surjection $H[p^{n+1}] \to H[p^{n}]$ with kernel $H[p]$. Using this observation, we define a $\zp$-local system on $Y$ by $T_p H:=\varprojlim_n H[p^{n+1}]$, the inverse limit being taken in $Y_{\mv}^{\lozenge}$. 

\subsubsection{Classification I} \label{subsub:ClassificationConstruction} Let $Y=\spa(R,R^+)$ be a sousperfectoid adic space over $\spa \qp$ or a rigid space over a non-archimedean extension of $\qp$. 
\begin{Def} \label{Def:EtaleHodgeTateTriple}
We define the category $\LAE(Y)$ of \emph{\'etale Hodge--Tate triples} to be the category of triples $(T,W,\alpha)$, where $T$ is a $\zp$-local system on $Y^{\lozenge}_{\mv}$, where $W$ is an \'etale vector bundle on $Y$, and where $\alpha:W^\lozenge \to T(-1) \otimes_{\ul{\zp}} \mathcal{O}_{Y^\lozenge}$ is a morphism of locally free $\mathcal{O}_{Y^\lozenge}$-modules on $Y^{\lozenge}_{\mv}$. Associated to a triple $\mathcal{L}=(T, W, \alpha) \in \LAE(Y)$ we consider the v-sheaf $\mathcal{H}_{\mathcal{L}}$ of commutative groups over $Y^\lozenge$ given by the fiber product
\begin{equation} \label{Eq:FiberProduct}
    \begin{tikzcd}
        H_{\mathcal{L}} \arrow{d} \arrow{r} \arrow[dr, phantom, "\scalebox{1.5}{$\lrcorner$}" , very near start, color=black] &  T(-1) \otimes_{\ul{\zp}} \widehat{\mathbb{G}}_{m, \eta,Y}^{\lozenge} \arrow{d}{1 \otimes \log^{\lozenge}}\\
        W^{\lozenge} \arrow{r}{\alpha} & T(-1) \otimes_{\ul{\zp}} \mathbb{G}_{a,Y}^{\lozenge}.
    \end{tikzcd}
\end{equation}
\end{Def}
\begin{Prop} \label{Prop:Representable}
Let $\mathcal{L}=(T,W,\alpha)$ and $H_{\mathcal{L}}$ be as above. If $(R,R^+)$ is fiercely v-complete, then there is a unique $p$-divisible rigid analytic group $H_{\mathcal{L}}^{\mathrm{rig}} \to Y$ such that $H_{\mathcal{L}}^{\mathrm{rig},\lozenge}$ is isomorphic to $H_{\mathcal{L}}$. 
\end{Prop}
In the proof of Proposition \ref{Prop:Representable}, we will reduce to the case that the vector bundle $W$ is trivial.
\begin{Lem} \label{Lem:GeometryOfHI}
Let $\mathcal{L}=(T,W,\alpha)$ and $H=H_{\mathcal{L}}$ be as above. If $(R,R^+)$ is fiercely v-complete and $W$ is free, then there are open subgroups $H_n \subset H$ indexed by $n \in \mathbb{Z}_{\ge 1}$ such that $H=\bigcup_n H_n$ and such that $H_n=\spd(A_n,A_n^+)$ with $(R,R^+) \to (A_n,A_n^+)$ strongly smooth.
\end{Lem}
\begin{proof}
Let $\varpi=2p$ and for $n \ge 0$ let 
\begin{align}
    \mathbb{B}(1, |\varpi|^{1/p^n}) \subset \widehat{\mathbb{G}}_{m, \eta}
\end{align}
denote the open subset defined by the inequality $|T|^{p^i} \leq |\varpi|$, and let
\begin{align}
   \mathbb{B}(0, |\frac{\varpi}{p^n}|) \subset \mathbb{G}_{a}
\end{align}
be the closed unit disk of radius $|\frac{\varpi}{p^n}|$. Consider the commutative diagram 
\begin{equation}\label{eq.first-fiber-diagram-functions-structure-proof}\begin{tikzcd}
	{\mathbb{B}(1,|\varpi|^{1/p^n})^{\lozenge} \otimes_{\ul{\zp}} T(-1)} & {\mathbb{B}(1,|\varpi|)^{\lozenge} \otimes_{\ul{\zp}} T(-1)} \\
	{\mathbb{B}(0, |\frac{\varpi}{p^n}|)^{\lozenge} \otimes_{\ul{\zp}} T(-1) } & {\mathbb{B}(0,|\varpi|)^{\lozenge} \otimes_{\ul{\zp}} T(-1)}. \\
	{}
	\arrow["{z \mapsto z^{p^n}}", from=1-1, to=1-2]
	\arrow["{\log \otimes 1}"{description}, from=1-1, to=2-1]
	\arrow["{\mathrm{log}\otimes 1}"{description}, from=1-2, to=2-2]
	\arrow["{t \mapsto p^nt}"', from=2-1, to=2-2]
\end{tikzcd}\end{equation}
The right vertical arrow and bottom horizontal arrow are isomorphisms (see Example \ref{Eg:RigidGM}). Since the top horizontal arrow is finite \'{e}tale, the left vertical arrow is also finite \'etale. We now consider an affinoid open neighborhood of $U \subseteq W$ of $U$ such that $U^{\lozenge}$ is contained in $\mathbb{B}(0,|\varpi|)^{\lozenge} \otimes_{\ul{\zp}} T(-1)$. We define $H_0^\lozenge$ to be the preimage of $U$ in $\mathbb{B}(1,|\varpi|)^{\lozenge}$ (so that $H_0^\lozenge$ is isomorphic by $\log \otimes 1$ to $U$). We then define $H_n^\lozenge$ to be the preimage of $H_0^\lozenge$ in ${\mathbb{B}(1,|\varpi|^{1/p^n})^{\lozenge} \otimes_{\ul{\zp}} T(-1)}$. It is a finite \'{e}tale cover of $H_0^\lozenge$ sitting in the subdiagram of \eqref{eq.first-fiber-diagram-functions-structure-proof}
\begin{equation}
        \begin{tikzcd}
        H_n^{\lozenge} \arrow{r} \arrow{d}{\log} & H_0^{\lozenge} \arrow{d}{\log} \\
        (\frac{1}{p^n} U)^{\lozenge} \arrow{r}{t \mapsto p^n t} & U^{\lozenge}.
\end{tikzcd}
\end{equation}
We thus have $H^{\lozenge}=\bigcup_n H_n^{\lozenge}$, and each $H_n^{\lozenge}$ is finite \'{e}tale over $U^{\lozenge}$ via the map $(t \mapsto p^n t) \circ \log$. By the equivalence of the finite \'{e}tale sites of $U$ and $U^{\lozenge}$, see \cite[Lemma 15.6]{EtCohDiam}, each $H_n^{\lozenge}$ is of the form $\spd(A_n, A_n^+)$ where $A_n$ is a finite \'{e}tale $A=\mathcal{O}(U)$-algebra. Now, since $W$ is trivializable, we may choose $U$ to be an affinoid ball so that $(R,R^+) \to (\mathcal{O}(U), \mathcal{O}^+(U)) \to (A_n,A_n^+)$ exhibits $(R,R^+) \to (A_n,A_n^+)$ as strongly smooth. 
\end{proof}

\begin{proof}[Proof of Proposition \ref{Prop:Representable}]
It follows from Lemma \ref{Lem:GeometryOfHI} and Lemma \ref{Lem:SmoothOverFierce} that, analytically locally on $Y$, we can write $H_{\mathcal{L}} \to Y^{\lozenge}$ as $H^{\mathrm{rig}}_{\mathcal{L}} \to Y$ with $H^{\mathrm{rig}}_{\mathcal{L}} \to Y$ locally fiercely v-complete. By the fully faithfulness proved in Lemma \ref{Lem:FiercelyFullyFaithful}, this also holds globally on $Y$ since the gluing datum for the v-sheaf $H_{\mathcal{L}}$ induces a gluing datum of adic spaces. By the same fully faithfulness and the compatibility of $X \mapsto X^{\lozenge}$ with products, we see that $H^{\mathrm{rig}}_{\mathcal{L}} \to Y$ has the structure of a relative smooth commutative group in adic spaces, and it suffices to show that it is a $p$-divisible rigid analytic group. \smallskip 

For this, we note that $H_{\mathcal{L}}=\Hom_{Y^{\lozenge}}(\ul{\zp}_{Y^{\lozenge}}, H_{\mathcal{L}})$ since this is true for the other three terms in the fiber product diagram. Moreover, the natural map $[p]:H_{\mathcal{L}} \to H_{\mathcal{L}}$ is (representable in) surjective finite \'etale morphisms since this is true for the other three terms in the diagram. By Lemma \ref{Lem:FiercelyFullyFaithful} and \cite[Lemma 15.6]{EtCohDiam}, this implies that $[p]:H^{\mathrm{rig}}_{\mathcal{L}} \to H^{\mathrm{rig}}_{\mathcal{L}}$ is finite \'etale surjective. Thus $H^{\mathrm{rig}}_{\mathcal{L}} \to Y$ is a $p$-divisible rigid analytic group.
\end{proof}
The construction $\mathcal{L} \to H_{\mathcal{L}}^{\mathrm{rig}}$ defines a functor from $\LAE(Y)$ to the category of $p$-divisible rigid analytic groups over $Y$.

\subsubsection{Classification II} \label{subsub:Classification} A result of Fargues, see \cite[Théorème 0.1]{FarguesI}, tells us that the functor $\mathcal{L} \mapsto H_{\mathcal{L}}^{\mathrm{rig}}$ is an equivalence of categories when $Y=\spa(K, \mathcal{O}_K)$ with $K$ a field. In the forthcoming PhD thesis of Gerth \cite{GerthThesis}, this will be generalized to work over an arbitrary perfectoid $Y$. 

\subsection{Families over small v-stacks}\label{Sub:RigidPDivII}  Let $Y$ be a small v-stack (see \S \ref{Sub:Prelim}).
\begin{Def} \label{Def:PdivisibleGroup}
  We say that a v-sheaf of commutative groups $H \to Y$ is a \emph{$p$-divisible v-group} if for all affinoid perfectoid spaces $Y'=\spd(A,A^+) \to Y$, the base change $H_{Y'} \to Y'$ is (the diamond associated to) a $p$-divisible rigid analytic group over $\spa(A,A^+)$.
\end{Def}
\subsubsection{} Let $Y$ be a small v-stack and let $H \to Y$ be a $p$-divisible v-group. 
\begin{Lem} \label{Lem:LieLog}
There is a locally free sheaf of $\mathcal{O}_Y$-modules $\operatorname{Lie} H$ and a short exact sequence
\begin{align}
    0 \to H[p^{\infty}] \to H \xrightarrow{\operatorname{Log}_{H}} \operatorname{Lie} H \to 0
\end{align}
of v-sheaves in abelian groups on $Y_{\mv}$. The formation of $(\operatorname{Lie} H, \operatorname{Log}_{H})$ is functorial in morphisms of $p$-divisible v-groups over $Y$ and commutes with pullback along $Y' \to Y$.
\end{Lem}
\begin{proof}
By v-descent, it suffices to prove this when $Y=\spd(A,A^+)$ with $(A,A^+)$ affinoid perfectoid. But then it is established in \S \ref{subsub:HeuerXu}.
\end{proof}
\subsubsection{} Let $Y=\spd(R,R^+)$, where $(R,R^+)$ is a fiercely v-complete Huber pair over $(\qp,\zp)$. 
\begin{Lem} \label{Lem:RepresentableFierce}
  If $H \to \spd(R,R^+)$ is a $p$-divisible v-group such that $\operatorname{Lie} H$ comes from an analytic vector bundle over $\spa(R,R^+)$, then there is a unique commutative group $H^{\mathrm{rig}} \to \spa(R,R^+)$ in smooth analytic adic spaces over $\spa(R,R^+)$ whose associated diamond is $H$. 
\end{Lem}
\begin{proof}
This follows from Proposition \ref{Prop:Representable} and its proof, using Lemma \ref{Lem:LieLog}. 
\end{proof}
This gives us a theory of $p$-divisible rigid analytic groups over arbitrary fiercely v-complete Huber pairs over $(\qp,\zp)$.
\begin{Lem} \label{Lem:vDescentDiamantine}
Let $H \to Y$ be a v-sheaf of commutative groups. If there is a v-cover $Y' \to Y$ such that $H_{Y'} \to Y'$ is (the diamond associated to) a $p$-divisible rigid analytic group, then $H \to Y$ is a $p$-divisible v-group.
\end{Lem}
\begin{proof}
The property that $\Hom_{Y}(\ul{\zp}_{Y}, H) \to H$ is an isomorphism can be checked v-locally, because it is a map of v-sheaves. By \cite[Proposition 10.11.(iii)]{EtCohDiam}, the same is true for the property that $[p]$ is finite \'etale surjective. The existence of the logarithm as in Lemma \ref{Lem:LieLog} can also be proved v-locally, and the representability over perfectoid spaces follows from the existence of the logarithm as in the proof of Lemma \ref{Lem:RepresentableFierce}.
\end{proof}
If $Y$ is a sousperfectoid adic space over $\spa \qp$ or a rigid space over a non-archimedean extension of $\qp$, and $H \to Y$ is a $p$-divisible rigid analytic group, then $H^{\lozenge} \to Y^{\lozenge}$ is a $p$-divisible v-group; this follows as in Lemma \ref{Lem:vDescentDiamantine}. We note however that $H \mapsto H^{\lozenge}$ is generally not essentially surjective.

\subsubsection{Classification} \label{subsub:ClassificationII} Let $Y$ be a small v-stack.
\begin{Def} \label{Def:HodgeTateTripleVstacks}
    Let $\LA(Y)$ be the category of Hodge--Tate triples consisting of triples $(T,W,\alpha)$, where $T$ is a $\zp$-local system on $Y_{v}$, where $W$ is a locally free sheaf of $\mathcal{O}_Y$-modules on $Y_{\mathrm{v}}$, and where $\alpha:W \to T(-1) \otimes_{\ul{\zp}} \mathcal{O}_Y$ is a morphism. For $\mathcal{L}=(T,W,\alpha) \in \LA(Y)$, we write $H_{\mathcal{L}}$ for the fiber product
\begin{equation} \label{Eq:FiberProductII}
    \begin{tikzcd}
        H_{\mathcal{L}} \arrow{d} \arrow{r} \arrow[dr, phantom, "\scalebox{1.5}{$\lrcorner$}" , very near start, color=black] &  T(-1) \otimes_{\ul{\zp}} \widehat{\mathbb{G}}_{m, \eta,Y}^{\lozenge} \arrow{d}{1 \otimes \log^{\lozenge}}\\
        W \arrow{r}{\alpha} & T(-1) \otimes_{\ul{\zp}} \mathbb{G}_{a,Y}^{\lozenge}.
    \end{tikzcd}
\end{equation}
\end{Def}
\begin{Lem} \label{Lem:CompatibleBaseChange}
    The functor $\mathcal{L} \mapsto H_{\mathcal{L}}$ commutes with basechange along $Y' \to Y$.
\end{Lem}
\begin{proof}
    This follows from the fact that tensor products of sheaves of modules and fiber products of sheaves both commute with base change. 
\end{proof}
\begin{Lem} \label{Lem:LocSpat}
Let $Y$ be a small v-stack. If $\mathcal{L} \in \LA(Y)$, then the v-sheaf in commutative groups $H_{\mathcal{L}} \to Y$ is a $p$-divisible v-group. 
\end{Lem}
\begin{proof}
By Lemma \ref{Lem:CompatibleBaseChange}, it suffices to prove representability when $Y=\spd(A,A^+)$ with $(A,A^+) \in \perfd$. In this case, the sheaf $W$ of $\mathcal{O}_Y$-modules is free \'etale locally on $Y$ by \cite[Theorem 3.5.8]{KedlayaLiuII}, and the result is Proposition \ref{Prop:Representable}.
\end{proof}

\subsubsection{} If $Y$ is a locally spatial diamond in the sense of \cite[Definition 11.17]{EtCohDiam}, then we define the full subcategory $\LAE(Y) \subset \LA(Y)$ of \'etale Hodge--Tate triples to be those triples $(T,W,\alpha)$ with $W$ an \'etale vector bundle (using the \'etale site of \cite[Definition 14.1]{EtCohDiam}). If $X$ is a sousperfectoid adic space over $\spa \qp$ or a rigid space over a non-archimedean extension of $\mathbb{Q}_p$, then, by \cite[Lemma 15.6]{EtCohDiam}, $\LAE(X^{\lozenge})=\LAE(X)$, where the latter was defined in Definition \ref{Def:EtaleHodgeTateTriple}. 
\begin{Conj} \label{Conj:Gerth}
    Let $Y$ be a small v-stack. The functor $\mathcal{L}=(T,W,\alpha) \to H_{\mathcal{L}}$ is an equivalence of categories between $\LA(Y)$ and the category of $p$-divisible $v$-groups over $Y$.
\end{Conj}

Note that both categories are the evaluation of a v-stack on $Y$: For $\LA(Y)$ this is clear from the definition and for $p$-divisible v-groups this is Lemma \ref{Lem:vDescentDiamantine}. Thus it suffices to prove Conjecture \ref{Conj:Gerth} for $Y$ a (strictly totally disconnected) perfectoid space; this will be done in the forthcoming PhD thesis of Gerth \cite{GerthThesis}.

\subsection{\texorpdfstring{$p$}{p}-divisible v-groups of finite height} \label{Sub:FiniteHeight} Let $Y$ be a small v-stack. 
\begin{Def} \label{Def:FiniteHeight}
    We say that a triple $\mathcal{L}=(T,W,\alpha) \in \LA(Y)$ has \emph{finite height} if $\alpha$ is injective. Note that this condition is stable under pullback along $Y' \to Y$. We define $\LAD(Y) \subset \LA(Y)$ to be the full subcategory of triples $\mathcal{L}=(T,W,\alpha)$ of finite height. We say that a $p$-divisible v-group $H \to Y$ has \emph{finite height} if $H=H_{\mathcal{L}}$ with $\mathcal{L}$ of finite height. 
\end{Def}
If $H$ arises as the generic fiber of a $p$-divisible group, then it is of finite height, see Proposition \ref{Prop.TateModule}. As a special case, this includes the generic fibers of $p$-divisible (equivalently, finite height) formal groups; this explains our choice of terminology. 
\begin{Lem} \label{Lem:LocalDirectSummand}
Let $Y$ be a small v-stack and $\mathcal{L}=(T,W,\alpha) \in \LA(Y)$. The triple $\mathcal{L}$ has finite height if and only if $\alpha$ realizes $W$ as a local direct summand of $T(-1) \otimes_{\ul{\zp}} \mathcal{O}_Y$. In this case, $\mathrm{coker}(\alpha)$ is a locally free $\mathcal{O}_Y$-module. 
\end{Lem}
\begin{proof}
The if direction is trivial. Thus, let us assume $\alpha$ is injective. We may assume that $Y=\Spd(A,A^+)$ is perfectoid and that $W$ has constant rank over $Y$. Via taking global sections, the category of locally free of finite rank $\mathcal{O}_Y$-modules is equivalent to the category of finite projective $A$ modules, see \cite[Theorem 3.5.8]{KedlayaLiuII}. Let $P=H^0(Y,W)$ and $Q=H^0(Y, T(-1) \otimes_{\ul{\zp}} \mathcal{O}_Y)$ be the associated finite projective $A$ modules, and let $f: P \hookrightarrow Q$ be the associated map. By testing on geometric points, we find that for any $x \in \Spa(A,A^+)$, the map $f \otimes_A \kappa(x): P\otimes_A \kappa(x) \rightarrow Q \otimes_A \kappa(x)$ is also injective. Thus $\coker(f) \otimes_A \kappa(x)$ has constant rank and, by \cite[Proposition 2.8]{KedlayaLiu}, $\coker(f)$ is finite projective. Thus the short exact sequence
\[ 0 \rightarrow P \rightarrow Q \rightarrow \coker f \rightarrow 0\]
is split, and any associated direct sum decomposition $Q=P \oplus \coker f$ induces the desired direct sum decomposition of sheaves
    \[ T(-1) \otimes_{\ul{\zp}} \mathcal{O}_Y = Q \otimes_A \mathcal{O}_Y = (P \otimes_A \mathcal{O}_Y) \oplus (\mathrm{coker}(f) \otimes_A \mathcal{O}_Y) = W \oplus W'. \] 
\end{proof}

\subsubsection{} The following proposition shows that the adic generic fiber of a (classical) $p$-divisible group is of finite height. Recall that given a $p$-divisible group $\mathcal{H}$ over a $p$-adic formal scheme $X$, there is a Serre-dual $p$-divisible group $\mathcal{H}^{\vee}$ and a Hodge--Tate map $T_p \mathcal{H}^{\vee} \to \omega_{\mathcal{H}}$, see for example Section \ref{subsub:HodgeTate} below.
\begin{Prop}\label{Prop.TateModule}
    Let $(R,R^+)/(\mathbb{Q}_p, \zp)$ be a fiercely v-complete Huber pair. If $\mathcal{H}$ is a $p$-divisible group over $\spf R^+$, then there is a canonical isomorphism 
    \[ \mathcal{H}_\eta^{\lozenge} \xrightarrow{\sim} H_{\mathcal{L}} \textrm{ for } \mathcal{L}={(T_p \mathcal{H}_\eta, \Lie \mathcal{H} \otimes_{R^+} \mathcal{O}_{Y},  \Lie \mathcal{H} \otimes_{R^+} \mathcal{O}_{Y} \hookrightarrow T_p \mathcal{H}_\eta(-1) \otimes_{\ul{\zp}} \mathcal{O}_{Y})}.\]
    Here $\alpha:\Lie \mathcal{H} \otimes_{R^+} \mathcal{O}_{Y} \hookrightarrow T_p \mathcal{H}_\eta(-1) \otimes_{\ul{\zp}} \mathcal{O}_{Y}$ is the dual Hodge--Tate map. 
\end{Prop}
\begin{proof}
    We follow the proof of \cite[Theorem 5.2.1]{ScholzeWeinstein}. Let $T_p(H)$ be the $p$-adic Tate module of $H=H_{\eta}$, which is a $\zp$-local system on $\spa(R,R^+)$ that we can identify with $\left(T_p \mathcal{H}^{\vee} \right)_{\eta}^{\mathrm{ad}}$. Let $H'= T_p(H) \otimes_{\ul{\zp}_{R}} \widehat{\mathbb{G}}_{m, \eta, R}$ and note that there is a natural map $H \to H'$ of $p$-divisible rigid analytic groups over $\spa(R,R^+)$ which induces an isomorphism $T_p(H) \xrightarrow{\sim} T_p(H')$ on $p$-adic Tate modules. The induced map $\Lie \mathcal{H}[\tfrac{1}{p}]=\operatorname{Lie} H \to \operatorname{Lie} H' = T_p(H)(-1) \otimes_{\ul{\zp}} \mathbb{A}^1_{\spa(R,R^+)}$ is moreover identified with the Hodge--Tate map. Thus we get a commutative diagram
    \begin{equation}
    \begin{tikzcd}
        H \arrow{r} \arrow{d} & H' \arrow{d} \\
        \operatorname{Lie} \mathcal{H} \otimes_{R^+} \mathbb{A}^1_{\spa(R,R^+)} \arrow{r} & T_p \mathcal{H}_{\eta}(-1) \otimes_{\ul{\zp}} \mathbb{A}^1_{\spa(R,R^+)}
       \end{tikzcd} 
    \end{equation}
    and the proposition comes down to showing that it is in fact Cartesian. This can be proved as in the proof of \cite[Theorem 5.2.1]{ScholzeWeinstein}. 
\end{proof}

\subsubsection{} \label{subsub:CharacterDatum} Let $Y$ be a small v-stack. 
\begin{Def} \label{Def:CharacterDatum}
    A locally analytic character datum over $Y$ is a triple $\mathcal{X}=(\Lambda,V,\gamma)$,  where $\Lambda$ is a $\zp$-local system over $Y$, where $V$ is a locally free of finite rank sheaf of $\mathcal{O}_Y$-modules and $\gamma$ is a surjection $\gamma:\Lambda \otimes_{\ul{\zp}} \mathcal{O}_Y \to V$. We will write $\CD(Y)$ for the category of locally analytic character data over $Y$. For $\mathcal{X}=(\Lambda,V,\gamma)$ we define the character v-group to be the fiber product
    \begin{equation}
        \begin{tikzcd}
           H_{\mathcal{X}} \arrow{r} \arrow{d} & \Hom_{Y}(\Lambda_{Y}, \widehat{\mathbb{G}}_{m,\eta,Y}^{\lozenge}) \arrow{d}{\operatorname{Log}} \\
            \Hom_{\mathcal{O}_{Y}}(V, \mathbb{G}_{a,Y}^{\lozenge}) \arrow{r}{\gamma^{\ast}} & \Hom_{Y}(\Lambda_{Y}, \mathbb{G}_{a,Y}^{\lozenge}).
        \end{tikzcd}
    \end{equation}
\end{Def}
There is a contravariant functor $\CD(Y) \to \LAD(Y)$ given by
\begin{align}
    \mathcal{X}=(\Lambda, V, \gamma) \mapsto \mathcal{X}^{\vee}=(\Lambda^{\ast}(1), V^{\ast}, \gamma^{\ast}), 
\end{align}
which is an equivalence of categories with inverse
\begin{align}
    \mathcal{L}=(T, W, \alpha) \mapsto \mathcal{L}^{\vee}=(T^{\ast}(1), W^{\ast}, \alpha^{\ast}).
\end{align}
\begin{Lem}
Let $Y$ be a small v-stack and let $\mathcal{X}$ be a locally analytic character datum over $Y$. There is a functorial isomorphism
    \begin{align}
        H_{\mathcal{X}} \to H_{\mathcal{X}^{\vee}}.
    \end{align}
    In particular, $H_{\mathcal{X}}$ is a $p$-divisible v-group, and every finite height $p$-divisible $v$-group is of the form $H_{\mathcal{X}}$ for some character datum $\mathcal{X}$. 
\end{Lem}
\begin{proof}
Observe that
\begin{align}
    \Hom_{Y}(\Lambda_{Y}, \widehat{\mathbb{G}}_{m,\eta,Y}^{\lozenge}) &= \Lambda^{\ast} \otimes_{\ul{\zp}} \widehat{\mathbb{G}}_{m,\eta,Y}^{\lozenge} \\
        \Hom_{Y}(\Lambda_{Y}, \mathbb{G}_{a,Y}^{\lozenge}) &= \Lambda^{\ast} \otimes_{\ul{\zp}} \mathbb{G}_{a,Y}^{\lozenge} \\
        \Hom_{Y}(V, \mathbb{G}_{a,Y}^{\lozenge})&= V^{\ast}. 
\end{align}
The first part of the lemma is now a direct consequence of the definition of $\mathcal{X}^{\vee}$, together with Definition \ref{Def:CharacterDatum} and Definition \ref{Def:PdivisibleGroup}. The second part of the lemma follows from Lemma \ref{Lem:LocSpat}.
\end{proof}

\subsubsection{Moduli} Fix an integer $h \ge 0$ and an integer $0 \le d \le h$ and let $Y$ be a small v-stack. Then there is a full subcategory $\mathcal{LA}_{h,d}(Y) \subset \LAD(Y)$ consisting of those triples $(T, W, \alpha)$ where $T$ is locally free of rank $h$ and $W$ is locally free of rank $d$. Note that the assignment $Y \mapsto \mathcal{LA}_{h,d}(Y) \subset \LAD(Y)$ is a substack for the v-topology. Consider the natural action of $\ul{\operatorname{GL}_h(\zp)}$ on $\operatorname{Gr}_{h,d}^{\lozenge}$, where $\operatorname{Gr}_{h,d}$ is the (schematic) Grassmannian of $d$-planes in $\qp^{\oplus h}$.
\begin{Lem} \label{Lem:Moduli}
    There is an equivalence of v-stacks on $\perf$
    \begin{align}
        \left[\operatorname{Gr}_{h,d}^{\lozenge}/ \ul{\operatorname{GL}_h(\zp)} \right] \to \mathcal{LA}_{h,d}.
    \end{align}
\end{Lem}
\begin{proof}
It suffices to define the equivalence on affinoid perfectoid test objects $(A,A^+)$, where it is essentially a tautology once we use the fact that v-vector bundles on affinoid perfectoid spaces descend to analytic vector bundles, see \cite[Theorem 3.5.8]{KedlayaLiuII}. 
\end{proof}

\section{Sheaves of functions and distributions} \label{Sec:SheavesOfLaFunctions}

In \S \ref{Sub:Functions}, we show that sheaves of functions on $p$-divisible v-groups over small v-stacks have good functional analytic properties, see Proposition \ref{Prop:GeometricFrechet}. In \S \ref{Sub:FunctionsII}, we introduce spaces of locally analytic functions associated with locally analytic character data, and show that these have good functional analytic properties, see Proposition \ref{Prop:GeometricLB}. 

\subsection{Functions on \texorpdfstring{$p$}{p}-divisible v-groups.} \label{Sub:Functions} Let $Y$ be a small v-stack and let $\mathcal{L}=(T, W, \alpha) \in \LA(Y)$ be a Hodge--Tate triple. Consider the $p$-divisible v-group $H_{\mathcal{L}} \xrightarrow{f} Y$ constructed from $\mathcal{L}$ as in Definition \ref{Def:HodgeTateTripleVstacks} and Lemma \ref{Lem:LocSpat}. Write $\mathcal{O}_{H_{\mathcal{L}}/Y}=f_{\ast} \mathcal{O}_{H_{\mathcal{L}}}$ for the pushforward on v-sites, and let $\mathcal{D}_{H_{\mathcal{L}}/Y}=\Hom_{\mathcal{O}_Y}(\mathcal{O}_{H_{\mathcal{L}}/Y}, \mathcal{O}_Y)$ be its dual. 
\begin{Prop} \label{Prop:GeometricFrechet}
If $Y=\spd(R,R^+)$ for a fiercely v-complete Huber pair $(R,R^+)$ over $(\qp,\zp)$ and $W$ is free, then the following hold: 
\begin{enumerate}
    \item For all morphisms of fiercely v-complete Huber pairs $(R,R^+) \to (B,B^+)$, the natural map (where $Y'=\spd(B,B^+)$)
        \begin{align}
            H^0_{\cond}(Y,  \mathcal{O}_{H_{\mathcal{L}}/Y}) \otimes_{\ul{R}}^{\blacksquare} \ul{B} \to H^0_{\cond}(Y', \mathcal{O}_{H_{\mathcal{L}}/Y})
        \end{align}
        is an isomorphism.

      \item The natural map
    \begin{align}
        H^0_{\cond}(Y,\mathcal{O}_{H_{\mathcal{L}}/Y}) \otimes_{\ul{R}}^{\blacksquare} H^0_{\cond}(Y,\mathcal{O}_{H_{\mathcal{L}}/Y}) \to  H^0_{\cond}(Y,\mathcal{O}_{H_{\mathcal{L} \times \mathcal{L}}/Y})
    \end{align}
    is an isomorphism. 

    \item The natural map
    \begin{align}
        H^0_{\cond}(Y,\mathcal{D}_{H_{\mathcal{L}}/Y}) &\to \Hom_{\ul{R}}(H^0_{\cond}(Y,\mathcal{O}_{H_{\mathcal{L}}/Y}), \ul{R})
    \end{align}
    is an isomorphism.

    \item The natural map
    \begin{align}
        H^0_{\cond}(Y,\mathcal{D}_{H_{\mathcal{L}}/Y}) \otimes^{\blacksquare}_{\ul{R}} H^0_{\cond}(Y,\mathcal{D}_{H_{\mathcal{L}}/Y}) \to  H^0_{\cond}(Y,\mathcal{D}_{H_{\mathcal{L} \times \mathcal{L}}/Y})
    \end{align}
    is an isomorphism.
\end{enumerate}
For $Y$ an arbitrary small v-stack, the following hold: 
\begin{enumerate}
\setcounter{enumi}{4}

    \item The sheaf $\mathcal{O}_{H_{\mathcal{L}}/Y}$ is reflexive.

    \item The sheaf $\mathcal{O}_{H_{\mathcal{L}}/Y}$ is locally strongly countably Fr\'echet of compact type in the sense of Definition \ref{Def:LocallyFrechet}.
        
    \item The natural map
    \begin{align}
        \mathcal{O}_{H_{\mathcal{L}}/Y} \otimes^{\blacksquare}_{\mathcal{O}_Y}\mathcal{O}_{H_{\mathcal{L}}/Y} \to \mathcal{O}_{H_{\mathcal{L} \times \mathcal{L}} /Y}.
    \end{align}
    is an isomorphism.

    \item The natural map
    \begin{align}
        \mathcal{D}_{H_{\mathcal{L}}/Y} \otimes^{\blacksquare}_{\mathcal{O}_Y} \mathcal{D}_{H_{\mathcal{L}}/Y} \to \left(\mathcal{O}_{H_{\mathcal{L}}/Y} \otimes^{\blacksquare}_{\mathcal{O}_Y} \mathcal{O}_{H_{\mathcal{L}}/Y} \right)^{\ast}
    \end{align}
    is an isomorphism.

\end{enumerate}

\end{Prop}

We start by proving a geometric lemma.
\begin{Lem} \label{Lem:GeometryofH}
Let $Y=\spa(R,R^+)$ for a fiercely v-complete Huber pair $(R,R^+)$ over $(\qp,\zp)$. Let $\mathcal{L}=(T, W, \alpha) \in \LAE(Y^{\lozenge})$ and let $H_{\mathcal{L}}^{\mathrm{rig}}=H \to Y$ be the corresponding commutative smooth group in adic spaces (see Lemma \ref{Lem:RepresentableFierce}). If $W$ is free, then there is an increasing union of open affinoid $Y$-subgroups $H_n \subset H$ indexed by $n \in \mathbb{Z}_{\ge 1}$ such that $H=\bigcup_n H_n$ and such that:
    \begin{itemize}
        \item For each $n$, there is a finite \'{e}tale map $H_n \rightarrow U$, where $U$ is a closed ball in $W$. In particular, $\mathcal{O}(H_n)$ is a (topological) direct summand of the countably orthonormalizable Banach $R$-module $\mathcal{O}(U)^{\oplus m_n}$ for some $m_{n}$.

        \item For $n < n'$ the restriction map $\mathcal{O}(H_{n'}) \rightarrow \mathcal{O}(H_n)$ has dense image and the map $\ul{\mathcal{O}(H_{n'})} \rightarrow \ul{\mathcal{O}(H_n)}$ is compact in the sense of Section \ref{subsub:compact}.
   \end{itemize}
\end{Lem}
\begin{proof}
Following the proof and notation of Lemma \ref{Lem:GeometryOfHI}, we will write $H_n=\bigcup_n H_n$ where each $H_n$ is finite \'{e}tale over a ball $U \subset W$. Then $H_n=\spa(A_n,A_n^+)$, where $A_n$ is a finite \'{e}tale $A=\mathcal{O}(U)$-algebra. In particular, $A_n$ is a finite projective $A$-module, so a direct summand of $A^{\oplus m_n}$ for some $m_n$ (see the proof of Lemma \ref{Lem:SheafFrechet}). Now, since $W$ is trivializable, we may choose $U$ to be an affinoid ball, so that $A$ is orthonormalizable and we obtain the claim (2) on the structure of $A_n$. 

For the claims about dense image and compactness, note that the diagram
\[\begin{tikzcd}
	{H_n} & {H_{n+1}} \\
	U & U
	\arrow[hook, from=1-1, to=1-2]
	\arrow["{p^n \cdot \log}"', from=1-1, to=2-1]
	\arrow["{p^{n+1}\cdot \log}", from=1-2, to=2-2]
	\arrow["p"', from=2-1, to=2-2]
\end{tikzcd}\]
realizes $H_n$ as a clopen subset of $H_{n+1} \times_U {pU}$. Thus $\mathcal{O}(H_n)$ is a direct summand of $\mathcal{O}(H_{n+1} \times_U {pU})$ as a topological $R$-module. We find that $\mathcal{O}(H_{n+1}) \rightarrow \mathcal{O}(H_n)$ can be factored as 
\[\begin{tikzcd}
	{\mathcal{O}(H_{n+1})} & {\mathcal{O}(H_{n+1} \times_U pU)} & {\mathcal{O}(H_n)} \\
	{\mathcal{O}(U)^{m}} & {\mathcal{O}(pU)^{m}}
	\arrow[from=1-1, to=1-2]
	\arrow["\oplus", hook, from=1-1, to=2-1]
	\arrow[two heads, from=1-2, to=1-3]
	\arrow["\oplus", hook, from=1-2, to=2-2]
	\arrow[from=2-1, to=2-2]
	\arrow[dashed, printersafe, two heads, from=2-2, to=1-3]
\end{tikzcd}\]
where the middle vertical arrow is obtained by tensoring the left vertical arrow over $\mathcal{O}(U)$ along the restriction map $\mathcal{O}(U) \rightarrow \mathcal{O}(pU)$. 

In particular, to see that $\mathcal{O}(H_{n+1}) \rightarrow \mathcal{O}(H_n)$ has dense image, we claim it suffices to see that $\mathcal{O}(U) \rightarrow \mathcal{O}(pU)$ has dense image. Indeed, choosing an idempotent $e$ projecting from $\mathcal{O}(U)^m$ onto $\mathcal{O}(H_{n+1})$, we see that image of $\mathcal{O}(H_{n+1}) \rightarrow \mathcal{O}(H_n)$ is dense because $\mathcal{O}(H_{n+1} \times_U pU) \rightarrow \mathcal{O}(H_n)$ is a surjection and $\mathcal{O}(U)^m \rightarrow \mathcal{O}(pU)^m$ has dense image thus $\mathcal{O}(H_{n+1}) \rightarrow \mathcal{O}(H_{n+1} \times_U pU)$ also has dense image as it can be identified with $e \mathcal{O}(U)^m \rightarrow  e\mathcal{O}(pU)^m$. Now, we obtain the desired density because $\mathcal{O}(U) \rightarrow \mathcal{O}(pU)$ evidently has dense image: Choosing coordinates, it is the natural inclusion $R\langle T_1, \ldots T_d\rangle \rightarrow R \langle \frac{T_1}{p} \ldots, \frac{T_d}{p} \rangle$, whose image contain the polynomials $R[T_1, \ldots, T_d]$, which are a dense subset. 

Similarly, to see $\mathcal{O}(H_{n+1}) \rightarrow \mathcal{O}(H_n)$ is compact, since it is obtained from $\mathcal{O}(U)^m \rightarrow \mathcal{O}(pU)^m$ by pre-composition with the inclusion of a direct summand and post-composition with projection onto a direct summand, it suffices to show that $\mathcal{O}(U)^m \rightarrow \mathcal{O}(pU)^m$ is compact. Since a finite direct sum of compact maps is compact, it thus suffices to show $\mathcal{O}(U) \rightarrow \mathcal{O}(pU)$ is compact. To prove this, we can moreover reduce to the case that $Y=\spd(\qp,\zp)$ by Lemma \ref{Lem:CompactTensor}, and then the result is well known in classical functional analysis; with the present condensed definitions see, e.g., \cite[Remark 4.9]{RJRC}.
\end{proof}
\begin{Cor} \label{Cor:IsFrechet}
Let $Y=\spa(R,R^+)$ for a fiercely v-complete Huber pair $(R,R^+)$ over $(\qp,\zp)$. Let $\mathcal{L}=(T, W, \alpha) \in \LAE(Y^{\lozenge})$ and let $H=H_{\mathcal{L}}^{\mathrm{rig}} \to Y$ be the corresponding commutative smooth group in adic spaces (see Lemma \ref{Lem:RepresentableFierce}). If $W$ is free, then $\ul{\mathcal{O}(H)}$ is a strongly countably Fr\'echet $\ul{R}$-module of compact type.
\end{Cor}
\begin{proof}[Proof of Proposition \ref{Prop:GeometricFrechet}]
Part (1) is a direct consequence of Lemma \ref{Lem:EvaluationFrechet} and Lemma \ref{Lem:GeometryofH}. Part (2) follows from Proposition \ref{Prop:Kunneth3} and Lemma \ref{Lem:GeometryofH}. Part (3) follows from Lemma \ref{Lem:EvaluationAbstract} and part (1). Part (4) follows from parts (2) and (3), in combination with Corollary \ref{Cor:IsFrechet} and Corollary \ref{Cor:DualTensorCommute}.

For the rest of the proposition, we can reduce to the case that $T$ and $W$ are trivial and that $Y=\spd(A,A^+)$ is affinoid perfectoid. Part (6) is then a direct consequence of Lemma \ref{Lem:SheafFrechet} and Lemma \ref{Lem:GeometryofH}. Part (5) follows from Corollary \ref{Cor:FrechetReflexive} and part (7) from Proposition \ref{Prop:Kunneth2}. Finally, part (8) follows by applying (3) and (4) to affinoid perfectoid $(A,A^+) \to (B,B^+)$.
\end{proof}

\subsubsection{} \label{subsub:Hopf} Let $Y$ be a small v-stack, let $\mathcal{L}=(T, W, \alpha) \in \LA(Y)$ and let $H=H_{\mathcal{L}}$ be the corresponding $p$-divisible v-group. Using part (7) of Proposition \ref{Prop:GeometricFrechet}, we see that the K\"unneth map
\begin{align}
    \mathcal{O}_{H/Y} \otimes^{\blacksquare}_{\mathcal{O}_Y}\mathcal{O}_{H/Y} \to \mathcal{O}_{\left(H \times_{Y} H\right)/ Y}
\end{align}
is an isomorphism. This means that we can use the group structure (given by the multiplication map and the identity section) to equip $\mathcal{O}_{H/Y}$ with the structure of a commutative solid Hopf $\mathcal{O}_Y$-algebra.\footnote{We define a commutative solid Hopf $\mathcal{O}_Y$-algebra to be a co-commutative co-group object in the category of commutative solid $\mathcal{O}_Y$-algebras with the solid tensor product. It follows in the usual way that for a commutative group $H \to Y$ satisfying the K\"unneth formula for the solid tensor product, the sheaf $\mathcal{O}_{H/Y}$ is a solid Hopf $\mathcal{O}_Y$-algebra.} \smallskip

Let $\mathcal{D}_{H/Y}$ be $\Hom_{\mathcal{O}_Y}(\mathcal{O}_{H/Y}, \mathcal{O}_Y)$. Then by part (8) of Proposition \ref{Prop:GeometricFrechet}, we see that the natural map
\begin{align}
    \mathcal{D}_{H/Y} \otimes^{\blacksquare}_{\mathcal{O}_Y} \mathcal{D}_{H/Y} \to \Hom_{\mathcal{O}_Y}(\mathcal{O}_{H/Y} \otimes^{\blacksquare}_{\mathcal{O}_Y} \mathcal{O}_{H/Y}, \mathcal{O}_Y)
\end{align}
is an isomorphism. This means that the commutative solid Hopf $\mathcal{O}_Y$-algebra structure on $\mathcal{O}_{H/Y}$ induces a commutative solid Hopf $\mathcal{O}_Y$-algebra structure on $\mathcal{D}_{H/Y}$.

\subsection{Locally analytic functions} \label{Sub:FunctionsII}

Let $Y$ be a small v-stack and let $\mathcal{X}=(\Lambda, V, \gamma)$ be a locally analytic character datum. We are going to define a sheaf $\mathcal{O}^{\gamma-\locan}_{\Lambda/Y} \subseteq \mathcal{O}_{\Lambda/Y}$ of functions that extend locally along $\gamma$ to an open neighborhood in $V$, and show that it is a subsheaf of $\mathcal{O}_{\Lambda/Y}$. 

\begin{Def}\label{def.loc-an-vectors}
Let $Y$ be a small v-stack.
    \begin{enumerate}
    \item Let $\mathcal{X}=(\Lambda, V, \gamma)$ be a locally analytic character datum over $Y$. We define
    \[ \mathcal{O}_{\Lambda/Y}^{\gamma-\locan}: = \varinjlim_{(n, U)} \mathcal{O}_{U/Y}, \]
    where the colimit is over pairs $(n,U)$ where $n \in \mathbb{Z}_{\ge 0}$ and $U$ is an open neighborhood of the image of $\Lambda$ in $\Lambda/p^n \Lambda \times V$ (under the product of $\Lambda \rightarrow \Lambda/p^n$ and $\gamma|_{\Lambda}$). We let
    \[ \mathcal{D}_{\Lambda/Y}^{\gamma-\locan}:= \Hom_{\mathcal{O}_Y}(\mathcal{O}_{\Lambda/Y}^{\gamma-\locan}, \mathcal{O}_Y). \]
    \item Let $\Lambda$ be a $\zp$-local system over $Y$. We write $\mathcal{O}_{\Lambda/Y}^{\locan} := \mathcal{O}_{\Lambda/Y}^{\mathrm{Id}_{\Lambda \otimes_{\ul{\mbb{Z}_p}} \mathcal{O}_Y}-\locan}$, i.e., the sheaf of locally analytic functions attached to the character datum $(\Lambda, \Lambda \otimes_{\ul{\zp}} \mathcal{O}_Y, \mathrm{Id}_{\Lambda \otimes_{\ul{\mbb{Z}_p}} \mathcal{O}_Y}).$ 
    \end{enumerate}
\end{Def}

\subsubsection{} The assignment $(\Lambda, V, \gamma) \mapsto \mathcal{O}_{\Lambda/Y}^{\gamma-\locan}$ is contravariantly functorial in $(\Lambda, V, \gamma)$. 
\begin{Lem}\label{lem.cofinal-system}
    In the notation of Definition \ref{def.loc-an-vectors}-(1), suppose $U_0$ is an open subsheaf of $V$ containing $\gamma(\Lambda \otimes_{\ul{\mbb{Z}_p}} \mathcal{O}_Y^+)$ and contained in $p^{-N} \gamma(\Lambda \otimes_{\ul{\mbb{Z}_p}} \mathcal{O}_Y^+)$ for some $N>0$. Then, the subsheaf $U_{0,n}$ of $\Lambda/p^n \times_Y V$ consisting of pairs $(\overline{\lambda}, v)$ where $\gamma(\lambda) - v \in p^n U_0$ for one, equivalently any, local lift $\lambda$ of $\overline{\lambda}$, is open. Moreover, there is a natural isomorphism
    \[ \mathcal{O}_{\Lambda/Y}^{\gamma-\locan} \to \varinjlim_{n} \mathcal{O}_{U_{0,n}/Y}. \]
\end{Lem}
\begin{proof}
To verify that $U_{0,n}$ is an open subsheaf, it suffices to show that $U_{0,n}$ is open after restriction to any affinoid perfectoid test object $Y'=\Spa(A,A^+)$ where $\Lambda\simeq \ul{\mbb{Z}}_p^{h}$ is trivial and $V\simeq\mathcal{O}_{Y'}^d$. We thus assume $Y=Y'$ and fix such trivializations. For any $n \geq 0$, we may fix elements $\lambda_1, \ldots, \lambda_{p^{nh}}$ of $\mbb{Z}_p^{h}$ that map bijectively onto $\mbb{Z}_p^{h}/p^n \mbb{Z}_p^{h}$. Then $U_{0,n}$ is the disjoint union of the open subsheaves
   \[ \{\overline{\lambda_i}\} \times (\gamma(\lambda_i) + p^n U_0) \subseteq \Lambda/p^n \Lambda \times_{Y} V. \]
To show the second statement, we observe that $U_{0,n}$ is an open neighborhood of the image of $\Lambda$ in $\Lambda/p^{n}\Lambda \times V$. Thus there is a restriction map
\begin{align}
    \varinjlim_n \mathcal{O}_{U_{0,n}/Y} \to \varinjlim_{(n,U)} \mathcal{O}_{U/Y}.
\end{align}
We can show it is an isomorphism after restricting to any affinoid perfectoid $Y'=\spa(A,A^+)$ where $\Lambda$ and $V$ are trivial. We thus assume $Y=Y'$ and fix such trivializations. Then, it suffices to show that for each $(U,n)$ as in Definition \ref{def.loc-an-vectors}-(1), there exists $m \geq n$ such that $U_{0,m}$ is contained in the pre-image of $U$ under $\Lambda/p^{m}\Lambda \times V \to \Lambda/p^{n}\Lambda \times V$. Equivalently, that the image of $U_{0,m}$ in $\Lambda/p^{n}\Lambda \times V$ is contained in $U$ for some $m$. \smallskip 

For this, we note that $U_{0,m}$ is the image in $\Lambda/p^{m}\Lambda \times V$ of $\Lambda \times p^{m} U_0$ (under the map given by the product of $\Lambda \rightarrow \Lambda/p^n$ and $\gamma|_{\Lambda}$ in the first factor).Thus the image of $U_{0,m}$ in $\Lambda/p^{n}\Lambda \times V$ is also the image of $\Lambda \times p^{m} U_0$. But these clearly form a basis of neighborhoods of the image of $\Lambda$ as $m$ goes to infinity. 
\end{proof}

\begin{Rem} \label{Rem:CanonicalUNaught}
In Lemma \ref{lem.cofinal-system}, one can always take $U_0 = \gamma (\Lambda \otimes_{\ul{\mbb{Z}_p}} \mathcal{O}_Y^+)$, and this choice is moreover functorial so that we could have simply defined $\mathcal{O}_{\Lambda/Y}^{\locan}$ as the colimit in Lemma \ref{lem.cofinal-system} for this fixed choice of $U_0$. However, to establish good functional analytic properties of $\mathcal{O}_{\Lambda/Y}^{\gamma-\locan}$, it is convenient to use the more general definition; in particular, it will be convenient to take $U_0$ to be an affinoid ball containing $\gamma (\Lambda \otimes_{\ul{\mbb{Z}_p}} \mathcal{O}_Y^+)$ so that one obtains a colimit of Banach modules under injective transition maps.
\end{Rem}

\begin{Lem}\label{lem.pullback-of-locan}
For any map of small v-stacks $f: Y' \rightarrow Y$, the natural map 
\[ f^*\mathcal{O}_{\Lambda/Y}^{\gamma-\locan} \rightarrow \mathcal{O}_{f^*\Lambda/Y'}^{f^*\gamma-\locan} \]
is an isomorphism. 
\end{Lem}
\begin{proof}
    The presentation of Lemma \ref{lem.cofinal-system} is compatible with pullback, and then we can use Corollary \ref{Cor:CohomologyAndBaseChange}.
\end{proof}

\subsubsection{} \label{subsub:derivationaction} Let $Y$ be a small v-stack and $\mathcal{X}=(\Lambda,V, \gamma)$ be a locally analytic character datum. Note that for an affinoid perfectoid $Y'=\spa(A,A^+) \to Y$, the basechange $U_{0,m,Y'}$ is the diamond associated with a smooth commutative group in adic spaces, with Lie algebra $V_{Y'}$. Then $V_{Y'}$ acts by invariant derivations on each term of the direct system, and thus on the colimit $\mathcal{O}_{\Lambda/Y}^{\gamma-\locan}$. This action is moreover compatible with pullback, and thus gives an action of $V$ on $\mathcal{O}_{\Lambda/Y}^{\gamma-\locan}$. In other words, we have constructed  a functorial $\mathcal{O}_Y$-linear map $$V \rightarrow \End_{\mathcal{O}_Y}(\mathcal{O}_{\Lambda/Y}^{\gamma-\locan}).$$
\begin{Lem} \label{Lem:LaAndCauchyRiemann}
 Let $(\Lambda, V, \gamma)$ be as before and let $W=\ker \gamma$. Then the natural map
 \begin{align}
    \mathcal{O}_{\Lambda/Y}^{\gamma-\locan} \to \mathcal{O}_{\Lambda/Y}^{\locan}
 \end{align}
induced by $(\Lambda, \Lambda \otimes_{\ul{\zp}} \mathcal{O}_Y, \operatorname{id}_{\Lambda \otimes_{\ul{\zp}} \mathcal{O}_Y}) \to(\Lambda, V, \gamma)$ is a monomorphism and identifies $\mathcal{O}_{\Lambda/Y}^{\gamma-\locan}$ with the $W$-fixed points of $\mathcal{O}_{\Lambda/Y}^{\locan}$ under the action of \S \ref{subsub:derivationaction}.
\end{Lem}
\begin{proof}
By the compatibility with pullback of Lemma \ref{lem.pullback-of-locan} we may assume $Y$ is an affinoid perfectoid $\Spa(A,A^+)$ and $\Lambda$, $V$, and $W$ are all trivial. Let $U'_0=\Lambda \otimes_{\ul{\zp}} \mathcal{O}_{Y}^+$ and $U_0=\gamma(\Lambda \otimes_{\ul{\zp}} \mathcal{O}_{Y}^+)$. Using the presentation of Lemma \ref{lem.cofinal-system} twice, we can identify the natural map of the lemma with
\begin{align}
    \varinjlim_{m} \mathcal{O}_{U_{0,m}/Y} \to \varinjlim_{m} \mathcal{O}_{U_{0,m}'/Y}.
\end{align}
By the discussion in \S \ref{subsub:derivationaction}, this natural map is $\Lambda \otimes_{\zp} \mathcal{O}_Y$-equivariant for the action on the left hand side given by $\gamma:\Lambda \otimes_{\zp} \mathcal{O}_Y \to V$. In particular, it induces a natural map
\begin{align}
    \varinjlim_{m} \mathcal{O}_{U_{0,m}/Y} \to \left(\varinjlim_{m} \mathcal{O}_{U_{0,m}'/Y}\right)^{W},
\end{align}
which we want to show is an isomorphism. Taking fixed points for the action of $W$ is a finite limit (since $W$ is a finitely generated free $\mathcal{O}_Y$-module) and thus passes through the filtered colimit. Thus it suffices to show that the natural map
\begin{align}
    \mathcal{O}_{U_{0,m}/Y}\to (\mathcal{O}_{U_{0,m}'/Y})^{W} 
\end{align}
is an isomorphism. In other words, we are trying to show that functions on a ball inside $\Lambda \otimes_{\ul{\zp}} \mathcal{O}_Y$ that are killed by the derivations coming from $W=\ker \gamma$,
must factor through $\gamma:\Lambda \otimes_{\ul{\zp}} \mathcal{O}_Y \to V$. Writing $\Lambda \otimes_{\ul{\zp}} \mathcal{O}_Y$ as $W \oplus V$, see Lemma \ref{Lem:LocalDirectSummand}, this can be checked in coordinates where it is straightforward.
\end{proof}

It is immediate from Definition \ref{def.loc-an-vectors} that there is a natural restriction map
\begin{equation}\label{eq.restriction-map-locally-analytic}\mathcal{O}_{\Lambda/Y}^{\gamma-\locan} \rightarrow \mathcal{O}_{\Lambda/Y}.\end{equation}

\begin{Lem}\label{lemma.restriction-map-injective}
The restriction map \eqref{eq.restriction-map-locally-analytic} is injective. 
\end{Lem}
\begin{proof}
    By Lemma \ref{Lem:LaAndCauchyRiemann}, it suffices to prove that the restriction map $\mathcal{O}_{\Lambda/Y}^\locan \rightarrow \mathcal{O}_{\Lambda/Y}$ is injective. Evaluating on an affinoid perfectoid $\Spa(A,A^+)/Y$ where $\Lambda\simeq \ul{\zp^h}$ is trivial, the restriction map is the natural injection $C^{\locan}(\zp^h, A) \hookrightarrow \Cont(\zp^h, A).$ 
    We conclude as such affinoid perfectoids form a basis for $Y_{\mathrm{v}}$. 
\end{proof}
Let $Y$ be a small v-stack and let $\mathcal{X}=(\Lambda,V,\gamma)$ be a locally analytic character datum over $Y$. 
\begin{Prop} \label{Prop:GeometricLB}
Assume that $Y=\spd(R,R^+)$ for a fiercely v-complete Huber pair $(R,R^+)$ over $(\qp,\zp)$ and that $V$ is free. The following hold: 
    \begin{enumerate}
        \item For all morphisms of fiercely v-complete Huber pairs $(R,R^+) \to (B,B^+)$, the natural map ($Y'=\spd(B,B^+)$)
        \begin{align}
            H^0_{\cond}(Y,  \mathcal{O}_{\Lambda/Y}^{\gamma-\locan}) \otimes_{\ul{A}}^{\blacksquare} \ul{B} \to H^0_{\cond}(Y', \mathcal{O}_{\Lambda/Y}^{\gamma-\locan})
        \end{align}
        is an isomorphism.
    
        \item The natural map
        \begin{align}
        H^0_{\cond}(Y,\mathcal{O}_{\Lambda/Y}^{\gamma-\locan}) \otimes_{\ul{R}}^{\blacksquare} H^0_{\cond}(Y,\mathcal{O}_{\Lambda/Y}^{\gamma-\locan}) \to  H^0_{\cond}(Y,\mathcal{O}_{(\Lambda \oplus \Lambda)/Y}^{(\gamma \oplus \gamma)-\locan})
    \end{align}
    is an isomorphism. 

    \item The natural map
    \begin{align}
        H^0_{\cond}(Y,\mathcal{D}_{\Lambda/Y}^{\gamma-\locan}) &\to \Hom_{\ul{A}}(H^0_{\cond}(Y,\mathcal{O}_{\Lambda/Y}^{\gamma-\locan}), \ul{R})
    \end{align}
    is an isomorphism. 

    \item The natural map
    \begin{align}
        H^0_{\cond}(Y,\mathcal{D}_{\Lambda/Y}^{\gamma-\locan}) \otimes_{\ul{R}}^{\blacksquare} H^0_{\cond}(Y,\mathcal{D}_{\Lambda/Y}^{\gamma-\locan}) \to  H^0_{\cond}(Y,\mathcal{D}_{(\Lambda \oplus \Lambda)/Y}^{(\gamma \oplus \gamma)-\locan})
    \end{align}
    is an isomorphism.

    \item The module $H^0_{\cond}(Y,\mathcal{D}_{\Lambda/Y}^{\gamma-\locan})$ is strongly countably Fr\'echet of compact type. 
     \end{enumerate}
For an arbitrary small v-stack $Y$, the following hold:
\begin{enumerate}
\setcounter{enumi}{5}
         \item The sheaf $\mathcal{O}_{\Lambda/Y}^{\gamma-\locan}$ is reflexive.

        \item The sheaf $\mathcal{D}_{\Lambda/Y}^{\gamma-\locan}$ is locally strongly countably Fr\'echet of compact type. 

        \item The natural map
        \begin{align}
            \mathcal{O}_{\Lambda/Y}^{\gamma-\locan} \otimes_{\mathcal{O}_Y}^{\blacksquare} \mathcal{O}_{\Lambda/Y}^{\gamma-\locan} \to \mathcal{O}_{(\Lambda \oplus \Lambda)/Y}^{(\gamma \oplus \gamma)-\locan}
        \end{align}
        is an isomorphism.

        \item The natural map
        \begin{align}
            \mathcal{D}_{\Lambda/Y}^{\gamma-\locan} \otimes_{\mathcal{O}_Y}^{\blacksquare} \mathcal{D}_{\Lambda/Y}^{\gamma-\locan} \to \left( \mathcal{O}_{\Lambda/Y}^{\gamma-\locan} \otimes_{\mathcal{O}_Y}^{\blacksquare} \mathcal{O}_{\Lambda/Y}^{\gamma-\locan} \right)^{\ast}
        \end{align}
        is an isomorphism.
\end{enumerate}
\end{Prop}
We start by proving a lemma.

\begin{Lem} \label{Lem:OgammalaIsLB}
Let $Y=\Spd (R,R^+)$ for a fiercely v-complete Huber pair $(R,R^+)$ over $(\qp,\zp)$, and let $\mathcal{X}=(\Lambda, \mathcal{O}_{Y}^{\oplus d}, \gamma)$ be a locally analytic character datum such that $\restr{\gamma}{\Lambda}$ factors through the unit ball $(\mathcal{O}^{+}_{Y})^{d}$. For any $n \geq 0$, let $U_n \subseteq \Lambda/p^n \Lambda \times \mathcal{O}_Y^{\oplus d}$ be the open subset of pairs $(\lambda, w)$ where $\gamma(\lambda)-w \in p^n (\mathcal{O}_Y^{+})^d$. Then   
    \[ \mathcal{O}_{\Lambda/Y}^{\gamma-\locan}(Y) = \varinjlim_n \mathcal{O}(U_n)=\varinjlim_n H^0(Y, \mathcal{O}_{U_n/Y}), \]
and $H^0_\cond(Y, \mathcal{O}_{U_n/Y})$ is a direct summand of an orthonormalizable Banach $R$-module. 

Moreover, suppose $(R,R^+) \to (B,B^+)$ is a map of fiercely v-complete Huber pairs, and let $Y'=\Spd(B,B^+)$. Then the natural map
    \begin{equation} \label{Eqn:BCforSummandsOfBalls}
    H^0_{\cond}(Y, \mathcal{O}_{U_{n}/Y}) \otimes_{\ul{R}}^{\blacksquare} \ul{B} \to H^0_{\cond}(Y', \mathcal{O}_{U_{n}/Y})
    \end{equation}
    is an isomorphism.
\end{Lem}
\begin{proof}
The first equality follows from Lemma \ref{lem.cofinal-system}. 

For the rest, we first choose a perfectoid cover $Y_\infty=\Spa(R_\infty, R^+_\infty)$ with $R_\infty = (\varinjlim_i R_i)^\wedge$ where each $R \rightarrow R_i$ is faithfully flat finite \'{e}tale and such that $\Lambda$ is trivial after pullback to $Y_\infty$ --- this is possible by first taking a cover as in \cite[Lemma 3.6.26]{KedlayaLiu} and then taking fiber products with the trivializing covers for $\Lambda/p^{i} \Lambda$. We write $R_i^+$ for the integral closure of $R^+$ in $R_i$, and $Y_i=\Spa(R_i, R_i^+)$. 

We are going to first show that, for some index $i$, $U_n \times_Y {Y_i}$ is a finite disjoint union of balls over $Y_i$. First we fix a section $\lambda: Y_\infty \rightarrow \Lambda_{Y_\infty}$. We write $\overline{\lambda}: Y_\infty \rightarrow (\Lambda/p^n \Lambda)_{Y_\infty}$ for the induced map. Then the map $p^n \mathcal{O}^+_{Y_\infty} \rightarrow \overline{\lambda}^* U_{n,Y_\infty}$, $v \mapsto \gamma(\lambda)+v$ is an isomorphism, i.e., $\overline{\lambda}^* U_{n,Y_\infty}$ is a closed ball of radius $|p^n|$ in $\mathbb{A}^d_{Y_\infty}$ centered at $\gamma(\lambda)$. Writing $\gamma(\lambda)=(a_1, \ldots, a_d) \in R_\infty^d$, we find that, by the density of $\varinjlim_i R_i$ in $R_\infty$, for $i$ sufficiently large there is a center $(b_1, \ldots, b_d) \in R_i^d$ for the same ball. 

In particular, if we first pass to $i$ large enough to fix an isomorphism $\Lambda/p^n \Lambda = (\mathbb{Z}_p/p^n\mathbb{Z}_p)^h$, then we may pass to a further $i$ so that for each $\alpha \in (\mathbb{Z}_p/p^n \mathbb{Z}_p)^h$, we have $U_n|_{\alpha \times Y_i}$ is a ball of radius $p^n$. Thus $U_n \times_Y {Y_i}$ is a disjoint union of $p^{nh}$ $d$-dimensional balls of radius $|p^n|$ over $Y_i$. It follows, in particular, that $\mathcal{O}(U_n \times_Y Y_i)$ is an orthonormalizable Banach $R_i$-module. 

Now, since $R \rightarrow R_i$ is faithfully flat finite \'{e}tale, descent is effective along $R \rightarrow R_i$, and thus we find
    \[ \mathcal{O}(U_n) \otimes_R R_i = \mathcal{O}(U_n \times_Y Y_i). \]
(in making the descent argument, note that $R_i/R$ is finite so the completed tensor product with $R_i$ agrees with the regular tensor product). Note that $R_i$ is projective over $R$, so $R$ is a direct summand of $R_i$ as an $R$-module and $R_i$ is a direct summand of $R^m$ as an $R$-module (see the proof of Lemma \ref{Lem:SheafFrechet}). It follows from Lemma \ref{Lem:CondensedRingStructure} that $H^0_\cond(Y, \mathcal{O}_{U_n/Y})$ is a direct summand of an orthonormalizable Banach $R$-module. 

Now suppose given a map $(R,R^+) \rightarrow (B,B^+)$ of fiercely $v$-complete Huber pairs, let $Y'=\Spa(B,B^+)$, and, for $i$ as above, let $Y_i' = Y_i \times_Y Y'=\spa(B_i, B_i^+)$. Then, since $U_n \times_Y {Y_i}$ is a union of $p^{nh}$ $d$-dimensional balls of radius $|p^n|$, we see that the natural map 
    \begin{equation} \label{Eqn:BCinproofforballs}
    H^0_{\cond}(U_n \times_Y Y_i, \mathcal{O}) \otimes^{\blacksquare}_{\ul{R_i}} \ul{B_i} \to H^0_{\cond}(U_n \times_Y Y_i', \mathcal{O})
    \end{equation}
    is an isomorphism. Indeed, this amounts to checking $\ul{R_i\langle T_1, \ldots, T_d} \rangle \otimes^{\blacksquare}_{\ul{R_i}} \ul{B_i} \simeq \ul{B_i\langle T_1, \ldots T_d \rangle}$, which
    is Proposition \ref{Prop:StandardBCproperty}. But now applying the projectors for the direct summands as above, we conclude 
    \[ H^0_{\cond}(U_n \times_Y Y', \mathcal{O})=H^0_{\cond}(U_n, \mathcal{O}) \otimes_{\ul{R}}^{\blacksquare} \ul{B}, \]
    giving the claimed result. 
\end{proof}

In the proof of Proposition \ref{Prop:GeometricLB}, we will need the following lemma.
\begin{Lem} \label{Lem:OgammaisLSspace}
With the notation as in Lemma \ref{Lem:OgammalaIsLB}, let $U_n^{\circ} \subset \Lambda/p^{n+1} \Lambda \times \mathcal{O}_Y^{\oplus d}$ denote the open subset of pairs $(\lambda, w)$ where $\gamma(\lambda) - w \in p^{n + \varepsilon}(\mathcal{O}_Y^+)^d$ for some $\varepsilon > 0$. Then we have 
    \begin{align} \label{Eqn:OgammalaPresentation}
    \mathcal{O}_{\Lambda/Y}^{\gamma-\locan}(Y) = \varinjlim_n H^0(Y, \mathcal{O}^{\mathrm{bdd}}_{U^{\circ}_n/Y}),
    \end{align}
    where $\mathcal{O}^{\mathrm{bdd}}_{U^{\circ}_n/Y} := \mathcal{O}^{+}_{U^{\circ}_n/Y}[\tfrac{1}{p}]$, and each term in the colimit is a direct summand of a countably free Smith $R$-algebra.
\end{Lem}
\begin{proof}
This follows as in the proof of Lemma \ref{Lem:OgammalaIsLB}, noting that $U_{n, Y_i}^{\circ}$ is union of $p^{(n+1)h}$ ``open balls'' of radius $|p^n|$, and that the restriction maps $\mathcal{O}_{U_n/Y} \to \mathcal{O}_{U_{n+1}/Y}$ factor through $\mathcal{O}_{U^{\circ}_{n}}^{\mathrm{bdd}}$. We now observe that each term in the colimit is a direct summand of a countably free Smith $R$-algebra because $\mathcal{O}^{\mathrm{bdd}}_{U^{\circ}_{n,Y_i}}$ is countably free Smith; indeed, one has $\mathcal{O}^{\mathrm{bdd}}_{U^{\circ}_{n,Y_i}}(R, R^+) \simeq \bigoplus_{\left(\mbb{Z}/p^{n+1}\mbb{Z}\right)^h} R^+[\![T_1, \dots, T_d]\!][\tfrac{1}{p}]$ for any affinoid open $\operatorname{Spd}(R, R^+) \to Y_i$, see Lemma \ref{Lem:BoundedFunctions}.
\end{proof}
\begin{proof}[Proof of Proposition \ref{Prop:GeometricLB}]
Part (1) follows from Lemma \ref{Lem:OgammalaIsLB} by commuting the colimit with the solid tensor product. Part (3) follows directly from the first property and Lemma \ref{Lem:EvaluationAbstract}. Part (2) follows directly from Lemma \ref{Lem:OgammalaIsLB} and Lemma \ref{Lem:Kunneth}, together with the fact that the solid tensor product commutes with direct limits. \smallskip

We now continue with the notation of Lemma \ref{Lem:OgammaisLSspace}. 

\begin{Claim}
The dual transition maps
\begin{align}
	\mathcal{O}^{\mathrm{bdd}}(U_n^{\circ})^{\ast} \to \mathcal{O}^{\mathrm{bdd}}(U_{n+1}^{\circ})^{\ast}
\end{align} 
are compact in the sense of Section \ref{subsub:compact} and have dense image.    
\end{Claim}
\begin{proof}[Proof of Claim]
We can choose an $i$ sufficiently large as in the proof of Lemma \ref{Lem:OgammaisLSspace} so that $U_{n+1}^\circ\times Y_{Y_i}$ is a union of $p^{(n+2)h}$ open balls of radius $|p^{(n+1)}|$, $U_{n}^\circ \times_Y {Y_i}$ is a union of $p^{(n+1)h}$ open balls of radius $|p^{n}|$, and the map $U_n^\circ \times_Y {Y_i} \rightarrow U_{n+1}^\circ \times_Y {Y_i}$ is described on each open ball of radius $|p^{n}|$ as the sum of the restriction maps to $p^h$ (not necessarily distinct) open sub-balls of radius $|p^{n+1}|$. We will reduce to analyzing these simple restriction maps.

To that end, we first note that for $e$ an $R=\mathcal{O}(Y)$-linear projector on $\mathcal{O}(Y_i)$ such that $e \mathcal{O}(Y_i)=\mathcal{O}(Y)$, we have 
\[ e \mathcal{O}^\bdd(U_\bullet^\circ \times_Y Y_i) = \mathcal{O}^\bdd(U_\bullet^\circ) \textrm{ for } \bullet=n,n+1.\]
This implies 
\[ e \mathcal{O}^\bdd(U_\bullet^\circ \times_Y Y_i)^* = \mathcal{O}^\bdd(U_\bullet^\circ)^* \textrm{ for } \bullet=n,n+1.\]
so it suffices to prove that 
\begin{equation}\label{eq.balls-dual-restriction-map}\mathcal{O}^\bdd(U_{n+1}^\circ \times_Y Y_i)^* \rightarrow \mathcal{O}^\bdd(U_n^\circ \times_Y  Y_i)^*\end{equation} has dense image and is compact, since these properties are preserved after applying $e$. 

To that end, we first note that for any of the single restriction maps from a ball of radius $|p^n|$ to $|p^{n+1}|$, we are free to choose coordinates such that both balls are centered at zero and rescale to obtain, at the level of functions, the map
\[ R^+_i[[T_1, \ldots, T_d]][\tfrac{1}{p}] \xrightarrow{ T_i \mapsto p T_i} R^+_i[[T_1, \ldots, T_d]][\tfrac{1}{p}].\]
The dual of this map is 
\begin{equation}\label{eq.dual-of-simple-map} R_i^* \langle S_1, \ldots, S_d\rangle \xrightarrow{ S_i \mapsto p S_i} R_i^*\langle S_1, \ldots, S_d\rangle \end{equation}
where $R_i^*$ denotes the $R$-linear dual. Since $R_i$ is projective, so is $R_i^*$, thus for some $N$ this is a direct summand of the map
\[ (R\langle S_1, \ldots, S_d \rangle \xrightarrow{ S_i \mapsto p S_i} R\langle S_1, \ldots, S_d\rangle)^{\oplus N}. \]
Thus, to see \eqref{eq.dual-of-simple-map} has dense image and is compact, it suffices to observe that 
\begin{equation}\label{eq.tate-mult-by-p-map} R\langle S_1, \ldots, S_d \rangle \xrightarrow{ S_i \mapsto p S_i} R\langle S_1, \ldots, S_d\rangle.\end{equation}
has dense image and is compact (since these properties are preserved under finite direct sums and, by applying a projector as earlier in the proof, direct summands). The image of \eqref{eq.tate-mult-by-p-map} is dense because it contains the polynomials $R[S_1, \ldots, S_d]$, and we have already argued at the end of the proof of Lemma \ref{Lem:GeometryofH} that \eqref{eq.tate-mult-by-p-map} is compact. 

Since \eqref{eq.dual-of-simple-map} has dense image, so does \eqref{eq.balls-dual-restriction-map}: this follows already by considering a single ball in $U_{n+1}^\circ \times_Y Y_i$ contained in each ball in $U_n^\circ \times_Y  Y_i$. Similarly, by grouping the restriction maps according to their source ball in $U_n^\circ \times_Y  Y_i$, we find that \eqref{eq.dual-of-simple-map} is a direct sum of $p^{nh}$ maps, each of which is a sum of $p^h$ compact maps, and thus it is compact. 
\end{proof}

The claim proves part (5). Now part (4) follows from (2) together with part (3) and Corollary \ref{Cor:DualTensorCommute}. \smallskip 

Parts (6)-(9) are v-local on the base and so we may assume that $Y=\spd(A,A^+)$ with $(A,A^+)$ perfectoid and with $\Lambda$ and $V$ free. We prove (6) directly, by computing
    \begin{align}
        &\Hom_{\mathcal{O}_Y}(\Hom_{\mathcal{O}_Y}(\varinjlim_n \mathcal{O}^{\mathrm{bdd}}_{U^{\circ}_n}, \mathcal{O}_Y), \mathcal{O}_Y) \\
        &=\Hom_{\mathcal{O}_Y}(\varprojlim_n \Hom_{\mathcal{O}_Y}( \mathcal{O}^{\mathrm{bdd}}_{U^{\circ}_n}, \mathcal{O}_Y), \mathcal{O}_Y) \\
        &=\varinjlim_n \Hom_{\mathcal{O}_Y}(\Hom_{\mathcal{O}_Y}( \mathcal{O}^{\mathrm{bdd}}_{U^{\circ}_n}, \mathcal{O}_Y), \mathcal{O}_Y) 
    \end{align}
    as in the proof of Corollary \ref{Cor:FrechetReflexive}, using (5). We conclude by the reflexivity of (direct summands of) countably free smith $\mathcal{O}_Y$-modules, see Proposition \ref{Prop:DualityvSheaves}. Part (7) follows by applying part (5) to affinoid perfectoid $(A,A^+)$ over $Y$ and similarly part (8) follows from part (2) and part (9) follows from parts (2)-(4). 
\end{proof}

\subsubsection{} Let $Y$ be a small v-stack and let $\mathcal{X}=(\Lambda, V, \gamma)$ be a locally analytic character datum over $Y$.

\begin{Lem} \label{Lem:GammaLaSheavesHopf}
    Both $\mathcal{O}_{\Lambda/Y}^{\gamma\mathrm{-la}}$ and $\mathcal{D}_{\Lambda/Y}^{\gamma\mathrm{-la}}$ are naturally solid Hopf $\mathcal{O}_Y$-algebras.
\end{Lem}
\begin{proof}
    For $\mathcal{O}_{\Lambda/Y}^{\gamma\mathrm{-la}}$, the algebra structure is the natural one (and the multiplication structure factors through the solid tensor product), induced from the individual algebra structures on the terms in the colimit defining $\mathcal{O}_{\Lambda/Y}^{\gamma\mathrm{-la}}$. For the co-algebra structure, consider the morphism 
    \[
    \Lambda \oplus \Lambda \xrightarrow{\gamma \oplus \gamma} V \oplus V .
    \]
The co-algebra structure is now induced from pulling back along the addition map $\Lambda \oplus \Lambda \to \Lambda$ and using Proposition \ref{Prop:GeometricLB}-(8). The co-unit is induced from the natural map $0 \to \Lambda$. \smallskip 

This solid Hopf $\mathcal{O}_Y$-algebra structure on $\mathcal{O}_{\Lambda/Y}^{\gamma\mathrm{-la}}$ induces a solid Hopf $\mathcal{O}_Y$-algebra structure on $\mathcal{D}_{\Lambda/Y}^{\gamma\mathrm{-la}}$ once we show that the natural map 
    \begin{equation} \label{Eqn:DualOfSolidEqualsSolidII}
    \left( \mathcal{O}_{\Lambda/Y}^{\gamma\mathrm{-la}} \otimes^{\blacksquare}_{\mathcal{O}_Y} \mathcal{O}_{\Lambda/Y}^{\gamma\mathrm{-la}} \right)^* \leftarrow \mathcal{D}_{\Lambda/Y}^{\gamma\mathrm{-la}} \otimes^{\blacksquare}_{\mathcal{O}_Y} \mathcal{D}_{\Lambda/Y}^{\gamma\mathrm{-la}}
    \end{equation}
    is an isomorphism. But this is part (9) of Proposition \ref{Prop:GeometricLB}.
\end{proof}

\section{The Fourier transform} \label{Sec:Fourier}

In this section we construct the Fourier transform for $p$-divisible v-groups over a small v-stack and establish the main results of this article. In \S \ref{Sub:ExponentialPairing}, we construct the universal character $\kappa$ and establish that it is analytic (Lemma \ref{Lem:UniversalKappaAsSection}). In \S \ref{Sub:FourierDef}, we use this to define the Fourier transform. In \S \ref{Sub:MainTheorem}, we prove Theorem \ref{Thm:IntroFourierVsheaves} (see Theorem \ref{Thm:Main}), and in \S \ref{Sub:Equivariance}, we prove Proposition \ref{Prop:Equivariance}. Finally, we prove Theorem \ref{thm.rational-fourier-theory} in \S \ref{Sub:Field}, and in \S \ref{Sub:STComparison}, we compare our results with those of Schneider--Teitelbaum \cite{SchneiderTeitelbaumFourier}. 

\subsection{The universal character} \label{Sub:ExponentialPairing}
Let $Y$ be a small v-stack and let $\Lambda$ be a locally free sheaf of $\ul{\mathbb{Z}_p}$-modules on $Y_{\mathrm{v}}$. Let $\Lambda^{\ast} := \Hom_{\ul{\mathbb{Z}_p}}(\Lambda, \ul{\mathbb{Z}_p})$ denote its linear dual. Recall that:
\begin{itemize}
    \item $\mathbb{G}_{a, Y}^{\lozenge}$ denotes the group over $Y$ satisfying $\mathbb{G}_{a, Y}^{\lozenge}(R, R^+) = R$ for any $(R, R^+) \in Y_{\mathrm{v}}$;
    \item $\widehat{\mathbb{G}}_{m, \eta, Y}^{\lozenge}$ denotes the group over $Y$ satisfying $\widehat{\mathbb{G}}_{m, \eta, Y}^{\lozenge}(R, R^+) = 1 + R^{00}$ for any $(R, R^+) \in Y_{\mathrm{v}}$, where $R^{00}$ denotes the ideal of topologically nilpotent elements.
\end{itemize}
Note that all of these sheaves of abelian groups come equipped with a natural action of $\ul{\zp}$, see \S \ref{subsub:HeuerXu}. 

\subsubsection{} Recall from Example \ref{Eg:RigidGM} that there is a logarithm map $\operatorname{log}^{\lozenge} \colon \widehat{\mathbb{G}}_{m, \eta, Y}^{\lozenge} \to \mathbb{G}_{a, Y}^{\lozenge}$. Recall moreover that, for $\varpi=2p$, the map $\operatorname{log}^{\lozenge}$ is an isomorphism from $\mathbb{B}(1, |\varpi|)$ to $\mathbb{B}(0, |\varpi|)$, with inverse defined by the usual exponential series.

\subsubsection{} We let $H_{\Lambda}=\Lambda^* \otimes_{\ul{\mathbb{Z}_p}} \widehat{\mathbb{G}}_{m,\eta,Y}^\lozenge$ be the full character group of $\Lambda$ (in terms of Definition \ref{Def:CharacterDatum}, this is the group $H_{\mathcal{X}}$ corresponding to the locally analytic character datum $\mathcal{X}=(\Lambda, \Lambda \otimes_{\ul{\zp}} \mathcal{O}_Y, \mathrm{id}_{\Lambda \otimes_{\ul{\zp}} \mathcal{O}_Y})$).  

\begin{Def}\label{def.H-lambda-charts}
    Let $i \geq 0$ and set $\varpi = 2p$. We define $H^{(i)}_{\Lambda} \subset H_{\Lambda}$ to be the open\footnote{To see that this is open, we can check this v-locally -- in particular, we may assume that $\Lambda$ is free, and then the result is clear.} subgroup defined by
    \[
    H^{(i)}_{\Lambda} = \Lambda^{\ast} \otimes_{\ul{\mathbb{Z}_p}} \mathbb{B}(1, |\varpi|^{1/p^i})_Y^{\lozenge}
    \]
    where $\mathbb{B}(1, |\varpi|^{1/p^i}) \subset \widehat{\mathbb{G}}_{m, \eta}$ denotes the subset defined by the inequality $|T|^{p^i} \leq |\varpi|$. 
\end{Def}

Note that $H_{\Lambda} = \bigcup_{i \geq 0} H^{(i)}_{\Lambda}$, and $\mathcal{O}_{H_{\Lambda}/Y}$ is locally strongly countably Fr\'{e}chet, with a presentation given by $\mathcal{O}_{H_{\Lambda}/Y} = \varprojlim_i \mathcal{O}_{H^{(i)}_{\Lambda}/Y}$, see Proposition \ref{Prop:GeometricFrechet}-(3).

\subsubsection{}

Let $Z_{\Lambda} = \Lambda \otimes_{\ul{\mathbb{Z}_p}} \mathbb{G}_{a, Y}^{\lozenge}$. We will also consider certain open subgroups of this space thickening $\Lambda$ as in Lemma \ref{Lem:OgammalaIsLB}. Explicitly:

\begin{Def}
For $i \geq 0$, consider the open affinoid subgroup $\mathbb{G}^{(n)}_a \subset \mathbb{G}_a$ given by $\mathbb{G}^{(n)}_a(R, R^+) = \ul{\mathbb{Z}_p}(R) + p^i R^+$. We let
\[
Z_{\Lambda,i} := \Lambda \otimes_{\ul{\mathbb{Z}_p}} \mathbb{G}_{a,Y}^{(i), \lozenge} \subset Z_{\Lambda} 
\]
denote the corresponding open subgroup. 
\end{Def}

\begin{Lem}
$\mathcal{O}_{Z_{\Lambda,i}/Y}$ is locally countably orthonormalizable Banach. 
\end{Lem}
\begin{proof}
For $(R, R^+) \in Y_{\mathrm{v}}$ such that $\Lambda|_{\operatorname{Spa}(R, R^+)}$ is free of rank $h$, if we fix an isomorphism $\Lambda|_{\Spa(R,R^+)}\cong \ul{\mathbb{Z}_p^h}$, then $Z_{\Lambda,n, \spd(R,R^+)}$ is a disjoint union of balls indexed by $(\mathbb{Z}_p/p^n\mathbb{Z}_p)^h$. The lemma now follows from Lemma \ref{Lem:SheafBanach}.
\end{proof}

For $i \geq 0$, also consider the subgroup $Z_{\Lambda,i}^{\circ} \subset Z_{\Lambda,i}$ given by the sheafification of 
    \[
    (R, R^+) \mapsto p^i\left( \Lambda(R) \otimes_{\ul{\mathbb{Z}_p}(R)} R^+ \right).
    \]

\begin{Lem}\label{lemma.lambda-thickening-pushout}
The natural map $\Lambda \oplus Z_{\Lambda,i}^\circ \rightarrow Z_{\Lambda,i},\; (\lambda, z) \mapsto \lambda + z$ is a surjection with kernel $p^i \Lambda$ embedded antidiagonally by $\lambda \mapsto (\lambda,-\lambda)$. 
\end{Lem}
\begin{proof}
    Immediate. 
\end{proof}

Consider the natural pairing
\begin{align*} \kappa_{\Lambda}: \Lambda \times H_{\Lambda} &\rightarrow \mathbb{G}_{m, \eta, Y}^{\lozenge} \\
 (\lambda, \lambda' \otimes z) & \mapsto z^{\lambda'(\lambda)}.
\end{align*} 
It is bilinear: 
\begin{equation}\label{eq.univ-char-bilinearity} \kappa(x_1+x_2, y)=\kappa(x_1, y)\kappa(x_2, y) \textrm{ and } \kappa(x, y_1y_2) = \kappa(x, y_1)\kappa(x, y_2).\end{equation}

Since $\mathbb{G}_{m, \eta, Y}^{\lozenge} \subseteq \mathbb{A}^{1,\lozenge}_Y$, we may view $\kappa_{\Lambda}$ as an element of $\mathcal{O}(\Lambda \times_Y H_{\Lambda})=\mathcal{O}_{\Lambda/Y}(H_{\Lambda}).$ 

\begin{Lem}
    For each $i\geq 0$, the restriction $\kappa|_{\Lambda \times H^{(i)}_{\Lambda}}$ extends naturally to a pairing
\[ \kappa_{\Lambda,i}: Z_{\Lambda,i} \times H_{\Lambda}^{(i)} \rightarrow \mathbb{G}_{m, \eta, Y}^{\lozenge}. \]
In particular, 
    $\kappa_{\Lambda} \in \mathcal{O}_{\Lambda/Y}^{\locan}(H_{\Lambda}) \subseteq \mathcal{O}_{\Lambda/Y}(H_{\Lambda})$. 
\end{Lem}
\begin{proof}
We can define a pairing on $Z_{\Lambda,i}^\circ \times H_{\Lambda}^{(i)}$
by 
\[ 
(\lambda \otimes r, \lambda' \otimes z) \mapsto \mathrm{exp}(\lambda'(\lambda) r \log(z)),
\]
where the exponential converges because $\lambda'(\lambda) \in \underline{\mathbb{Z}_p} \subseteq \mathcal{O}^+$, $r \in p^i \mathcal{O}^+$, and $\log(z) \in p^{-i}\varpi \mathcal{O}^+$, so that the product lies in $\varpi \mathcal{O}^+$. This pairing agrees with $\kappa_\Lambda$ on $p^i \Lambda \times H^{(i)}_{\Lambda}$, thus we obtain the desired pairing $\kappa_{\Lambda, i}$ by invoking Lemma \ref{lemma.lambda-thickening-pushout}. Since 
\[ 
\mathcal{O}_{\Lambda/Y}^\locan(H_{\Lambda})=\varprojlim_i \mathcal{O}_{\Lambda/Y}^\locan(H_{\Lambda}^{(i)})=\varprojlim_i \varinjlim_j \mathcal{O}(Z_{\Lambda,j} \times H_{\Lambda}^{(i)})
\]
we conclude $\kappa_{\Lambda} \in \mathcal{O}_{\Lambda/Y}^\locan(H_{\Lambda})$ (note that for $j \geq i$, we have $\kappa_{\Lambda,j}|_{Z_{\Lambda, j} \times_Y H_{\Lambda}^{(i)}} = \kappa_{\Lambda,i}|_{Z_{\lambda,j} \times_Y H_{\Lambda}^{(i)}}$).
\end{proof}

\begin{Eg} \label{Eg:kappaForTrivialLambda}
    Suppose that $Y = \operatorname{Spa}(R, R^+)$ for $(R,R^+)$ a fiercely v-complete Huber pair over $(\qp,\zp)$, and let $\Lambda = \underline{\mathbb{Z}_p}$ be the trivial local system. Then $H_{\Lambda}(Y) = \mathbb{G}_{m, \eta, Y}^{\lozenge}(Y) = 1 + R^{00}$ and $Z_{\Lambda, i}(Y) = \underline{\mathbb{Z}_p}(R) + p^i R^+$. The pairing $\kappa$ is given by
    \[
    \kappa(z, 1+t) = (1+t)^z = \sum_{k=0}^{\infty} { z \choose k} t^k \; \in \; 1 + R^{00}.
    \]
    For $z \in 1+p^i R^+$, we find this series converges for $t \in \varpi^{1/p^i} R^+$. This can also be deduced from Amice's description of an orthonormal basis for $\mathcal{O}(Z_{\ul{\mathbb{Z}_p},n})$ (see \S \ref{sub:TheAmiceTransform}).
\end{Eg}

\begin{Lem}\label{Lem.DerivationActionLocAn}\hfill
\begin{enumerate} 
\item For the derivation action $\cdot$ of $\Lambda_{H_{\Lambda}} \otimes_{\ul{\zp}} \mathcal{O}_{H_\Lambda}$ on $\mathcal{O}_{\Lambda_{H_\Lambda}/H_{\Lambda}}^\locan$,
\[
\lambda \cdot \kappa_{\Lambda} = \langle \lambda, \log_{H_{\Lambda}}\rangle \kappa_{\Lambda},
\]
where we view $\log_{H_{\Lambda}}$ as a section of $\Lambda_{H_\Lambda}^* \otimes_{\ul{\zp}} \mathcal{O}_{H_\Lambda}$ over $H_\Lambda$. 
\item For the derivation action $\cdot$ of $\Lambda_{H_\Lambda}^* \otimes_{\ul{\zp}} \mathcal{O}_{H_\Lambda}$ on $\mathcal{O}^{\locan}_{\Lambda_{H_\Lambda}/H_{\Lambda}}$, 
\[
\lambda' \cdot \kappa_{\lambda} = \lambda' \kappa_\Lambda,
\]
where on the right-hand side $\lambda'$ is viewed as a function on $\Lambda_{H_\Lambda}$. Here we are identifying $\Lambda_{H_\Lambda}^* \otimes_{\ul{\zp}} \mathcal{O}_{H_\Lambda} = \operatorname{Lie}H_{\Lambda} \otimes_{\mathcal{O}_Y} \mathcal{O}_{H_{\Lambda}}$.
\end{enumerate}
\end{Lem}
\begin{proof}
For both statements, it suffices to work over an affinoid perfectoid $Y'=\Spa (A,A^+) \rightarrow H_\Lambda$ such that $\Lambda|_{Y'} \cong \mathbb{Z}_p^h$ is trivial. Then, fixing such a trivialization, we may view $\kappa_\Lambda$ as an element of $C^\locan(\mathbb{Z}_p^h, A)$ and view $\log_{H_{\Lambda}}$ as an element of $(\mathbb{Z}_p^h)^* \otimes_{\mathbb{Z}_p} A$. By construction, for $x$ in a small enough neighborhood of $0 \in \mathbb{Z}_p^h$, $\kappa(x)=\exp( \langle x, \log_{H_{\Lambda}}\rangle).$

Let $A[\varepsilon]$ denote the ring of dual numbers over $A$ (so $\varepsilon^2 = 0$). For $\lambda \in (\Lambda_{H_{\Lambda}} \otimes_{\ul{\zp}} \mathcal{O}_{H_\Lambda})(Y') = A^h$, we can then compute
\begin{align*} (\lambda \cdot \kappa_\Lambda)(x) & =  \varepsilon^{-1} \cdot (\kappa_\Lambda(x+\varepsilon \lambda) - \kappa_\Lambda(x)) \\
&= \kappa_\Lambda(x) \varepsilon^{-1}(\kappa_\Lambda(\varepsilon \lambda)- 1) \\
&= \kappa_\Lambda(x) \varepsilon^{-1}\exp( \langle \varepsilon \lambda, \log_{H_{\Lambda}} \rangle - 1)\\
&= \kappa_\Lambda(x) \langle \lambda, \log_{H_{\Lambda}} \rangle. \end{align*}
This establishes (1). Similarly, we obtain (2) by computing 
\begin{align*} (\lambda' \cdot \kappa_{\Lambda})(x) & =  \varepsilon^{-1}(\exp(\lambda' \varepsilon) \kappa_\Lambda(x)) - \kappa_\Lambda(x)) \\
&= \varepsilon^{-1}((1+ \lambda'\varepsilon)\kappa_{\Lambda}(x)-\kappa_{\Lambda}(x))\\
&=\lambda' \kappa_\Lambda(x). 
\end{align*}

\end{proof}

\subsubsection{General locally analytic character groups} \label{SubSub:GeneralPDivGps}

Let $Y$ be a small v-stack, and let $\mathcal{X}=(\Lambda, V, \gamma)$ be a locally analytic character datum as in Definition \ref{Def:CharacterDatum} with associated character group $H_\mathcal{X} \subseteq H_{\Lambda}$. We let $\kappa_{\mathcal{X}}$ be the restriction of $\kappa_{\Lambda}$ to $\Lambda \times_Y H_{\mathcal{X}}$.

\begin{Lem} \label{Lem:UniversalKappaAsSection}
One has $\kappa_{\mathcal{X}} \in \mathcal{O}_{\Lambda/Y}^{\gamma-\locan}(H_{\mathcal{X}})$. Furthermore, this section is contravariantly functorial in the character datum $\mathcal{X}$ in the following sense: If $\mathcal{X}'=(\Lambda', V', \gamma')$ is another locally analytic character datum and we have a morphism $\mathcal{X} \to \mathcal{X}'$, then the images of $\kappa_{\mathcal{X}}$ and $\kappa_{\mathcal{X}'}$ agree under the natural maps $\mathcal{O}_{\Lambda/Y}^{\gamma-\locan}(H_{\mathcal{X}}) \to \mathcal{O}_{\Lambda/Y}^{\gamma-\locan}(H_{\mathcal{X}'})$ and $\mathcal{O}_{\Lambda/Y}^{\gamma'-\locan}(H_{\mathcal{X}'}) \to \mathcal{O}_{\Lambda'/Y}^{\gamma-\locan}(H_{\mathcal{X}'})$ respectively.
\end{Lem}
\begin{proof}
The functoriality is immediate from the definitions, so we will just verify that $\kappa_{\mathcal{X}} \in \mathcal{O}_{\Lambda/Y}^{\gamma-\locan}(H_{\mathcal{X}})$. 
It is equivalent to verify that $\kappa_{\mathcal{X}}$, viewed as a global section of $\mathcal{O}_{\Lambda_{H_{\mathcal{X}}}/H_{\mathcal{X}}}^\locan$ over $H_{\mathcal{X}}$, lies in $\mathcal{O}_{\Lambda_{H_{\mathcal{X}}}/H_{\mathcal{X}}}^{\gamma-\locan}$. By Lemma \ref{Lem:LaAndCauchyRiemann}, for $W:=\ker \gamma$, it is then equivalent to verify that $\kappa_{\mathcal{X}}$ is annihilated by the invariant derivation action of $W \subseteq \Lambda_{H_{\Lambda}} \otimes_{\ul{\zp}} \mathcal{O}_{H_{\Lambda}}.$ But by Lemma \ref{Lem.DerivationActionLocAn}, a section $w$ of $W$ acts as multiplication by the function $\langle w, \log_{H_\Lambda}|_{H_{\mathcal{X}}} \rangle$. Since, by the definition of $H_{\mathcal{X}}$, we have that $\log_{H_{\Lambda}}|_{H_{\mathcal{X}}}=\log_{H_{\mathcal{X}}}$ lies in the annihilator of $W$, we conclude. 
\end{proof}

\subsection{Constructing the Fourier transforms} \label{Sub:FourierDef}

Let $Y$ be a small v-stack and let $\mathcal{X}=(\Lambda, V, \gamma)$ be a locally analytic character datum over $Y$ as in Definition \ref{Def:CharacterDatum}. Recall that 
\[ \mathcal{D}_{\Lambda/Y}^{\gamma-\locan} := \Hom_{\mathcal{O}_Y}(\mathcal{O}_{\Lambda/Y}^{\gamma-\locan}, \mathcal{O}_Y) \textrm{ and } \mathcal{D}_{H_{\mathcal{X}}}= \Hom_{\mathcal{O}_Y}(\mathcal{O}_{H_{\mathcal{X}}/Y}, \mathcal{O}_Y). \]

We now use the universal character $\kappa_{\mathcal{X}}$ to define our Fourier transforms. 
\subsubsection{} For $\pi: \mathcal{H}_{\mathcal{X}} \rightarrow Y$, we define $\mathbb{F}_{\mathcal{X}}: \mathcal{D}_{\Lambda/Y}^{\gamma-\locan}  \rightarrow \mathcal{O}_{H_{\mathcal{X}}/Y}$ to be the composition of the natural maps
\[ \mathcal{D}_{\Lambda/Y}^{\gamma-\locan} \rightarrow \pi_*\pi^*\mathcal{D}_{\Lambda/Y}^{\gamma-\locan}=\pi_* \mathcal{D}_{\Lambda_{H_{\mathcal{X}}}/H_\mathcal{X}}^{\gamma-\locan} \xrightarrow{\mu \mapsto \mu(\kappa_{\mathcal{X}})} \pi_*\mathcal{O}_{H_{\mathcal{X}}}=\mathcal{O}_{H_{\mathcal{X}}/Y}. \]
From the functoriality properties of Lemma \ref{Lem:UniversalKappaAsSection}, it follows that $\mathbb{F}_{\mathcal{X}}$ is a natural transformation between covariant functors of the character datum $\mathcal{X}$. 

\subsubsection{} For $\rho: \Lambda \rightarrow Y$, we define $\mathbb{F}^{\mathcal{X}}_0$ to be the composition of the natural maps (where now we consider $\kappa$ as an element of $\mathcal{O}_{H \times_{Y} \Lambda / \Lambda}(\Lambda)$)
\[ \mathcal{D}_{H_{\mathcal{X}}/Y} \rightarrow \rho_*\rho^* \mathcal{D}_{H_{\mathcal{X}}/Y} = \rho_* \mathcal{D}_{H_{\mathcal{X}} \times_Y \Lambda/\Lambda} \xrightarrow{ \mu \mapsto  \mu(\kappa_{\mathcal{X}})} \rho_*\mathcal{O}_{\Lambda}=\mathcal{O}_{\Lambda/Y}.\]

\begin{Lem}\label{lem.fourier-factorization} 
With notation as above, $\mathbb{F}^\mathcal{X}_0$ factors through a unique map
\[
\mathbb{F}^{\mathcal{X}}: \mathcal{D}_{H_{\mathcal{X}}/Y} \to \mathcal{O}_{\Lambda/Y}^{\gamma-\locan}
\]
via the inclusion $\mathcal{O}_{\Lambda/Y}^{\gamma-\locan} \subseteq \mathcal{O}_{\Lambda/Y}$. 
\end{Lem}
\begin{proof}
Uniqueness follows from the injectivity of $\mathcal{O}_{\Lambda/Y}^{\gamma-\locan} \rightarrow \mathcal{O}_{\Lambda/Y}$ (Lemma \ref{lemma.restriction-map-injective}). Thus, to see that the factorization exists, it suffices to verify it assuming $Y$ is affinoid perfectoid and $V$ is free. We then take subgroups $H_n$ as in Lemma \ref{Lem:GeometryofH}. Applying Lemma \ref{Lem:FrechetDualSheaf}, we find 
\[
 \mathcal{D}_{H_{\mathcal{X}}/Y} = \varinjlim_n \mathcal{D}_{H_{n}/Y}, \]
so it suffices to treat the analogous maps 
\[ \mathcal{D}_{H_n/Y} \rightarrow \rho_*\rho^{*} \mathcal{D}_{H_n/Y} = \rho_* \mathcal{D}_{H_n \times_Y \Lambda/ \Lambda }\xrightarrow{\mu \mapsto \mu(\kappa_{\mathcal{X}}|_{H_n \times_Y \Lambda})} \rho_*\mathcal{O}_{\Lambda}=\mathcal{O}_{\Lambda/Y}.\]
Restricting to any affinoid perfectoid $Y'/Y$ where $\Lambda$ is free of finite rank and $W=\ker \gamma$ is free, Lemma \ref{Lem:UniversalKappaAsSection} implies
\[ \restr{\kappa_{\mathcal{X}}}{Y'} \in C^\locan(\mathbb{Z}_p^h, \mathcal{O}_{H_n/Y}(Y'))^{W=0} \subseteq \Cont(\mathbb{Z}_p^h, \mathcal{O}_{H_n/Y}(Y'))=\mathcal{O}_{H_{n} \times_Y \Lambda/\Lambda}(Y'_\Lambda).\]
Applying a section $\mu$ of $\mathcal{D}_{H_n/Y}(Y')$ is accomplished on the coefficients, so we find that the image of such a $\mu$ lies in 
$C^\locan(\mathbb{Z}_p^h, \mathcal{O}(Y'))^{W=0}$
and thus is $\gamma$-locally analytic by Lemma \ref{Lem:LaAndCauchyRiemann}.
\end{proof}

From the functoriality properties of Lemma \ref{Lem:UniversalKappaAsSection}, it follows that $\mathbb{F}^{\mathcal{X}}$ is a natural transformation of contravariant functors on character data $\mathcal{X}$.

\subsubsection{}
Recall from \S \ref{subsub:Hopf} and Lemma \ref{Lem:GammaLaSheavesHopf} that all of the solid sheaves of $\mathcal{O}_Y$-modules $\mathcal{O}_{\Lambda/Y}^{\gamma\mathrm{-la}}$, $\mathcal{D}_{\Lambda/Y}^{\gamma\mathrm{-la}}$, $\mathcal{O}_{H_{\mathcal{X}}/Y}$, $\mathcal{D}_{H_{\mathcal{X}}/Y}$, are naturally solid Hopf $\mathcal{O}_Y$-algebras.

\begin{Lem} \label{Lem:MorphismOfHopfAlgebras}
    Let $\mathcal{X} = (\Lambda, V, \gamma)$ be a locally analytic character datum. 
    \begin{enumerate}
        \item The Fourier transforms $\mbb{F}_{\mathcal{X}}$ and $\mbb{F}^{\mathcal{X}}$ are morphisms of solid Hopf $\mathcal{O}_Y$-algebras.
        \item The Fourier transforms $\mbb{F}_{\mathcal{X}}$ and $\mbb{F}^{\mathcal{X}}$ are mutually dual under the reflexivity of $\mathcal{O}_{H_\mathcal{X}}$ and $\mathcal{O}_{\Lambda/Y}^{\gamma\mathrm{-la}}$ proved in Proposition \ref{Prop:GeometricFrechet}-(2) and Proposition \ref{Prop:GeometricLB}-(2).
    \end{enumerate}
\end{Lem}
\begin{proof}
    The first part follows from functoriality in the character datum. For example, consider the diagonal morphism
    \[
    \Delta \colon \Lambda \to \Lambda \oplus \Lambda 
    \]
    which induces a morphism of character data  $\mathcal{X} \to \mathcal{X} \oplus \mathcal{X} = (\Lambda \oplus \Lambda, V \oplus V, \gamma \oplus \gamma)$. The linear dual of this map $\Delta^{\ast} \colon \Lambda^{\ast} \oplus \Lambda^{\ast} \to \Lambda^{\ast}$ is the addition map, and induces the multiplication map $H_{\mathcal{X}} \times H_{\mathcal{X}} \to H_{\mathcal{X}}$ via the identification $H_{\mathcal{X}} \times_Y H_{\mathcal{X}} = H_{\mathcal{X} \oplus \mathcal{X}}$. On the other hand, the diagonal map induces the diagonal $Z_{\Lambda,n} \to Z_{\Lambda \oplus \Lambda,n} = Z_{\Lambda,n} \times Z_{\Lambda,n}$, which in turn induces the algebra structure on $\mathcal{O}_{Z_{\Lambda, n}/Y}$ (and hence $\mathcal{O}_{\Lambda/Y}^{\gamma\operatorname{-la}}$). The functoriality of $\mathbb{F}^{\bullet}$ with respect to the diagonal map $\Delta$ now implies that $\mbb{F}^{\mathcal{X}}$ is a morphism of $\mathcal{O}_Y$-algebras. The other assertions about $\mbb{F}_{\mathcal{X}}$ and the comultiplication structures can be proved similarly, and we leave the details to the reader. 

We now establish the second part. Let $\rho: \Lambda \rightarrow Y$ and $\pi: H_{\mathcal{X}} \rightarrow Y$ denote the structure maps. We claim that for $Y'/Y$, $\mu\in \mathcal{D}_{\Lambda}^{\gamma-\locan}(Y')$, and $\delta \in \mathcal{D}_{H_{\mathcal{X}}}(Y')$,  
\begin{equation}\label{eq.order-doesnt-matter-concat} \delta((\pi^*\mu)(\kappa_{\mathcal{X}})) = \mu( (\rho^*\delta)(\kappa_{\mathcal{X}})).  \end{equation}
Indeed, as in the proof of Lemma \ref{lem.fourier-factorization}, it suffices to assume $Y=Y'=\Spd(R,R^+)$ is affinoid perfectoid with $V$ trivial, and to replace $H_{\mathcal{X}}$ with one of the $H_n$'s as in Lemma~\ref{Lem:GeometryofH}. Then, for $U_i$ as in Lemma \ref{Lem:OgammalaIsLB},
\begin{align*} \pi_*\mathcal{O}^{\gamma-\locan}_{\Lambda \times_Y H_n / H_n} &=  \pi_* \varinjlim_i \mathcal{O}_{U_i \times_Y H_n / H_n}	 \\
&=  \varinjlim_i  \left(\pi_* \mathcal{O}_{U_i \times_Y H_n/H_n}\right) \\
&=  \varinjlim_i \left(\mathcal{O}_{U_i \times_Y H_n / Y} \right)\\
&=  \varinjlim_i \left(\mathcal{O}_{U_i / Y} \otimes_{\mathcal{O}_Y}^{\blacksquare} \mathcal{O}_{H_n / Y} \right)\\
&= \left( \varinjlim_i \mathcal{O}_{U_i / Y} \right)\otimes_{\mathcal{O}_Y}^{\blacksquare} \mathcal{O}_{H_n / Y}\\
&=  \mathcal{O}_{\Lambda/Y}^{\gamma\mathrm{-la}} \otimes_{\mathcal{O}_Y}^{\blacksquare} \mathcal{O}_{H_n / Y} 
\end{align*}
where the second equality is because $\pi:H_n \rightarrow Y$ is quasi-compact, the fourth equality is by Lemma \ref{Lem:Kunneth1}, and the fifth equality is because colimits and solid tensor products commute. 
In this presentation, \eqref{eq.order-doesnt-matter-concat} can be unwound to the statement that we can concatenate $\kappa_{\mathcal{X}}|_{\Lambda \times_Y H_n}$, viewed as a section of this sheaf, with $\mu$ and $\delta$ in either order. 

Given \eqref{eq.order-doesnt-matter-concat}, the verification that $\mathbb{F}^{\mathcal{X}}$ and $\mathbb{F}_{\mathcal{X}}$ are mutually dual is essentially formal: for example, we have 
\[ 
(\mathbb{F}_{\mathcal{X}}^*\delta)(\mu) = \delta(\mathbb{F}_{\mathcal{X}}(\mu))=\delta(\pi^*\mu(\kappa_{\mathcal{X}}))=\mu(\rho^*\delta(\kappa_{\mathcal{X}}))=\mu(\mathbb{F}^{\mathcal{X}}(\delta)).
\]
\end{proof}

\begin{Lem} \label{Lem:TheAmiceCase}
If $Y=\spa \qp$ and $\mathcal{X}=(\ul{\zp}, \mathcal{O}_Y, \ul{\zp} \xrightarrow{} \mathcal{O}_Y)$, then $\mathbb{F}_{\mathcal{X}}$ is an isomorphism.
\end{Lem}
\begin{proof}
For $(A,A^+)$ affinoid perfectoid over $(\qp,\zp)$, the natural map
\begin{align}
    H^0_{\cond}(Y, \mathcal{O}_{\gmhateta/Y}) \otimes_{\ul{\qp}}^{\blacksquare} \ul{A} \to H^0_{\cond}(\spd(A,A^+), \mathcal{O}_{\gmhateta/Y}) 
\end{align}
is an isomorphism by Proposition \ref{Prop:GeometricFrechet}-(1). Furthermore, the natural map
\begin{align}
    H^0_{\cond}(Y, \mathcal{D}_{\ul{\zp}/Y}^{\mathrm{la}}) \otimes_{\ul{\qp}}^{\blacksquare}\ul{A} \to H^0_{\cond}(\spd(A,A^+), \mathcal{D}_{\ul{\zp}/Y}^{\mathrm{la}}) 
\end{align}
is an isomorphism by Proposition \ref{Prop:GeometricLB}-(1),(3) and Corollary \ref{Cor:DualTensorFrechetBaseChange}. Thus it suffices to show that $\mathbb{F}_{\mathcal{X}}$ induces an isomorphism on condensed global sections. \smallskip

It follows from Proposition \ref{Prop:GeometricLB}-(3) that 
\begin{align}
    H^0_{\cond}(Y, \mathcal{D}_{\ul{\zp}/Y}^{\mathrm{la}}) = \Hom_{\ul{\qp}}(H^0_{\cond}(Y, \mathcal{O}^{\la}_{\ul{\zp}/Y}), \ul{\qp}),
\end{align}
and from Example \ref{Eg:kappaForTrivialLambda} that we may identify $(\mathbb{F}_{\mathcal{X}})(\ast)$ with the bijection $\mathbb{F}_{\gmhateta}$ of Corollary \ref{cor.amice-bijections}. We conclude\footnote{They are Fr\'echet spaces by Proposition \ref{Prop:FrechetVSClassicalFrechet} together with Corollary \ref{Cor:IsFrechet} and Proposition \ref{Prop:GeometricLB}-(5).} that $F_{\mathcal{X}}(\ast)_{\mathrm{top}}$ is a continuous bijection between two Fr\'echet spaces over $\qp$ and thus an isomorphism by the open mapping theorem for $\mathbb{Q}_p$-Frechet spaces (see \cite[Proposition 8.6]{SchneiderFA}). It follows from Proposition \ref{Prop:FrechetVSClassicalFrechet} that ${\mathbb{F}}_{\mathcal{X}}$ itself is an isomorphism. 
\end{proof}

\subsection{The Fourier transforms are isomorphisms} \label{Sub:MainTheorem}

The main theorem of this article is the following (it is Theorem \ref{Thm:IntroFourierVsheaves} from the introduction). Let $Y$ be a small v-stack and let $\mathcal{X}=(\Lambda,V,\gamma)$ be a locally analytic character datum over $Y$. 
\begin{Thm} \label{Thm:Main}
    The Fourier transforms $\mbb{F}^{\mathcal{X}}$ and 
    $\mbb{F}_{\mathcal{X}}$ are isomorphisms of solid Hopf $\mathcal{O}_Y$-algebras. 
\end{Thm}

By Lemma \ref{Lem:MorphismOfHopfAlgebras}-(1), it suffices to prove that $\mathbb{F}^{\mathcal{X}}$ is an isomorphism, and this is the main goal of the rest of \S\ref{Sub:MainTheorem}. 

\begin{Prop} \label{Prop:InvariantsCoinvariants}
Let $Y$ be a small v-stack and let $\mathcal{X}=(\Lambda,V,\gamma)$ be a locally analytic character datum over $Y$. Let $W=\ker \gamma$. 
\begin{enumerate}
    \item The natural map $\mathcal{O}_{H_{\Lambda}/Y} \to \mathcal{O}_{H_{\mathcal{X}}/Y}$ identifies the target with the quotient of the source by the action of $W$ given by precomposing with multiplication by $\langle w, \operatorname{log}_{H_{\Lambda}} \rangle$.
    \item The natural map $\mathcal{D}_{H_{\mathcal{X}}/Y} \to \mathcal{D}_{H_{\Lambda}/Y}$
identifies the source with the kernel of the dual $W$-action on the target. 
\end{enumerate}
\end{Prop}
\begin{proof}[Proof of Proposition \ref{Prop:InvariantsCoinvariants}]
The second claim follows from the first since $\Hom_{\mathcal{O}_Y}(-,\mathcal{O}_Y)$ turns colimits into limits. For the first claim, we may assume that $Y=\spa (A,A^+)$ such that $V$, $\Lambda$, and $W$ are all trivial of constant rank.  Consider the Cartesian diagram
\begin{equation}
    \begin{tikzcd}
        H_{\mathcal{X}} \arrow{r} \arrow{d}{\log} & H_{\Lambda} \arrow{d}{\log} \\
        V^{\ast} \arrow{r}{\iota} & \Lambda^{\ast} \otimes_{\ul{\zp}} \mathbb{G}_{a,Y}^{\lozenge}.
    \end{tikzcd}
\end{equation}
The ideal sheaf $\mathcal{J} \subset \mathcal{O}_{H_{\Lambda}^{\mathrm{rig}}}$ cutting out $H_{\mathcal{X}}^{\mathrm{rig}}$ is a pseudo-coherent sheaf in the sense of \cite[Section 2.5]{KedlayaLiuII} by\footnote{Note that $H_{\Lambda}^{\mathrm{rig}}$ and $H_{\Lambda}^{\mathrm{rig}}$ are smooth sousperfectoid adic spaces over $Y$, this follows from Lemma \ref{Lem:GeometryofH} together with \cite[Definition 7.1, Lemma 7.5, and Lemma 7.3]{HansenKedlaya}.} \cite[Proposition IV.4.19]{FarguesScholze} and is generated by functionals of the form $\langle w , - \rangle$ for $w \in W(\Spa(A,A^+))$ by \cite[Proposition IV.4.19]{FarguesScholze}. This shows that the natural map $\mathcal{O}(H_{\Lambda}^{\mathrm{rig}}) \to \mathcal{O}(H_{\mathcal{X}}^{\mathrm{rig}})$ factors through the natural map $\mathcal{O}(H_{\Lambda}^{\mathrm{rig}}) \to  \mathcal{O}(H_{\Lambda}^{\mathrm{rig}})_{W}$ and an injection $\mathcal{O}(H_{\Lambda}^{\mathrm{rig}})_{W} \hookrightarrow \mathcal{O}(H_{\mathcal{X}}^{\mathrm{rig}})$; it now suffices to show that $\mathcal{O}(H_{\Lambda}^{\mathrm{rig}}) \to \mathcal{O}(H_{\mathcal{X}}^{\mathrm{rig}})$ is surjective. \smallskip 

The natural map $\mathcal{O}_{H_{\Lambda}^{\mathrm{rig}}} \to \mathcal{O}_{H_{\mathcal{X}}^{\mathrm{rig}}}$ is surjective by \cite[Proposition IV.4.19]{FarguesI}. The surjectivity on global sections follows directly from 
\cite[Theorem 2.6.5.(iii)]{KedlayaLiuII} since $H_{\Lambda}$ is a quasi-stein space in the sense of \cite[Definition 2.6.2]{KedlayaLiuII} by Lemma \ref{Lem:GeometryofH}. We give an alternative argument below for the surjectivity on global sections. \smallskip

Let $k$ be the rank of $W$, and choose a filtration $W=W_0 \supsetneq W_1 \supsetneq \cdots \supsetneq W_k=0$ with free graded pieces, using the freeness of $W$. Set $V_i=\Lambda \otimes_{\zp} \mathcal{O}_Y / W_i$, let $\gamma_i$ be the natural map $\Lambda \otimes_{\zp} \mathcal{O}_Y \to V_i$ and let $\mathcal{X}_i=(\Lambda, V_i, \gamma_i)$. There is a filtration
\begin{align}
    H_{\mathcal{X}}^{\mathrm{rig}}&=H_0 \subset H_1 \subset \cdots \subset H_{k}=H_{\Lambda}^{\mathrm{rig}},
\end{align}
where $H_i=H_{\mathcal{X}_i}^{\mathrm{rig}}$. It remains to show that for all $i$ the map
\begin{align}
    \mathcal{O}(H_i) \to \mathcal{O}(H_{i-1})
\end{align}
is surjective. Write $\mathcal{J}_i \subset \mathcal{O}_{H_i}$ for the ideal sheaf cutting out $H_{i-1}$, which is a line bundle because $H_{i-1}$ has codimension one, see \cite[Proposition IV.4.19]{FarguesScholze} and its proof. Write $H_i = \cup_{n \in \mathbb{Z}_{\ge 0}} H_i^{(n)}$ as in the proof of Lemma \ref{Lem:GeometryofH}. Evaluating the short exact sequence
\begin{align}
    0 \to \mathcal{J}_i \to  \mathcal{O}_{H_i} \to \mathcal{O}_{H_{i-1}} \to 0
\end{align}
on $H_{i}^{(n)}$ gives the short exact sequence (again by \cite[Proposition IV.4.19]{FarguesScholze})
\begin{align}
    0 \to \mathcal{J}_i^{(n)} \to \mathcal{O}(H_i^{(n)}) \to  \mathcal{O}(H_{i-1}^{(n)}) \to 0. 
\end{align}
We now consider the commutative diagram
\begin{equation}
    \begin{tikzcd}
        & 0 \arrow{d} & 0 \arrow{d} & 0 \arrow{d} \\
        0 \arrow{r} & \varprojlim_{n} \mathcal{J}_i^{(n)} \arrow{r} \arrow{d} & \prod_n \mathcal{J}_i^{(n)} \arrow{r} \arrow{d} & \prod_n \mathcal{J}_i^{(n)} \arrow{r} \arrow{d} & 0 \\
        0 \arrow{r} & \varprojlim_n \mathcal{O}(H_i^{(n)}) \arrow{d} \arrow{r} & \prod_{n} \mathcal{O}(H_i^{(n)}) \arrow{r} \arrow{d} & \prod_{n}\mathcal{O}(H_i^{(n)}) \arrow{d} \arrow{r} & 0 \\
        0 \arrow{r} & \varprojlim_n \mathcal{O}(H_{i-1}^{(n)}) \arrow{r} \arrow{d} & \prod_{n} \mathcal{O}(H_{i-1}^{(n)}) \arrow{r} \arrow{d} & \prod_{n} \mathcal{O}(H_{i-1}^{(n)})  \arrow{r} \arrow{d} & 0 \\
         & 0  & 0 & 0. 
    \end{tikzcd}
\end{equation}
Since the transition maps in all three inverse system have dense image by Lemma \ref{Lem:GeometryofH} and the fact that $\mathcal{J}_i$ is a line bundle, the rows are exact by the proof of Proposition \ref{Prop:FrechetFlat}. The middle and right columns are exact by the exactness of countable products, and the exactness of the left column now follows from the snake lemma. Since taking global sections of a sheaf on $H_i$ is the same as taking the inverse limit of the sections on $H_i^{(n)}$, we see that 
\begin{align}
   \mathcal{O}(H_i) \to \mathcal{O}(H_{i-1})
\end{align}
is surjective, concluding the proof. 
\end{proof}

\begin{proof}[Proof of Theorem \ref{Thm:Main}]
By Lemma \ref{Lem:MorphismOfHopfAlgebras}, the Fourier transforms are dual morphisms of solid Hopf $\mathcal{O}_Y$-algebras, so it suffices to show that $\mbb{F}^{\mathcal{X}}$ is an isomorphism of abelian sheaves. The main idea is to reduce to a statement about the classical Amice transform. More precisely, the claim is local on $Y$, so we may assume that $Y = \operatorname{Spa}(R, R^+)$ is affinoid perfectoid, $\Lambda \cong \underline{\mbb{Z}}_p^{\oplus h}$ is trivial, and $W$ is free. In this setting, consider the locally analytic character datum $\mathcal{X}'=(\Lambda, \Lambda \otimes_{\ul{\zp}} \mathcal{O}_Y, \mathrm{id}_{\Lambda \otimes_{\ul{\zp}} \mathcal{O}_Y})$. We of course have a morphism $\mathcal{X} \to \mathcal{X}'$, which induces a morphism $H_{\mathcal{X}} \to H_{\mathcal{X}'} = H_{\Lambda} = \widehat{\mathbb{G}}_{m,\eta, Y}^{\lozenge, \oplus h}$. By Proposition \ref{Prop:InvariantsCoinvariants} and Lemma \ref{Lem:LaAndCauchyRiemann}, we see that one can obtain $\mbb{F}^{\mathcal{X}}$ by passing to the $W$-invariants of 
    \[
    \mbb{F}^{\mathcal{X}'} \colon \mathcal{D}_{H_{\Lambda}/Y} \to \mathcal{O}_{\Lambda/Y}^{\mathrm{la}} .
    \]
    It therefore suffices to show that $\mbb{F}^{\mathcal{X}'}$ is an isomorphism. Using the K\"unneth formulas for $\mathcal{O}^{\operatorname{la}}_{\Lambda/Y}$, see Proposition \ref{Prop:GeometricLB}-(4), and for $\mathcal{D}_{H_{\Lambda}/Y}$, see Proposition \ref{Prop:GeometricFrechet}-(8), we deduce that it is enough to prove the theorem for $\Lambda=\ul{\zp}$. Since $\Lambda=\ul{\zp}$ is base-changed from $\spd(\qp,\zp)$, we may assume that $Y=\spd(\qp,\zp)$. Using Lemma \ref{Lem:MorphismOfHopfAlgebras}, we can equivalently show that the dual map
\begin{align}
    \mbb{F}_{\mathcal{X}'} \colon \mathcal{D}_{\ul{\zp}/Y}^{\mathrm{la}} \to \mathcal{O}_{\gmhateta/Y}
\end{align}
is an isomorphism. and then the result is Lemma \ref{Lem:TheAmiceCase}. This concludes the proof.
\end{proof}

\subsubsection{}\label{sss.solid-hopf-and-duality-discussion-fiercely} We now deduce a version of Theorem \ref{thm.rational-fourier-theory} over a fiercely v-complete affinoid. 

Let $(R,R^+)$ be a fiercely v-complete Huber pair over $(\qp,\zp)$, and let $Y=\spd(R,R^+)$. Let $\mathcal{X}=(\Lambda, V, \gamma)$ be a locally analytic character datum over $Y$ such that $V$ is free. In this case, the solid Hopf $\mathcal{O}_Y$-algebra structures on $\mathcal{O}_{H_{\mathcal{X}}/Y}$, $\mathcal{D}_{H_{\mathcal{X}}/Y},$ $\mathcal{O}_{\Lambda/Y}^{\gamma\mathrm{-la}}$, and $\mathcal{D}_{\Lambda/Y}^{\gamma\mathrm{-la}}$ induce solid Hopf $\ul{R}$-algebra structures upon taking global sections:
\begin{itemize}
    \item For $H^0_{\cond}(Y,\mathcal{O}_{H_{\mathcal{X}}/Y})$, this follows from Proposition \ref{Prop:GeometricFrechet}-(2),(7).

    \item For $H^0_{\cond}(Y,\mathcal{D}_{H_{\mathcal{X}}/Y})$, this follows from Proposition \ref{Prop:GeometricFrechet}-(4),(7),(8). 

    \item For $H^0_{\cond}(Y,\mathcal{O}_{\Lambda/Y}^{\gamma\mathrm{-la}})$, this follows from Proposition \ref{Prop:GeometricLB}-(2),(8).

    \item For $H^0_{\cond}(Y,\mathcal{D}_{\Lambda/Y}^{\gamma\mathrm{-la}})$, this follows from Proposition \ref{Prop:GeometricLB}-(4),(8),(9). 
\end{itemize}    
Moreover, we find the solid Hopf $\ul{R}$-algebras $H^0_{\cond}(Y,\mathcal{O}_{H_{\mathcal{X}}/Y})$ and $H^0_{\cond}(Y,\mathcal{D}_{H_{\mathcal{X}}/Y})$ are mutually dual Hopf $\ul{R}$-algebras by further invoking Proposition \ref{Prop:GeometricFrechet}-(3), and similarly $H^0_{\cond}(Y,\mathcal{O}_{\Lambda/Y}^{\gamma\mathrm{-la}})$ and $H^0_{\cond}(Y,\mathcal{D}_{\Lambda/Y}^{\gamma\mathrm{-la}})$ are mutually dual Hopf $\ul{R}$-algebras by Proposition \ref{Prop:GeometricLB}-(3).

\begin{Cor} \label{Cor:Main} 
Let $(R,R^+)$ be a fiercely v-complete Huber pair over $(\qp,\zp)$, let $Y=\spd(R,R^+)$ and let $\mathcal{X}=(\Lambda, V, \gamma)$ be a locally analytic character datum over $Y$. If $V$ is free, then the Fourier transforms
 \begin{align*}
        \mbb{F}^{\mathcal{X}} &\colon H^0_{\cond}(Y,\mathcal{D}_{H_{\mathcal{X}}/Y}) \to H^0_{\cond}(Y,\mathcal{O}_{\Lambda/Y}^{\gamma\mathrm{-la}})  \\
        \mbb{F}_{\mathcal{X}} &\colon H^0_{\cond}(Y,\mathcal{D}_{\Lambda/Y}^{\gamma\mathrm{-la}}) \to H^0_{\cond}(Y,\mathcal{O}_{H_{\mathcal{X}}/Y})
    \end{align*}
    are dual isomorphisms of solid $\ul{R}$-Hopf algebras. 
\end{Cor}
\begin{proof}
In light of the above discussion, this is immediate from Theorem \ref{Thm:Main} by taking global sections. 
\end{proof}

Specializing Corollary \ref{Cor:Main} to $(K, \mathcal{O}_K)$ for $K/\mathbb{Q}_p$ a $p$-adic non-archimedean extension gives Theorem~\ref{thm.rational-fourier-theory}. 

\begin{Rem}\label{remark.etale-expectations}
    We expect the computations of \S\ref{sss.solid-hopf-and-duality-discussion-fiercely} and thus also Corollary \ref{Cor:Main} to hold whenever $V$ is an \'{e}tale (equivalently, analytic) vector bundle. 
\end{Rem}

\subsection{Equivariance under invariant derivations} \label{Sub:Equivariance} We now establish the equivariance properties for the Fourier transforms $\mbb{F}^{\mathcal{X}}$ and $\mbb{F}_{\mathcal{X}}$ for invariant derivations generalizing those appearing in Proposition \ref{Prop:Equivariance}. 

\subsubsection{}
Recall from \S\ref{subsub:derivationaction} that there is a natural action of $V$ on $\mathcal{O}^{\gamma-\locan}_{\Lambda/Y}$ by invariant derivations. It extends to an action of $\Sym^\bullet V$ (by invariant differential operators; note the Lie bracket is trivial). We equip $\mathcal{D}_{\Lambda/Y}^{\gamma\mathrm{-la}}$ with the dual action of $\operatorname{Sym}^{\bullet}V$. These actions are evidently functorial in $\mathcal{X}$. 

There is a natural map $\log_{H_\mathcal{X}}: H_{\mathcal{X}} \rightarrow V^*$, so $\mathcal{O}_{H_{\mathcal{X}}/Y}$ is naturally a $\Sym^\bullet V$-algebra by pullback along $\log_{H_\mathcal{X}}$ of the polynomial functions $\Sym^\bullet V \subseteq \mathcal{O}_{V^*/Y}$. We equip $\mathcal{D}_{H_{\mathcal{X}}/Y}$ with the dual action of $\Sym^\bullet V$. These actions are evidently functorial in $\mathcal{X}$.

\begin{Prop} \label{Prop:FTequivSymV}
    The Fourier transforms $\mbb{F}^{\mathcal{X}}$ and $\mbb{F}_{\mathcal{X}}$ are equivariant for the above actions of $\operatorname{Sym}^{\bullet}V$.
\end{Prop}
\begin{proof}
By the duality of Lemma \ref{Lem:MorphismOfHopfAlgebras}, it suffices to treat $\mbb{F}_{\mathcal{X}}$. For $v$ a section of $V$ and $\mu$ a section of $\mathcal{D}^{\gamma-\locan}_{\Lambda/Y}$, 
\[ \mbb{F}_{\mathcal{X}}(v \cdot \mu)= (v\cdot \mu)(\kappa_{\mathcal{X}})=\mu(v \cdot \kappa_{\mathcal{X}}) = \mu(\langle v, \log_{H_{\mathcal{X}}}\rangle \kappa_\mathcal{X}) \]
where the final equality is deduced from Lemma \ref{Lem.DerivationActionLocAn}-(1) and functoriality in $\mathcal{X}$. This is equal to 
$\langle v, \log_{H_{\mathcal{X}}}\rangle\mu(\kappa_\mathcal{X})$
because, by the construction of $\mathbb{F}_{\mathcal{X}}$, we are applying $\mu$ after pulling it back to an element of $D^{\gamma-\locan}_{\Lambda_{H_{\mathcal{X}}}/H_{\mathcal{X}}}$, so that it is $\mathcal{O}_{H_{\mathcal{X}}}$-linear.  
\end{proof}

\subsubsection{}
Since $\operatorname{Lie}H_{\mathcal{X}}$ is identified with $V^{\ast}$, we obtain an action of $V^{\ast}$ on $\mathcal{O}_{H_{\mathcal{X}}/Y}$ through invariant derivations. It extends to an  action of $\operatorname{Sym}^{\bullet}V^{\ast}$ (by invariant differential operators; note the Lie bracket is trivial). We equip $\mathcal{D}_{H_{\mathcal{X}}/Y}$ with the dual action of $\operatorname{Sym}^{\bullet}V^{\ast}$. These actions are evidently functorial in $\mathcal{X}$. 

There is a natural map $\gamma|_{\Lambda}: \Lambda \rightarrow V$, which induces by pullback an algebra map $\mathcal{O}_{V/Y} \rightarrow \mathcal{O}_{\Lambda/Y}$. By definition, it factors through $\mathcal{O}_{\Lambda/Y}^{\gamma-\locan}$. In particular, restricting to the polynomial functions $\Sym^{\bullet} V^{\ast} \subseteq  \mathcal{O}_{V/Y}$, we obtain a $\Sym^\bullet V^{\ast}$-module structure on $\mathcal{O}^{\gamma-\locan}_{\Lambda/Y}$. We equip $\mathcal{D}^{\gamma-\locan}_{\Lambda/Y}$ with the dual action of $\Sym^\bullet V^{\ast}$. These actions are evidently functorial in $\mathcal{X}$.

\begin{Prop} \label{Prop:FTSymVstarEquiv}
    The Fourier transforms $\mbb{F}^{\mathcal{X}}$ and $\mbb{F}_{\mathcal{X}}$ are equivariant for the actions of $\operatorname{Sym}^{\bullet}V^{\ast}$ described above.
\end{Prop}
\begin{proof}
By the duality of Lemma \ref{Lem:MorphismOfHopfAlgebras}, it suffices to treat $\mathbb{F}_\mathcal{X}$. For $v'$ a section of $V^*$ and $\mu$ a section of $\mathcal{D}^{\gamma-\locan}_{\Lambda/Y}$, 
\[ \mbb{F}_{\mathcal{X}}( v' \cdot \mu)= (v' \cdot \mu)(\kappa_{\mathcal{X}})=\mu( (\gamma^*v') \kappa_{\mathcal{X}}). \]
On the other hand, 
\[ v' \cdot \mathbb{F}_\mathcal{X}(\mu)  = v' \cdot \mu(\kappa_{\mathcal{X}})=\mu(v' \cdot \kappa_{\mathcal{X}}) \]
where the final equality holds because the differentiation is happening in the coefficients direction. By functoriality, the result then follows from Lemma \ref{Lem.DerivationActionLocAn}-(2). 
\end{proof}

Specialising Propositions \ref{Prop:FTequivSymV} and \ref{Prop:FTSymVstarEquiv} to $Y = \operatorname{Spd}(K, \mathcal{O}_K)$ completes the proof of Proposition \ref{Prop:Equivariance} in the introduction.

\subsection{Classical interpretation over a field} \label{Sub:Field} Let $K$ be a non-archimedean field with completed algebraic closure $C$. We will now explain how to interpret Corollary \ref{Cor:Main} for $Y=\spa(K,\mathcal{O}_K)$ in terms of locally convex topological vector spaces over $K$. 

\subsubsection{} Let $\mathcal{X}=(\Lambda,V,\gamma)$ be a locally analytic character datum and suppose that $V$ is trivial (equivalently, an \'etale vector bundle). Then concretely, we can think of $\Lambda$ as a continuous representation of $\mathrm{Gal}(\overline{K}/K)$ on a finite free $\mathbb{Z}_p$-module $\Lambda$, of $V$ as a finite dimensional $K$-vector space, and of $\gamma: \Lambda_C \twoheadrightarrow V_{C}$ as a $\mathrm{Gal}(\overline{K}/K)$-equivariant surjection. Let $\mathcal{O}^{\gamma-\locan}(\Lambda)$ be the locally convex vector space of locally analytic functions from $\Lambda$ to $C$ that are $\Gal(\overline{K}/K)$-equivariant and $\gamma$-locally analytic, i.e., locally on $\Lambda$, factor as the composition of $\gamma$ with an analytic function on an open ball in $V_C$. We let $\mathcal{D}^{\gamma-\locan}(\Lambda)$ be its dual equipped with the strong topology, see Section \ref{subsub:weakdual}.

Let $H_{\mathcal{X}}$ be the $p$-divisible v-group over $\spd(K, \mathcal{O}_K)$ associated with $\mathcal{X}$. Since $V$ is an \'etale vector bundle, this is the diamond associated to a $p$-divisible rigid analytic group over $\spa(K, \mathcal{O}_K)$ by Proposition \ref{Prop:Representable}, which we will denote by $H_{\mathcal{X}}^{\mathrm{rig}}$. Let us write $\mathcal{O}(H_{\mathcal{X}}^{\mathrm{rig}})$ for the locally convex $K$-Fr\'echet space of global sections. 
\begin{Lem} \label{Lem:LocallyConvexVSCondensed}
    There are natural isomorphisms
    \begin{align}
        H^0_{\cond}(Y,\mathcal{O}_{H_{\mathcal{X}}/Y}) &\xrightarrow{\sim} \ul{\mathcal{O}(H_{\mathcal{X}}^{\mathrm{rig}})} \\
         H^0_{\cond}(Y,\mathcal{D}_{\Lambda/Y}^{\gamma-\locan}) &\xrightarrow{\sim}\ul{\mathcal{D}^{\gamma\mathrm{-la}}(\Lambda)}.
    \end{align}
\end{Lem}
\begin{proof}
For the second isomorphism, we first recall from \S\ref{sss.solid-hopf-and-duality-discussion-fiercely}
 that
\begin{align}
     H^0_{\cond}(Y,\mathcal{D}_{\Lambda/Y}^{\gamma-\locan}) &= \Hom_{\ul{K}}(H^0_{\cond}(Y,\mathcal{O}_{\Lambda/Y}^{\gamma-\locan}), \ul{K}).
\end{align}
We now continue in the notation of Lemma \ref{Lem:OgammaisLSspace}
\begin{align}
     \Hom_{\ul{K}}(H^0_{\cond}(Y,\mathcal{O}_{\Lambda/Y}^{\gamma-\locan}), \ul{K})&=\Hom_{\ul{K}}(H^0_{\cond}(Y, \varinjlim_{n} \mathcal{O}^{\mathrm{bdd}}_{U_{n}^{\circ}/Y}), \ul{K}) \\
     &=\Hom_{\ul{K}}(\varinjlim_{n}  H^0_{\cond}(Y, \mathcal{O}^{\mathrm{bdd}}_{U_{n}^{\circ}/Y}), \ul{K}) \\
     &=\varprojlim_{n} \Hom_{\ul{K}}(H^0_{\cond}(Y, \mathcal{O}^{\mathrm{bdd}}_{U_{n}^{\circ}/Y}), \ul{K}) \\
     &=\varprojlim_{n} \Hom_{\ul{K}}(\ul{\mathcal{O}^{\mathrm{bdd}}(U_n^{\circ})}, \ul{K}),
\end{align}
where the last equality is Lemma \ref{Lem:BoundedFunctions}. We can similarly write ($\mathrm{lcv}$ stands for locally convex) 
\begin{align}
    \mathcal{D}^{\gamma\mathrm{-la}}(\Lambda) &= \left(\mathcal{O}^{\gamma\mathrm{-la}}(\Lambda)\right)'_{b} \\
    &=\left(\varinjlim_n^{\mathrm{lcv}} \mathcal{O}^{\mathrm{bdd}}(U_n^{\circ})\right)'_{b} \\
    &=\varprojlim_n \left(\mathcal{O}^{\mathrm{bdd}}(U_n^{\circ})\right)'_{b}. 
\end{align}
It is a consequence of Lemma \ref{Lem:TopologicalDuals} that 
\begin{align}
   \ul{\left(\mathcal{O}^{\mathrm{bdd}}(U_n^{\circ})\right)'_b} \simeq \Hom_{\ul{K}}(\ul{\mathcal{O}^{\mathrm{bdd}}(U_n^{\circ})}, \ul{K}).
\end{align}
This, together with the commutation of $W \mapsto \ul{W}$ with limits, concludes the proof of the second isomorphism. 

For the first isomorphism, we write (in the notation of Lemma \ref{Lem:GeometryofH})
\begin{align}
    H^0_{\cond}(Y,\mathcal{O}_{H_{\mathcal{X}}/Y}) = \varprojlim_n H^0_{\cond}(Y,\mathcal{O}_{H_{n}/Y}) = \varprojlim_n \ul{\mathcal{O}(H_n)},
 \end{align}
using Lemma \ref{Lem:CondensedRingStructure}. We similarly have
\begin{align}
    \mathcal{O}(H_{\mathcal{X}}^{\mathrm{rig}})= \varprojlim_n \mathcal{O}(H_n),
\end{align}
and we conclude as before. 
\end{proof}

\subsubsection{} We have the following corollary of Theorem \ref{thm.rational-fourier-theory}.
\begin{Cor} \label{Cor:LocallyConvexMain}
    There is an isomorphism of locally convex $K$-vector spaces
    \begin{align}
        \mathcal{D}^{\gamma\mathrm{-la}}(\Lambda) &\xrightarrow{\sim} \mathcal{O}(H_{\mathcal{X}}^{\mathrm{rig}}) \\
    \end{align}
    functorial in $(\Lambda, V, \gamma)$.
\end{Cor}
\begin{proof}
Lemma \ref{Lem:LocallyConvexVSCondensed} and Theorem \ref{thm.rational-fourier-theory} give an isomorphism
\[ \underline{\mathcal{D}^{\gamma\mathrm{-la}}(\Lambda)} \xrightarrow{\sim} \underline{\mathcal{O}(H_{\mathcal{X}}^{\mathrm{rig}})} \]
and then we conclude by Proposition \ref{Prop:FrechetVSClassicalFrechet}. 
\end{proof}

\begin{Rem} \label{Rem:WhyCondense}
It is possible to upgrade Corollary \ref{Cor:LocallyConvexMain} to include Hopf algebra structures (with respect to the completed projective tensor product of locally convex $K$-vector spaces). Indeed, since $V \mapsto \ul{V}$ commutes with products, we see that there are natural commutative algebra structures on the two locally convex $K$-vector spaces in the statement of Corollary \ref{Cor:LocallyConvexMain}. Moreover, the solid Hopf $K$-algebra structures on the corresponding condensed $K$-vector spaces, see Lemma \ref{Lem:LocallyConvexVSCondensed}, induce locally convex Hopf $K$-algebra structures on the two locally convex $K$-vector spaces in the statement of Corollary \ref{Cor:LocallyConvexMain}. Indeed, this follows from \cite[Proposition A.68]{Bosco}. 
\end{Rem}

\subsection{Comparison with Schneider--Teitelbaum} \label{Sub:STComparison} 
Let us explain how Corollary \ref{Cor:LocallyConvexMain} compares with the work of Schneider--Teitelbaum \cite{SchneiderTeitelbaumFourier}. Suppose that we have embeddings $\mbb{Q}_p \subset L \subset K \subset \mathbb{C}_p$, where $K$ is complete and $L/\mbb{Q}_p$ is finite, and take $Y =\operatorname{Spa}(K, \mathcal{O}_K)$. Let $\iota_0 \colon L \hookrightarrow \mathbb{C}_p$ denote this fixed embedding. We will consider the locally analytic character datum $\mathcal{X}=(\Lambda, V, \gamma)$ where $\Lambda=\mathcal{O}_L$ has the trivial Galois action, $V=K$, and, for $\iota_0: L \hookrightarrow \mathbb{C}_p$ the given embedding, $\gamma$ is the map
\begin{align}
    \gamma=\iota_0|_{\mathcal{O}_L} \otimes \mathrm{id}_{\mathbb{C}_p}:\mathcal{O}_L \otimes_{\zp} \mathbb{C}_p \rightarrow  \mathbb{C}_p= K \otimes_K \mathbb{C}_p.
\end{align}
In particular, $H_{\mathcal{X}}^{\mathrm{rig}}$ is a 1-dimensional $p$-divisible rigid analytic group.

\subsubsection{} We recall the notation of \cite[Section 1]{SchneiderTeitelbaumFourier}. Let $G=\mathcal{O}_L$ considered as an $L$-analytic manifold, and $G_0=\mathcal{O}_L$ considered as a $\qp$-analytic manifold. Schneider and Teitelbaum consider a quotient map
\begin{align}
    D(G_0,K) \to D(G,K)
\end{align}
of $K$-Fr\'echet spaces, where $D(-,K)$ denotes the strong dual to the locally convex $K$-vector space of locally analytic $K$-valued functions on $-$. 
\begin{Lem} \label{Lem:DistributionST}
    There is a commutative diagram of topological $K$-modules
\begin{equation} 
    \begin{tikzcd}
             D(G_0, K) \arrow[r] \arrow{d} & D(G, K) \arrow{d} \\
             \mathcal{D}^{\locan}(\Lambda) \arrow[r] & \mathcal{D}^{\gamma-\locan}(\Lambda)
    \end{tikzcd}
\end{equation}
where the vertical maps are isomorphisms.
\end{Lem}
\begin{proof}
The left vertical arrow is an isomorphism by definition. To show that this induces an isomorphism on the right vertical arrows, we need to check that our $\gamma$-locally analytic condition agrees with the ``$L$-linearity'' condition of Schneider--Teitelbaum. Using Proposition \ref{Prop:InvariantsCoinvariants}, this follows from the observation that $W=\ker \gamma$ is generated by elements of the form $\lambda \otimes x - 1 \otimes \iota_0(\lambda) x$ for $\lambda \in \Lambda=\mathcal{O}_K$ and $x \in \mathbb{C}_p$; this matches the ``$L$-linearity'' condition in \cite[Lemma 1.1]{SchneiderTeitelbaumFourier}. 
\end{proof}

\subsubsection{} Let $\widehat{G}_0$ be the smooth rigid analytic group over $K$
whose $\mathbb{C}_p$-points are the continuous characters $\chi:\mathcal{O}_L \to \mathcal{O}_{\mathbb{C}_p}^{\times}$, and write $\widehat{G} \subset \widehat{G}_0$ for the closed subgroup whose $\mathbb{C}_p$ points are the $L$-analytic characters $\chi$, see \cite[Corollary 1.5]{SchneiderTeitelbaumFourier}. 
\begin{Lem} \label{Lem:CharacterVarieties}
    There is a commutative diagram 
    \begin{equation}
    \begin{tikzcd}
         \widehat{G} \arrow[r, hook] \arrow{d}& \widehat{G}_0 \arrow{d} \\
        H_{\mathcal{X}}^{\mathrm{rig}} \arrow[r, hook] & \Hom(\Lambda, \gmhateta),
    \end{tikzcd}
    \end{equation}
    where the vertical arrows are isomorphisms.
\end{Lem}
\begin{proof}
The existence of the right vertical isomorphism is straightforward. That it identifies $\widehat{G}$ and $H_{\mathcal{X}}^{\mathrm{rig}}$ can be checked on Lie algebras, where it follows from \cite[Corollary 1.5]{SchneiderTeitelbaumFourier}. 
\end{proof}

\subsubsection{} Schneider and Teitelbaum prove, see \cite[Theorem 2.2, Theorem 2.3]{SchneiderTeitelbaumFourier}, that there is a commutative diagram of locally convex $K$-vector spaces (they write $\mathcal{F}$ for what we call $\mathcal{F}_0$)
\begin{equation} \label{Eq:STFourier}
    \begin{tikzcd}
        D(G_0, K) \arrow{r} \arrow{d}{\mathcal{F}_0} & D(G,K) \arrow{d}{\mathcal{F}} \\
        \mathcal{O}(\widehat{G_0}) \arrow{r} & \mathcal{O}(\widehat{G}), 
    \end{tikzcd}
\end{equation}
where the vertical maps are isomorphisms and where the left vertical map is the (multidimensional) Amice transform.

\begin{Prop} \label{Prop:Comparison}
    Under the isomorphisms of Lemma \ref{Lem:CharacterVarieties} and Lemma \ref{Lem:DistributionST}, the isomorphism $\mathcal{F}$ agrees with the map $\mathbb{F}_{\mathcal{X}}$ of Corollary \ref{Cor:LocallyConvexMain}.
\end{Prop}
\begin{proof}
By Example \ref{Eg:kappaForTrivialLambda}, this is true for $\mathcal{F}_0$ because our Fourier transform agrees with the Amice transform for $H=\gmhateta$ and thus with the multidimensional Amice transform for $(\gmhateta)^h$. The result follows for $\mathcal{F}$, because of the functoriality of our Fourier transform, the commutativity of \eqref{Eq:STFourier} and the surjectivity of $D(G_0, K) \to D(G, K)$.
\end{proof}

\section{Integral Fourier theory}\label{s.integral-fourier}

\newcommand{\oh}{\mathcal{O}(\mathcal{H})}
Let $R$ be a $p$-adically complete and separated ring. In this section, we are going to prove an integral analog of Theorem \ref{thm.rational-fourier-theory} for $p$-divisible groups over $R$. In nice situations, we will also compare our integral theory with the rational theory on the generic fiber. 

\subsubsection{} We give the statement of our main result in this context. Recall that given a $p$-divisible group $\mathcal{H}$ over $\spf R$, there is a Serre-dual $p$-divisible group $\mathcal{H}^{\vee}$ and a Hodge--Tate map $T_p \mathcal{H}^{\vee} \to \omega_{\mathcal{H}}$, see for example Section \ref{subsub:HodgeTate} below. For $\mathcal{H}$ a $p$-divisible group over $\spf R$, we show that $\mathcal{O}(\mathcal{H})$ and $\mathcal{O}(T_p \mathcal{H}^\vee)$, along with their continuous $R$-linear duals (equipped with the weak topology), are commutative topological Hopf $R$-algebras\footnote{By which we mean a co-commutative co-group object in the category of complete linearly topologized $R$-algebras with the completed tensor product of linearly topologized $R$-modules as in \cite[Tag 0AMU]{stacks-project}.}, and there is a universal character\footnote{See Section \ref{subsub:StructureSheaf} for our conventions about evaluating $\mathcal{O}$ on ind-schemes.} 
\[ \kappa \in \mathcal{O}(T_p \mathcal{H}^\vee) \; \widehat{\otimes}_R \; \mathcal{O}(\mathcal{H}). \]
Building up from Cartier duality for the finite flat group schemes $\mathcal{H}[p^n]_{R/p^k}$, we prove

\begin{Thm}\label{thm.integral-fourier-theory}
    Integration against the universal character $\kappa$ defines functorial and naturally dual isomorphisms of complete linearly topologized Hopf $R$-algebras
    \[ \mathbb{F}^{\mathcal{H}}: \mathcal{O}(\mathcal{H})^{\ast} \xrightarrow{\sim} \mathcal{O}(T_p \mathcal{H}^\vee) \textrm{ and } \mathbb{F}_{\mathcal{H}}: \mathcal{O}(T_p \mathcal{H}^\vee)^{\ast} \xrightarrow{\sim} \mathcal{O}(\mathcal{H}). \]
\end{Thm}
We equip $\mathcal{O}(\mathcal{H})$ with the action of $\Sym^\bullet \Lie \mathcal{H}$ induced by the action of $\Lie \mathcal{H}$ by differentiation along invariant vector fields, and on $\mathcal{O}(\mathcal{H})^{\ast}$ we take the dual action. We equip $\mathcal{O}(T_p \mathcal{H}^\vee)$, with the action of $\mathrm{Sym}^\bullet \Lie \mathcal{H}=\mathcal{O}(\omega_{\mathcal{H}})$ through multiplication by functions pulled back along the Hodge--Tate map, and on $\mathcal{O}(T_p \mathcal{H}^\vee)^{\ast}$ we take the dual action. With these actions, we obtain an analogue of Proposition \ref{Prop:Equivariance}.

\begin{Prop}\label{Prop:IntegralEquivariance}
The Fourier transforms of Theorem \ref{thm.integral-fourier-theory} are equivariant for the actions of $\Sym^\bullet \operatorname{Lie} \mathcal{H}$  described above. 
\end{Prop}

\begin{Rem} \label{remark.integral-artier-duality}
If $\mathcal{H}$ is represented by a formal group, then the natural map
\begin{align}
    \mathrm{Sym}^\bullet \Lie \mathcal{H} \to \mathcal{O}(T_p \mathcal{H}^\vee)
\end{align}
induced by pulling back along the Hodge--Tate map extends to an isomorphism $\operatorname{Dif}(\mathcal{H}) \xrightarrow{\sim} \mathcal{O}(T_p \mathcal{H}^\vee)$, where $\operatorname{Dif}(\mathcal{H})$ is the $p$-adic completion of the ring of invariant differential operators on $\mathcal{H}$. This is due to Cartier, see \cite[Theorem 1.56]{ZinkCartier}.
\end{Rem}

\begin{Eg}\label{example.mahler}
    For $\mathcal{H}=\widehat{\mbb{G}}_m = \varinjlim_n \mu_{p^n}$, we have $\mathcal{H}^\vee=\mbb{Q}_p/\mbb{Z}_p= \varinjlim \frac{1}{p^n}\mbb{Z}/\mbb{Z}$ and $T_p \mathcal{H}^\vee=\ul{\mathbb{Z}_p}=\spf \Cont(\mbb{Z}_p, R)$. The functions on $\mathcal{H} \times T_p \mathcal{H}^\vee$ are naturally identified with
    \[ \Cont(\mbb{Z}_p, R) \widehat{\otimes}_{R} R[[t]] = \Cont(\mbb{Z}_p, R)[[t]] \]
    and the universal character $\kappa$ is the element
    \[ z \mapsto (1+t)^z = \sum_{k \geq 0} \binom{z}{k} t^k \textrm{ where } \binom{z}{k}:=\frac{z(z-1)\ldots(z-k+1)}{k!}.\]
    For $\mu \in \Cont(\mathbb{Z}_p, R)^{\ast} \cong \hom_{\mathbb{Z}_p}(\Cont(\mathbb{Z}_p, \mathbb{Z}_p), R)$, i.e., an $R$-valued measure on $\mathbb{Z}_p$, the Fourier transform $\mathbb{F}_{\mathcal{H}}$ sends $\mu$ to 
    \begin{equation}\label{eq.mahler-character-expansion} \mu(z \mapsto (1+t)^z) = \sum_{k \geq 0} \mu\left(\binom{z}{k}\right) t^k \in R[[t]]. \end{equation}
    That this is an isomorphism is equivalent to the claim that the binomial coefficients $\left(\binom{z}{k}\right)_{k \geq 0}$ are a Hilbert basis for $\Cont(\mathbb{Z}_p, \mathbb{Z}_p)$, as first established in \cite{Mahler}. Under this isomorphism, multiplication by the coordinate function $z$ on $\mathbb{Z}_p$ (i.e., the algebra map $\mathbb{Z}_p \hookrightarrow R$) is matched with precomposition with the invariant derivation $(1+t) \partial_t$, as one readily verifies in the formula by applying the chain rule to compute 
    \[ (1+t)\partial_t \left((1+t)^z\right) = (1+t) (z (1+t)^{z-1})=z(1+t)^z. \]
\end{Eg}

\begin{Eg}
In Example \ref{example.mahler}, $\Cont(\mathbb{Z}_p, R)$ has the $p$-adic topology so the linear dual is equal to the continuous dual. When one writes down the dual Fourier transform $\mathbb{F}_{\mathcal{H}}$ in the same setting, it is important to keep track of the correct topology already mod $p$ when computing the dual. Indeed, $R[[t]]$ is equipped with the $(p,t)$-adic topology, so that the elements of $R[[t]]^{\ast}$ are the $R$-linear functionals $\mu$ satisfying $\mu(t^k) \rightarrow 0$ as $k \rightarrow \infty$; this is necessary to have that $\sum_{k \geq 0} \binom{z}{k} \mu(t^k)$ converges in $\Cont(\mathbb{Z}_p, R)$ so the Fourier transform is well-defined.
\end{Eg}

\begin{Rem}
As in Remark \ref{remark.locally-compact-analog}, it is natural to look for a more general category of commutative group functors on nilpotent $R$-algebras that is preserved under Cartier duality and contains both $p$-divisible groups and their duals, and to extend Theorem \ref{thm.integral-fourier-theory} to this setting. Any reasonable such category should also be closed under extensions; for example, just as the category of locally compact abelian groups contains the extension $\mathbb{R}$ of the compact group $\mathbb{R}/\mathbb{Z}=S^1$ by the discrete group $\mbb{Z}$, this category should also contain the universal cover $\widetilde{\mathcal{H}}=\varprojlim_{[p]} \mathcal{H}$ of any $p$-divisible group $\mathcal{H}$, which is an extension of $\mathcal{H}$ by $T_p \mathcal{H}$. In the integral case, one can deduce Fourier isomorphisms for universal covers from the Fourier isomorphisms for $p$-divisible groups; this will be explained in \cite{GHHBC}. From this more general perspective, the compatibility for the action of $\mathrm{Sym}^\bullet \Lie \mathcal{H}$ is closely related to the functoriality of Cartier duality for the inclusion into $\mathcal{H}$ of the divided powers neighborhood of the identity (which is naturally identified with the divided powers neighborhood of the identity in $\Lie \mathcal{H}$). 
\end{Rem}

Theorem \ref{thm.integral-fourier-theory} and Proposition \ref{Prop:IntegralEquivariance} are proved in \S \ref{ss.integral-fourier-proof}. In \S \ref{ss.integral-rational-comparison} we establish, for $(R,R^+)$ a fiercely v-complete Huber pair, a compatibility between the integral Fourier theory over $R^+$ and the rigid analytic Fourier theory on the generic fiber (Theorem \ref{thm.integral-rational-compatibility}). 

\begin{Eg}
Suppose $R=\mathcal{O}_K$, for $K/\mathbb{Q}_p$ a non-archimedean extension. By a result of Amice \cite{AmiceInterpolation}, after inverting $p$, the formula \eqref{eq.mahler-character-expansion} for the integral Fourier transform $\mathbb{F}^{\mathcal{H}}$ of Example \ref{example.mahler} restricts to a bijection between the continuous dual of the space of locally analytic $K$-valued functions on $\mathbb{Z}_p$ and the subring of $K[[t]]$ consisting of power series converging for $|t|<1$, i.e., the space of global functions on the rigid analytic generic fiber $\widehat{\mathbb{G}}_{m,\eta}$. This compatibility between the integral and rigid analytic Fourier transforms is a special case of the general compatibility in Theorem \ref{thm.integral-rational-compatibility}. The general case is more subtle because, for a general $p$-divisible group, we do not know whether the integral functions on the Tate module and the Hodge--Tate locally analytic functions on its generic fiber live naturally inside a common space (see Remark \ref{remark.non-injectivity}). 
\end{Eg}

\subsection{Proofs for integral Fourier theory}\label{ss.integral-fourier-proof}

\subsubsection{} 
We first explain how Cartier duality can be interpreted as Fourier theory for finite flat group schemes. Theorem \ref{thm.integral-fourier-theory} will be proved by bootstrapping from this case. 

Let $R$ be any commutative ring and let $\mathcal{H}$ be a finite locally free commutative group scheme over $\spec R$. For $\Hom$ the sheaf hom of groups on the fpqc site of $R$, let
\begin{align}
    \mathcal{H}^D=\Hom(\mathcal{H}, \gm)
\end{align}
be the Cartier dual of $\mathcal{H}$, which is itself representable by a finite locally free commutative group scheme, see, e.g, \cite[\S1.2]{Tate}. There is a canonical pairing
\begin{align}
    \kappa:\mathcal{H}\times \mathcal{H}^{D} \to \gm
\end{align}
and its composition with $\gm \to \mathbb{A}^1$ gives us an element
\begin{align}
    \kappa \in \mathcal{O}(\mathcal{H}\times \mathcal{H}^D) \simeq \oh \otimes_{R} \mathcal{O}(\mathcal{H}^{D}).
\end{align}
Contracting with this element defines $R$-linear maps (where the superscript $\ast$-denotes $R$-linear dual)
\begin{align}
    \mathbb{F}^{\mathcal{H}}:\oh^{\ast} &\to \mathcal{O}(\mathcal{H}^{D}) \\
    \mathbb{F}_{\mathcal{H}}:\mathcal{O}(\mathcal{H}^{D})^{\ast} &\to \oh.
\end{align}

We equip $\mathcal{O}(\mathcal{H})^{\ast}$ and $\mathcal{O}(\mathcal{H}^D)^{\ast}$ with the dual Hopf $R$-algebra structures. 

\begin{Lem} \label{Lem:CartierDualityFiniteFlat}
    The maps $ \mathbb{F}_{\mathcal{H}}$ and $\mathbb{F}^{\mathcal{H}}$ are dual isomorphisms of Hopf $R$-algebras.
\end{Lem}
\begin{proof}
    See \cite[II.1.2.10]{DemazureGabriel}.
\end{proof}

\subsubsection{} \label{subsub:StructureSheaf} Let $X=\varinjlim_n X_n$ be an ind-scheme written as a (countable) direct limit of qcqs schemes with transition maps being closed immersions (for example an affine formal scheme). We set $\mathcal{O}(X)=\varprojlim_n \mathcal{O}(X_n)$ with its inverse limit topology; this is a linearly topologized ring in the sense of \cite[Tag 07E8]{stacks-project}. Note that the topological ring $\mathcal{O}(X)$ does not depend on the choice of presentation: If $X=\varinjlim_m Y_m$ is another presentation, then the natural map $Y_m \to X$ factors through $X_n$ for some $n \gg 0$ since $Y_m$ is qcqs.  \smallskip

\textbf{Warning:} In the rest of this section we will consider all objects as ind-schemes, including $p$-divisible groups like $\qp/\zp$ which happen to be (non quasi-compact) schemes. In this example our definition gives $\mathcal{O}(\qp/\zp)= \textstyle \prod_{\qp/\zp} R$ with the product topology. 

\begin{Lem} \label{Lem:CompletedKunneth}
Let $R$ be a commutative ring and let $X= \varinjlim_n X_n$ and $Y=\varinjlim_m Y_m$ be ind-$R$-schemes. Then there is a natural isomorphism
    \begin{align}
        \mathcal{O}(X) \widehat{\otimes}_{R} \mathcal{O}(Y) \to \mathcal{O}(X \times_{\spec R} Y),
    \end{align}
    where the completed tensor product is defined as in \cite[Tag 0AMU]{stacks-project}.
\end{Lem}
\begin{proof}
    It follows from the definition that
    \begin{align}
         \mathcal{O}(X \times_{\spec R} Y) = \varprojlim_{n,m} \mathcal{O}(X_n \times_{\spec R} Y_m) = \varprojlim_{n,m} \mathcal{O}(X_n) \otimes_{R} \mathcal{O}(Y_m),
    \end{align}
    which equals $\mathcal{O}({X}) \widehat{\otimes}_{R} \mathcal{O}(Y)$ per definition. 
\end{proof}

\subsubsection{} \label{subsub:HodgeTate} Let $R$ be a commutative ring and let $\mathcal{H}$ be a $p$-divisible group over $\spec R$ (in the sense of \cite{Tate}). Let $\mathcal{H}^{\vee}$ be the Serre-dual of $\mathcal{H}$, which is defined as the $p$-divisible group with $\mathcal{H}^{\vee}[p^n]=\mathcal{H}[p^n]^{D}$, where the inclusions $\mathcal{H}[p^n]^{D} \to \mathcal{H}[p^{n+1}]^{D}$ are induced by the multiplication by $[p]$ map $\mathcal{H}[p^{n+1}] \to \mathcal{H}[p^n]$. Since
\begin{align}
    \mathcal{H}=\varinjlim_n \mathcal{H}[p^n],
\end{align}
it follows that
\begin{align}
    \Hom(\mathcal{H}, \gm) = \varprojlim_n \mathcal{H}^{\vee}[p^n]=T_p \mathcal{H}^{\vee}.
\end{align}
The map $T_p \mathcal{H}^{\vee}=\Hom(\mathcal{H}, \gm) \to \omega_{H}$ sending $f \mapsto f^{\ast} \tfrac{\mathrm{d}t}{1+t}$ is called the \emph{Hodge--Tate map}. 

\subsubsection{} Recall that $\oh=\varprojlim_n \mathcal{O}(\mathcal{H}[p^n])$ equipped with the inverse limit topology.  By Lemma \ref{Lem:CompletedKunneth}, the commutative group structure on $\mathcal{H}$ over $R$ induces on $\oh$ the structure of a topological Hopf algebra over $R$. 

\subsubsection{} Write $\oh^{\ast}$ for the continuous $R$-linear dual of $\oh$, so that $\oh^{\ast}=\varinjlim_n \mathcal{O}(\mathcal{H}[p^n])^{\ast}$, and equip $\oh^\ast$ with the weak (equivalently, discrete) topology. Here we are using the fact that a continuous $R$-linear morphism from an inverse limit to a discrete $R$-module factors through one of the terms in the inverse limit. Then the natural map
\begin{align}
    \oh^{\ast} \otimes_{R} \oh^{\ast} \to ( \oh \widehat{\otimes}_{R} \oh)^{\ast}
\end{align}
is an isomorphism, showing that $\oh^{\ast}$ is naturally an Hopf $R$-algebra. We also have a canonical identification
\[ (\oh^{\ast})^\ast = \varprojlim_n (\mathcal{O}(\mathcal{H}[p^n])^{\ast})^\ast = \varprojlim_n \mathcal{O}(\mathcal{H}[p^n]) = \oh. \]

Using the group structure on $T_p \mathcal{H}$, we find $\mathcal{O}(T_p \mathcal{H})=\varinjlim_n \mathcal{O}(\mathcal{H}[p^n])$ is naturally an Hopf $R$-algebra. We equip $\mathcal{O}(T_p \mathcal{H})^{\ast} = \varprojlim_n \mathcal{O}(\mathcal{H}[p^n])^{\ast}$ with the inverse limit topology (equivalently, the weak topology). The natural map 
\[ \mathcal{O}(T_p \mathcal{H})^{\ast} \widehat{\otimes}_R \mathcal{O}(T_p \mathcal{H})^{\ast} \rightarrow \left(\mathcal{O}(T_p \mathcal{H}) \otimes_R \mathcal{O}(T_p \mathcal{H}) \right)^{\ast} \]
is an isomorphism, showing that $\mathcal{O}(T_p \mathcal{H})^{\ast}$ is a topological Hopf algebra over $R$. We also have a canonical identification
\begin{multline*}(\mathcal{O}(T_p \mathcal{H})^{\ast})^\ast =  (( \varinjlim_n \mathcal{O}(\mathcal{H}[p^n]))^{\ast})^\ast = (\varprojlim_n \mathcal{O}(\mathcal{H}[p^n])^{\ast})^{\ast}\\=\varinjlim_n (\mathcal{O}(\mathcal{H}[p^n])^{\ast})^{\ast} = \varinjlim_n \mathcal{O}(\mathcal{H}[p^n])=\mathcal{O}(T_p \mathcal{H}). \end{multline*}

\subsubsection{} There is a canonical pairing
\begin{align}
    \kappa: \mathcal{H} \times T_p \mathcal{H}^{\vee} \to \gm
\end{align}
giving us an element (see Lemma \ref{Lem:CompletedKunneth})
\begin{align}
    \kappa \in \oh \widehat{\otimes}_{R} \mathcal{O}(T_p \mathcal{H}^{\vee}). 
\end{align}
This element induces continuous $R$-linear maps
\begin{align}
    \mathbb{F}^{\mathcal{H}}&:\oh^{\ast}
    \to \mathcal{O}(T_p \mathcal{H}^{\vee}) \\
    \mathbb{F}_{\mathcal{H}}&:\mathcal{O}(T_p \mathcal{H}^{\vee})^{\ast} \to \oh.
\end{align}
\begin{Prop} \label{Prop:IntegralFourierDiscrete}
The morphisms $ \mathbb{F}_{\mathcal{H}}$ and $\mathbb{F}^{\mathcal{H}}$ are dual isomorphisms of topological Hopf $R$-algebras.
\end{Prop}
\begin{proof}
The map $\mathbb{F}_{\mathcal{H}}$ is a direct limit over $n$ of the maps $\mathbb{F}_{\mathcal{H}[p^n]}$ from Lemma \ref{Lem:CartierDualityFiniteFlat}, and similarly $\mathbb{F}^{\mathcal{H}}$ is an inverse limit over $n$ of the maps $\mathbb{F}^{\mathcal{H}[p^n]}$ from Lemma \ref{Lem:CartierDualityFiniteFlat}. The result thus follows directly from Lemma \ref{Lem:CartierDualityFiniteFlat}.
\end{proof}

\begin{Rem}
    Note that the topologies on $\oh^{\ast}$ and $ \mathcal{O}(T_p \mathcal{H}^{\vee})$ are discrete, so that $\mathbb{F}_{\mathcal{H}}$ is an isomorphism between discrete Hopf $R$-algebras. 
\end{Rem}

\subsubsection{} Let $R$ be a $p$-adically complete and separated ring, and $\mathcal{H}$ a $p$-divisible group over $\spf R$. Then $\mathcal{H}$ is an ind-scheme with presentation $\mathcal{H}=\varinjlim_{n,k} \mathcal{H}[p^n]_{R/p^k}$ and thus 
 \[ \mathcal{O}(\mathcal{H})=\varprojlim_k \varprojlim_n \mathcal{O} (\mathcal{H}[p^n]_{R/p^k}) \]
with the inverse limit topology. Equivalently, this is 
\[ \mathcal{O}(\mathcal{H})=\varprojlim_k \mathcal{O}(\mathcal{H}_{R/p^k}) \]
 with the inverse limit of the topologies considered above, or, swapping the order, the inverse limit of the $p$-adic topologies on $\mathcal{O} (\mathcal{H}[p^n])$ in
 \[ \mathcal{O}(\mathcal{H}) = \varprojlim_n \mathcal{O} (\mathcal{H}[p^n]). \]
It similarly follows that
\[ \mathcal{O}(T_p \mathcal{H}^{\vee})=\varprojlim_k \mathcal{O}(T_p \mathcal{H}^{\vee}_{R/p^k}) \]
acquires the inverse limit topology (i.e., the $p$-adic topology, since the terms in the limit have the discrete topology). Once again we find that $\oh$, $\oh^{\ast}$, $\mathcal{O}(T_p \mathcal{H}^{\vee})$ and, $\mathcal{O}(T_p \mathcal{H}^{\vee})^{\ast}$ are topological Hopf $R$-algebras, where a superscript $\ast$ denotes the continuous $R$-linear dual equipped with the weak topology.
There is a pairing
\begin{align}
    \kappa: \mathcal{H} \times T_p \mathcal{H}^{\vee} \to \mathbb{G}_{m, \spf R}
\end{align}
giving us an element (see Lemma \ref{Lem:CompletedKunneth})
\begin{align}
    \kappa \in \mathcal{O}( \mathcal{H} \times T_p \mathcal{H}^{\vee}) =\oh \widehat{\otimes}_{R} \mathcal{O}(T_p \mathcal{H}^{\vee}). 
\end{align}
This element induces continuous $R$-linear maps
\begin{align}
    \mathbb{F}^{\mathcal{H}}&:\oh^{\ast}
    \to \mathcal{O}(T_p \mathcal{H}^{\vee}) \\
    \mathbb{F}_{\mathcal{H}}&:\mathcal{O}(T_p \mathcal{H}^{\vee})^{\ast} \to \oh.
\end{align}
The following proposition is Theorem \ref{thm.integral-fourier-theory}. 
\begin{Prop} \label{Prop:IntegralFourier}
The morphisms $ \mathbb{F}_{\mathcal{H}}$ and $\mathbb{F}^{\mathcal{H}}$ are isomorphisms of complete linearly topologized Hopf $R$-algebras. 
\end{Prop}
\begin{proof}
We apply Proposition \ref{Prop:IntegralFourierDiscrete} mod $p^k$ for all $k \geq 1$. 
\end{proof}

Finally, we prove the equivariance for the actions of $\Sym^\bullet \Lie \mathcal{H}$. 

\begin{proof}[Proof of Proposition \ref{Prop:IntegralEquivariance}]
The dual Hodge-Tate map can be constructed relatively by viewing $\kappa$ as a character of $\mathcal{H}$ based changed to $T_p \mathcal{H}^\vee$ (noting that $T_p \mathcal{H}^\vee$ is itself an affine $p$-adic formal scheme),
\[ \mathcal{H}_{T_p \mathcal{H}^\vee} \xrightarrow{\kappa} \widehat{\mathbb{G}}_{m, T_p \mathcal{H}^\vee}, \]
and then pulling back $\frac{dt}{1+t}$ to an element of $\omega_{(\mathcal{H}^\vee \times_{\Spf R} T_p \mathcal{H}^\vee )/ T_p \mathcal{H}^\vee}$.

In particular, for $\partial \in \Lie \mathcal{H}$ an invariant derivation, if we view $\kappa$ as a function by the natural inclusion $\gmhat \subseteq \mathbb{A}^1$ we have (cf. the proof of Lemma \ref{Lem.DerivationActionLocAn})
\[ \partial\kappa= \langle \partial, \kappa^* \frac{dt}{1+t}\rangle \kappa = \HT^*(\partial)\kappa.\] 
The result then follows as in the proof of Proposition \ref{Prop:FTSymVstarEquiv}.
\end{proof}

\subsection{Comparison with Fourier theory on the generic fiber}\label{ss.integral-rational-comparison}
\subsubsection{} \label{subsub:IntegralNotation} Let $(R,R^+)/(\mathbb{Q}_p,\mathbb{Z}_p)$ be fiercely v-complete and let $\mathcal{H}$ be a $p$-divisible group over $\spf R^+$ with rigid generic fiber $\mathcal{H}_\eta$ over $Y=\Spa(R,R^+)$.
We write $\HT: T_p \mathcal{H}_\eta^\vee \otimes_{\ul{\mathbb{Z}_p} } \mathcal{O}_{Y_v} \rightarrow \omega_{\mathcal{H}} \otimes_{R^+} \mathcal{O}_{Y_v}$ for the rational Hodge--Tate map, and we view $T_p \mathcal{H}^\vee_{\eta}$ as a v-sheaf. To simplify our notation, below we will write
    $\mathcal{O}^{\HT-\locan}(T_p\mathcal{H}^\vee_\eta) :=\mathcal{O}^{\HT-\locan}_{T_p \mathcal{H}^\vee_\eta / Y}(Y)$, where $Y=\spd(R,R^+)$.

As in \eqref{eq.restriction-map-locally-analytic}, we have a natural $R$-linear restriction map
\[  r^\locan: \mathcal{O}^{\HT-\locan}(T_p \mathcal{H}_\eta^\vee) \hookrightarrow \mathcal{O}(T_p \mathcal{H}_\eta^\vee)\]
that is injective by Lemma \ref{lemma.restriction-map-injective}. 
We also have natural $R$-linear restriction maps
\[ r_{\mathcal{H}}: \mathcal{O}(\mathcal{H})[\tfrac{1}{p}] \rightarrow \mathcal{O}(\mathcal{H}_{\eta}) \textrm{ and }  r^\mathrm{int}: \mathcal{O}(T_p \mathcal{H}^\vee)[\tfrac{1}{p}] \rightarrow \mathcal{O}(T_p \mathcal{H}_\eta^\vee). \]

\begin{Rem}\label{remark.injectivity-r-H}
If $\mathcal{H}$ is an extension of an \'{e}tale $p$-divisible group by a formal group, then $\mathcal{H}$ is representable by a formal scheme by \cite[Lemma 3.1.1]{ScholzeWeinstein} and $r_{\mathcal{H}}$ is injective (as can be verified using the local presentation of $\widehat{\mathcal{H}}$ in \cite[Lemma 3.1.2]{ScholzeWeinstein}). We expect this injectivity to hold in general.  
\end{Rem}

\begin{Eg}
    If $R=\mathcal{O}_K$ for $K/\mathbb{Q}_p$ a non-archimedean extension, then 
    \[ \mathcal{O}(T_p\mathcal{H}^\vee_\eta)= \Cont(T_p \mathcal{H}^\vee(\overline{K}), \overline{K}^\wedge)^{\Gal(\overline{K}/K)}, \]
    where $\overline{K}^{\wedge}$ denotes a completed algebraic closure of $K$.
\end{Eg}

\begin{Rem}\label{remark.non-injectivity}
    The map $r^{\mathrm{int}}$ is an isomorphism if $\mathcal{H}=\mu_{p^\infty}$ so that $T_p \mathcal{H}^\vee=\mathbb{Z}_p$. However, in all other cases we expect the behavior of the map $r^\mathrm{int}$ to be quite subtle. For example, if $\mathcal{H}=\mathbb{Q}_p/\mathbb{Z}_p$ so that $T_p \mathcal{H}^\vee=\mathbb{Z}_p(1)$ and $R=\mathcal{O}_{E^\wedge}$ for $E$ an algebraic extension of $\mathbb{Q}_p$, then the results of \cite{FresneldeMathan.AlgebresL1padiques} imply that:
    \begin{itemize}
        \item If $E$ is deeply ramified (equivalently, $E^\wedge$ is perfectoid), then $r^\mathrm{int}$ is surjective but not injective (its kernel even contains a dense set of nilpotents!).
        \item If $E$ is not deeply ramified (equivalently, $E^\wedge$ is not perfectoid) then $r^{\mathrm{int}}$ is injective with dense image but not surjective. 
    \end{itemize}
We note that the surjectivity in the deeply ramified case is a special case of \cite[Theorem 7.4]{BhattScholzePrisms}. In fact, it is the specific case invalidating \cite[Remark II.2.4]{ScholzeTorsion}. 

Katz \cite[top of p.60]{Katz.FormalGroupsAndPAdicInterpolation} (see also \cite{BergerGaloisMeasures} and \cite{ArdakovBerger}) has also noted that $(r^{\mathrm{int}})^{\ast}$ is injective whenever $R=\mathcal{O}_{E^\wedge}$ for $E$ a discretely valued algebraic extension of $\mathbb{Q}_p$ and $\mathcal{H}$ a one-dimensional $p$-divisible formal group over $R$. When furthermore $E/\mathbb{Q}_p$ is unramified and $\mathcal{H}$ is of height $1$ or $2$, he has given an explicit description of the image of $\mathbb{F}^{\mathcal{H}}[\tfrac{1}{p}] \circ (r^\mathrm{int})^{\ast}$ in \cite[Theorem on p.60]{Katz.FormalGroupsAndPAdicInterpolation}. 

It would be interesting to have a more systematic general understanding of these results!

\end{Rem}
\subsubsection{} We continue with the notation from Section \ref{subsub:IntegralNotation}. 
Below we will also write 
\[ \mathcal{D}(\mathcal{H}_\eta):=\mathcal{D}_{\mathcal{H}_\eta/Y}(Y) \textrm{ and } \mathcal{D}^{\HT-\locan}(T_p \mathcal{H}^\vee_\eta):=\mathcal{D}_{T_p \mathcal{H}^\vee_\eta/Y}^{\HT-\locan}(Y).\]
We note that, when $\Lie \mathcal{H}$ is free, the discussion of \S\ref{sss.solid-hopf-and-duality-discussion-fiercely} implies that
\[ \ul{\mathcal{D}(\mathcal{H}_\eta)}=\ul{\mathcal{O}(\mathcal{H}_\eta)}^* \textrm{ and } \ul{\mathcal{D}^{\HT-\locan}(T_p \mathcal{H}^\vee)}=\ul{\mathcal{O}^{\HT-\locan}(T_p \mathcal{H}^\vee_\eta)}^*,\]
where the superscript $\ast$ denotes the condensed $\ul{R}$-linear dual. Following Remark \ref{remark.etale-expectations}, we expect this equality to hold always. To state our compatibility without establishing this expectation or imposing the condition that $\Lie \mathcal{H}$ be free, we observe that $r_{\mathcal{H}}^*$ and $r_\mathrm{int}^*$ can be glued, by working Zariski locally on $\Spf R^+$ where $\Lie \mathcal{H}$ is free, in order to obtain maps
\[ r_{\mathcal{H}}^*: \mathcal{D}(\mathcal{H}_\eta) \to \mathcal{O}(\mathcal{H})^*[\tfrac{1}{p}] \textrm{ and } r_\mathrm{int}^*: \mathcal{D}^{\HT-\locan}(T_p \mathcal{H}^\vee_\eta) \to \mathcal{O}(T_p \mathcal{H}^\vee)^*[\tfrac{1}{p}] \]
where here the superscript $*$ is the continuous $R^+$-linear dual. 

By Proposition \ref{Prop.TateModule}, we have $\mathcal{H}_\eta = H_{\mathcal{L}}$ for $$\mathcal{L}={(T_p \mathcal{H}_\eta, \Lie \mathcal{H} \otimes_{R^+} \mathcal{O}_{Y_v},  \Lie \mathcal{H} \otimes_{R^+} \mathcal{O}_{Y_v} \hookrightarrow T_p \mathcal{H}_\eta(-1) \otimes_{\ul{\mathbb{Z}_p}} \mathcal{O}_{Y_v})}.$$ 
In particular, taking $\mathcal{L}^{\vee}$ to be the locally analytic character datum associated to $\mathcal{L}$ in \S \ref{subsub:CharacterDatum} the Fourier transforms of $\mathbb{F}^{\mathcal{L}^\vee}$ and $\mathbb{F}_{\mathcal{L}^\vee}$ of \S\ref{Sub:FourierDef} induce the following maps on global sections:
\begin{align*} \mathbb{F}^{\mathcal{H}_\eta} :& \mathcal{D}(\mathcal{H}_\eta)  \rightarrow \mathcal{O}^{\HT-\locan}(T_p \mathcal{H}^\vee_\eta) \\
\mathbb{F}_{\mathcal{H}_\eta} :&  \mathcal{D}^{\HT-\locan}(T_p \mathcal{H}^\vee_\eta)   \rightarrow \mathcal{O}(\mathcal{H}_\eta). 
\end{align*}

\begin{Thm} \label{thm.integral-rational-compatibility}
Suppose $(R,R^+)/(\mathbb{Q}_p,\mathbb{Z}_p)$ is fiercely v-complete. If $\mathcal{H}$ is a $p$-divisible group over $\spf R^+$ with rigid analytic generic fiber $\mathcal{H}_\eta$ over $\Spa(R,R^+)$, then the following diagrams commute

\[\begin{tikzcd}
	{\mathcal{O}(\mathcal{H})^{\ast}[\tfrac{1}{p}]} & {\mathcal{O}(T_p\mathcal{H}^\vee)[\tfrac{1}{p}]} & {\mathcal{O}(\mathcal{H})[\tfrac{1}{p}]} & {\mathcal{O}(T_p \mathcal{H}^\vee)^{\ast}[\tfrac{1}{p}]} \\
	& {\mathcal{O}(T_p \mathcal{H}^\vee_\eta)} && {\mathcal{O}(T_p \mathcal{H}^\vee_\eta)^{\ast}} \\
	{\mathcal{D}(\mathcal{H}_\eta)} & {\mathcal{O}^{\HT-\locan}(T_p \mathcal{H}^\vee_\eta)} & {\mathcal{O}(\mathcal{H}_\eta)} & {\mathcal{D}^{\HT-\locan}(T_p\mathcal{H}^\vee_\eta)}.
	\arrow["{\mathbb{F}^{\mathcal{H}}[\tfrac{1}{p}]}", from=1-1, to=1-2]
	\arrow["{r^{\mathrm{int}}}", from=1-2, to=2-2]
	\arrow["{r_{\mathcal{H}}}", from=1-3, to=3-3]
	\arrow["{\mathbb{F}_{\mathcal{H}}[\tfrac{1}{p}]}"', from=1-4, to=1-3]
	\arrow["{(r^{\mathrm{int}})^{\ast}}"', from=2-4, to=1-4]
	\arrow["{(r^\locan)^{\ast}}", from=2-4, to=3-4]
	\arrow["{r_{\mathcal{H}}^{\ast}}", from=3-1, to=1-1]
	\arrow["{\mathbb{F}^{\mathcal{H}_\eta}}", from=3-1, to=3-2]
	\arrow["{r^{\locan}}"', from=3-2, to=2-2]
	\arrow["{\mathbb{F}_{\mathcal{H}_\eta}}"', from=3-4, to=3-3]
\end{tikzcd}\]
\end{Thm}
\begin{proof}
The diagram is functorial for change of base. Since the generic fiber terms are v-sheaves, it suffices to observe that, when $(R,R^+)$ is perfectoid and $T_p \mathcal{H}^\vee_\eta\cong \ul{\mathbb{Z}_p}^{\oplus h}$ is trivial,  
\[ \kappa \in \mathcal{O}(T_p \mathcal{H}^\vee \times \mathcal{H}) \]
maps to the universal character of Lemma \ref{Lem:UniversalKappaAsSection} inside 
\[ \mathcal{O}(T_p \mathcal{H}^\vee_\eta \times \mathcal{H}_\eta)=\Cont(\mathbb{Z}_p^h, \mathcal{O}(\mathcal{H}_\eta)). \]
But this is clear from the construction.
\end{proof}

As a consequence we obtain a partial section of $r^\mathrm{int}$ (which is very interesting as Remark \ref{remark.non-injectivity} shows that $r^\mathrm{int}$ need not be surjective nor injective!):
\begin{Cor}\label{cor.la-section}
    There is a canonical section $\mathcal{O}^{\HT-\locan}(T_p \mathcal{H}_\eta^\vee) \rightarrow \mathcal{O}(T_p \mathcal{H}^\vee)[\tfrac{1}{p}]$ of $r^\mathrm{int}$. 
\end{Cor}
\begin{proof}
    Since $\mathbb{F}^{\mathcal{H}_\eta}$ is an isomorphism by Theorem \ref{Thm:Main}, we obtain the section as $\mathbb{F}^{\mathcal{H}}[\tfrac{1}{p}] \circ r_{\mathcal{H}}^{\ast} \circ (\mathbb{F}^{\mathcal{H}_\eta})^{-1}$.
\end{proof}

\begin{Eg}
    When $\mathcal{H}=\mathbb{Q}_p/\mathbb{Z}_p$, so $T_p\mathcal{H}^\vee = \mathbb{Z}_p(1)$, the Hodge--Tate map is identically zero (its target is $\omega_{\mathbb{Q}_p/\mathbb{Z}_p}=0$). In particular, when $R=\mathcal{O}_{\mathbb{C}_p}$, the space 
    $\mathcal{O}^{\HT-\locan}(T_p \mathcal{H}_\eta^\vee)$ is given by the locally constant $\mathbb{C}_p$-valued functions on $\mathbb{Z}_p(1)$. Each locally constant function has a canonical expression as a finite linear combination of characters, and this provides the canonical section to $\mathcal{O}(\mathbb{Z}_p(1))[\tfrac{1}{p}]$. In other situations, there are many more locally analytic functions, expressible locally as power series using the Hodge--Tate map, and the existence of the canonical section expresses a non-trivial convergence condition on these series. When $\mathcal{H}$ is, e.g., a Lubin--Tate formal group, the canonical section should be closely related to the description of locally analytic functions in the associated Lubin--Tate extension given in \cite{BergerColmez.Sen}.
\end{Eg}

\section{The global Eisenstein measure} \label{Sec:Eisenstein}

\subsection{Background} \label{Sub:IntroEisenstein}

For $\Lambda \subseteq \mathbb{C}$ a lattice and $k\geq 3$ an integer, consider the series 
\[ A_k(\Lambda)=\sum_{\lambda \in \Lambda \backslash \{0\}}\frac{1}{\lambda^k}. \]
Certain values of $A_k$ can be related to special values of $L$-functions. For example, if $\Lambda=\mathbb{Z}[i]$, these are the Bernoulli--Hurwitz numbers -- in terms of $L$-functions, if we write $\chi$ for the Grossencharacter sending an ideal $I$ of $\mathbb{Z}[i]$ to $z$, for $z$ the unique generator of $I$ such that $z \equiv 1 \mod (1+i)^3$, then 
\[ A_k(\mathbb{Z}[i])=\begin{cases}
    4 L(\chi^{-k}, 0) & \textrm{ if $k \equiv 0 \mod 4$} \\
    0 & \textrm{otherwise}. 
\end{cases}\]
On the other hand, the function $A_k$ defines a modular form, the level one Eisenstein series of weight $k$ (which is identically zero if $k$ is odd) --- viewed as a function on the upper half-plane, it is given by $\tau \mapsto A_k(\mathbb{Z} + \mathbb{Z}\tau)$.

\subsubsection{} In \cites{Katz.p-adic-L-via-moduli, Katz.TheEisensteinMeasureAndpAdicInterpolation}, Katz explained how to construct, for primes $p \equiv 1 \mod 4$, a $p$-adic $L$-function interpolating the values $L(\chi^{-k}, 0)$ into a measure on $\mathbb{Z}_p$ by first interpolating the renormalized Eisenstein series 
\[ G_k := \frac{(-1)^k (k-1)!}{2} A_k \]
into a measure on $\mathbb{Z}_p$ valued in $p$-adic modular forms. One can then specialize at the lemniscate elliptic curve $E_0$ defined by the equation $y^2=4x^3-4x$, whose complex points are isomorphic to $\mathbb{C}/\mathbb{Z}[i]$, in order to obtain a measure whose values on the monomials $z \mapsto z^k$, $k\geq 3$, are renormalized Bernoulli--Hurwitz numbers. The $p$-adic $L$-function is obtained by restricting this measure to $\mathbb{Z}_p^\times$, which modifies an Euler factor at $p$ and allows to evaluate at any character $\kappa$ of $\mathbb{Z}_p^\times$ -- in particular, since these characters satisfy many $p$-adic congruences, so do the renormalized Bernoulli--Hurwitz numbers. 

The basis of this approach is an observation of Deligne that, in order to analyze congruences between the Eisenstein series, one can use the geometry of the modular curve to reduce to interpolating the \emph{positive} Fourier coefficients of the renormalized Eisenstein series. These Fourier coefficients are expressed in terms of divisor power sums, whose congruences can be analyzed trivially using Fermat's little theorem. The Kubota--Leopoldt $p$-adic zeta function is also obtained in this way by evaluating at a cusp instead of a CM point. 

\subsubsection{} This interpolation of Eisenstein series explained above occurs over the ordinary locus, thus yields a $p$-adic $L$-function interpolating Bernoulli--Hurwitz numbers only for primes $p \equiv 1 \mod 4$ (i.e., splitting in $\mathbb{Z}[i]$ so that $E_0$ has ordinary reduction).
In \cites{Katz.FormalGroupsAndPAdicInterpolation, Katz.DivisibilitiesCongruencesAndCartierDuality}, Katz explained a different construction of the $p$-adic $L$-function interpolating the Bernoulli--Hurwitz numbers that is valid also at the inert primes $p\equiv 3 \mod 4$ where $E_0$ has supersingular reduction. This construction exploits Cartier duality and the appearance of the Eisenstein series as coefficients of the Weierstrass $\wp$-function. In our language, the key observation is that, after a small modification to remove the pole at the origin, $\wp$ restricts to a regular function on the $p$-divisible group $E_0[p^\infty]$ and thus by Cartier duality yields a measure on $T_p E_0^\vee$. Katz shows moreover that, in this specific case, $(r^\mathrm{int})^*$ is injective and the measure is in its image (cf. Remark \ref{remark.non-injectivity}), i.e., the measure comes from an element of 
\[ \mathcal{O}(T_p (E_{0,\eta}^\vee))^* = \left(\Cont(T_p E_{0}^{\vee}(\mathbb{C}_p), \mathbb{C}_p)^{\Gal(\overline{\mathbb{Q}}_p/\mathbb{Q}_p)} \right)^*. \]
In the language of \cite{Katz.FormalGroupsAndPAdicInterpolation}, one thus obtains a Galois-equivariant measure; its restriction to $T_p E_{0}^{\vee}(\mathbb{C}_p) \setminus p T_p E_0^{\vee}(\mathbb{C}_p) \cong \mathbb{Z}_{p^2}^\times$ (with the Lubin--Tate Galois action) is then a natural $p$-adic $L$-function interpolating the Bernoulli--Hurwitz numbers.  

In fact, rather than viewing $\mathcal{O}(E_0[p^\infty])$ as  $\mathcal{O}(T_p E_0^\vee)^*$, in \cite{Katz.FormalGroupsAndPAdicInterpolation} Katz exploits that $E_0[p^\infty]$ is connected to view the dual of $\mathcal{O}(E_0[p^\infty])$ as the ring of invariant differential operators, see Remark \ref{remark.integral-artier-duality}; he then argues from this perspective to obtain an element of $\mathcal{O}(T_p (E_{0,\eta}^\vee))^*$. The interpretation via invariant differential operators applies only to formal groups, so cannot be used over the entire modular curve. However, by instead using our Fourier transform to identify $\mathcal{O}(E_0[p^\infty])=\mathcal{O}(T_p E_0^\vee)^*$, we can extend this approach to obtain a global Eisenstein measure. This construction and its consequences are explained in \S\ref{ss.global-eisenstein-measure}. 

\subsubsection{} As observed already by Schneider and Teitelbaum \cite{SchneiderTeitelbaumFourier}, the exotic interpretation of Katz's Bernoulli--Hurwitz $p$-adic $L$-function when $p\equiv 3 \mod 4$ as a Galois equivariant measure becomes much cleaner if one passes to the generic fiber: the $p$-adic $L$-function becomes a $\mathbb{Q}_{p^2}$-analytic distribution on $\mathbb{Z}_{p^2}^\times$. In our global setting, we may also pass to the generic fiber to obtain a Hodge--Tate analytic distribution on the Tate module of the dual of the universal elliptic curve. This distribution specializes at the lemniscate curve to recover the construction of Schneider--Teitelbaum, but also has other applications. Using this distribution we recover the usual overconvergence of Katz's Eisenstein family of $p$-adic modular forms (Example \ref{eg.overconvergent-eisenstein-series}) and construct new overconvergent versions of the Schneider--Teitelbaum $p$-adic $L$-functions that exist on open loci in the supersingular locus (Example \ref{eg.quaternionic-eisenstein-series}). These results are explained in \S\ref{ss.global-eisenstein-distribution}.

\subsection{The global Eisenstein measure}\label{ss.global-eisenstein-measure}
To state our result, we first introduce some terminology. Let $N \geq 3$ coprime to $p$. Let $\mathfrak{Y}$ be the modular curve over $\spf \zp$ of full level $N$. Let $\mathcal{E}$ be the universal elliptic curve over $\mathfrak{Y}$ and let $\omega=\omega_{\mathcal{E}}$ be the sheaf of invariant differentials relative to $Y$. We denote by $\mathfrak{Y}_{\mathrm{ord}} \subset \mathfrak{Y}$ the open locus where $\mathcal{E}$ is ordinary.

We write $\Igkatz \to \mathfrak{Y}_{\mathrm{ord}}$ for the Katz--Igusa formal scheme parameterizing trivializations $\varphi:\widehat{\mathcal{E}} \xrightarrow{\sim}\gmhat$, or equivalently $(T_p \mathcal{E}^\vee)^\et \xrightarrow{\sim} \ul{\mathbb{Z}_p}$, which is
a profinite \'{e}tale $\mathbb{Z}_p^\times$-torsor over $\mathfrak{Y}_\mathrm{ord}$. There is a canonical section $\omega_\mathrm{can}=\varphi^*\frac{dt}{1+t}$ of $\omega$ over $\Igkatz$, and $\mathbb{V}=\mathcal{O}(\Igkatz)$ is the ring of Katz--Serre $p$-adic modular forms of prime-to-$p$ level $N$. Let $\operatorname{HT} \colon T_p \mathcal{E}^{\vee} \to \omega$ denote the Hodge--Tate map and, for an integer $m \in \mbb{Z}$, we let $\operatorname{HT}^m \colon T_p \mathcal{E}^{\vee} \to \omega^{\otimes m}$ be the $m$-th fold tensor product of $\operatorname{HT}$. We can view $\operatorname{HT}^m \in \mathcal{O}(T_p\mathcal{E}^{\vee}) \otimes_{H^0(\mathfrak{Y}, \mathcal{O})} H^0(\mathfrak{Y}, \omega^m)$. 

\begin{Thm} \label{thm.integral-eisenstein}
Let $p>3$ be an integer prime. For any integer $n > 1$  coprime to $p$, there is an element 
\[ \Eis^{(n)} \in H^0(\mathfrak{Y}, \omega^2) \otimes_{H^0(\mathfrak{Y}, \mathcal{O})} \mathcal{O}(T_p \mathcal{E}^\vee)^*, \]
i.e., a weight $2$ modular form $\Eis^{(n)}$ over $\Spf \mathbb{Z}_p$ valued in measures on $T_p \mathcal{E}^\vee$, such that: For $k \geq 3$, we have
\[ 
\int \HT^{k-2} d\Eis^{(n)} = 2(1-n^k)G_{k}
\]
and $\int 1 d\Eis^{(n)} = 0$ (here we are considering $G_k$ as an element of $H^0(\mathfrak{Y}, \omega^k)$ in the usual way).
\end{Thm}
\begin{proof}
Consider the weight $2$ modular form which sends a non-vanishing differential $\eta$ to the meromorphic function $x$ in the unique Weierstrass equation $y^2=4x^3-g_2 x-g_3$ with $\eta=\frac{dx}{y}$, see \cite[Section 2.2]{KatzMazur}. The function $x^{(n)}:=x - n^2[n]^*x$, where $[n]:\mathcal{E} \to \mathcal{E}$ is the multiplication by $n$ map, has poles at the non-trivial $n$-torsion points (one can see that the poles at the identity of $\mathcal{E}$ cancel, e.g., from the transcendental computation given below). In particular, $x^{(n)}|_{\mathcal{E}[p^\infty]} \in \mathcal{O}(\mathcal{E}[p^\infty])$, and we obtain the weight $2$ modular form $\Eis^{(n)}$ as $(\mathbb{F}^{\mathcal{E}[p^\infty]})^{-1}(x^{(n)}|_{\mathcal{E}[p^\infty]})$, for $\mathbb{F}^{\mathcal{E}[p^\infty]}$ the integral Fourier transform as in Theorem \ref{thm.integral-fourier-theory}. 

We explain the computation for Theorem \ref{thm.integral-eisenstein}. First, unwinding the definitions, one finds that $\int \HT^{j}d\Eis^{(n)}$ is the weight $j+2$ modular form that sends a non-vanishing differential $\eta$ with dual $\partial$ to $\partial^{j} x^{(n)}|_{x=1}$, where $x^{(n)}$ is as above. This is a section of $\omega^{j+2}$ defined over $\mathbb{Z}[1/6N]$, and thus we can verify the claimed identity over $\mathbb{C}$. Then, taking $x^{(n)}$ corresponding to $(\mathbb{C}/\Lambda, dz)$ and using the expansion (see \cite[Theorem 3.5]{Silverman})
\[ 
x(z)=\wp_{\Lambda}(z)= \frac{1}{z^2} + \sum_{k\geq 3} (k-1)A_{k}(\Lambda) z^{k-2}, 
\] 
we find 
\[ x^{(n)}(z) = \wp_{\Lambda}(z)-n^2\wp_{\Lambda}(nz) = \sum_{k \geq 3} (1-n^{k})(k-1)A_k(\Lambda)z^{k-2}. \]
In particular, the constant term is zero, and for $k \geq 3$, differentiating $k-2$ times with respect to $z$ we obtain the constant term $(1-n^{k})(k-1)!A_k=2(1-n^k)G_k$, as desired.  
\end{proof}
\subsubsection{} Given a map $\spf R \to \mathfrak{Y}$, which corresponds to an elliptic curve $E \to \spf R$ (with full level $N$ structure), and a non-vanishing differential $\eta \in \omega_E$, we let $\Eis^{(n)}(E, \eta)$ denote the element of $\mathcal{O}(T_p \mathcal{E}^{\vee})^*$ given by pullback along $\spf R \to \mathfrak{Y}$ and using the trivialization $\omega_E \cong R$ given by $\eta$. Let $\mu^{(n)}:C^0(\zp,\zp) \to \mathbb{V}^{\mathrm{GL}_2(\mathbb{Z}/N\mathbb{Z})}$ be the Eisenstein measure of \cite{Katz.p-adic-L-via-moduli}. 
\begin{Thm} \label{thm.integral-eisensteinII}
Let $\pi: T_p \mathcal{E}^\vee|_{\Igkatz} \rightarrow \ul{\mathbb{Z}_p}$ be the natural map induced by $\varphi$. Then for $f \in C^0(\mathbb{Z}_p, \mathbb{Z}_p)$ and $z$ the identity function $\mathbb{Z}_p \rightarrow \mathbb{Z}_p$, we have
\[ \int_{\zp} f \; \mathrm{d}\left(\pi_* \Eis^{(n)}(\mathcal{E}|_{\Igkatz}, \omega_\mathrm{can})\right) = \int_{\zp} z \cdot \left(f(z)-f(0)\right) \;  \mathrm{d}\mu^{(n)}. \]
\end{Thm}
\begin{proof}
Since $\mathbb{V}$ is $p$-torsion free, it suffices to verify the statement on $f=z^j$, $j \geq 0$. For $j=0$, both sides give zero. For $j\geq 1$, the left-hand side gives $2(1-n^{j+2})G_{j+2}(\omega_\mathrm{can})$ by Theorem \ref{thm.integral-eisenstein}. For $j \geq 1$, the right-hand side is $\int z^{j+1}d\mu^{(n)}$, which agrees by \cite[Theorem on p.501]{Katz.p-adic-L-via-moduli}. 
\end{proof}

\begin{Rem}
    We have introduced level $N$ structure just to have a moduli scheme instead of a moduli stack, but are only interpolating the level $1$ Eisenstein series. It is also possible to construct more general global Eisenstein measures interpolating all of the level $N$ Eisenstein series and specializing to the measures of \cite{Katz.TheEisensteinMeasureAndpAdicInterpolation}; a natural way to accomplish this is to restrict the Weierstrass function to $(\mathcal{E}[p^\infty] \times \mathcal{E}[N]) \backslash (\mathcal{E}[p^\infty] \times 0)$ to obtain a measure on $T_p \mathcal{E}^\vee \times \mathcal{E}^\vee[N]$ supported on the complement of $T_p \mathcal{E}^\vee \times 0$. We leave the details to the interested reader. 
\end{Rem}

\begin{Eg}
In the setting of Theorem \ref{thm.integral-eisensteinII}, the usual weight families of $p$-depleted Eisenstein series parameterized by characters of $\mathbb{Z}_p^\times$ are obtained by restricting the measure $\mu^{(n)}$ from $\mathbb{Z}_p$ to $\mathbb{Z}_p^\times$ 
(i.e., taking the measure on $\mathbb{Z}_p^\times$ given by evaluating $\mu^{(n)}$ on elements of $\Cont(\mathbb{Z}_p^\times, \mathbb{Z}_p)$ by extending by zero). In particular, the constant term discrepancy and multiplication by $z$ in (2) do not change these resulting families: In other words, the usual weight family can be recovered from $\Eis^{(n)}$ by
\[ \kappa \mapsto \int_{\mathbb{Z}_p^\times} z^{-1}\kappa \; \mathrm{d} \left(\pi_* \Eis^{(n)}(\mathcal{E}|_{\Igkatz}, \omega_\mathrm{can})\right).\]
\end{Eg}

\begin{Eg}\label{e.g.p-equiv-3-mod-4}
For $p \equiv 3 \mod 4$, the measure $\Eis^{(n)}(E_0, \frac{dx}{y})$ is, by a direct comparison of constructions, the measure obtained by applying $(r^\mathrm{int})^*$ (as in Theorem \ref{thm.integral-rational-compatibility}) to the Galois-equivariant measure on $T_p E_0^{\vee}(\mathbb{C}_p)$ constructed in \cite{Katz.FormalGroupsAndPAdicInterpolation}.
\end{Eg}

\subsection{The global Eisenstein distribution}\label{ss.global-eisenstein-distribution}
Passing to the generic fiber $\mathfrak{Y}_\eta$ (the good reduction locus inside the rigid analytic modular curve of level $N$), we obtain a Hodge--Tate analytic distribution $\Eis^{(n)}_\eta$ on $T_p \mathcal{E}^\vee_\eta$, the Tate module of the universal elliptic curve over $\mathfrak{Y}_\eta$. Applying Theorem \ref{thm.integral-rational-compatibility}, we find it could also be described directly by applying the inverse Fourier transform on the generic fiber to the same restriction of the modified Weierstrass function. We explain now how to use $\Eis^{(n)}_\eta$ to construct interesting weight families. 

\begin{Eg}\label{eg.overconvergent-eisenstein-series} For $\epsilon=p^{-n}$, consider the $\epsilon$ neighborhood $\mathbb{Z}_{p,\epsilon}$ of $\ul{\mathbb{Z}_p}$ in $\mathbb{A}^1$ representing the functor $(R,R^+) \mapsto \ul{\zp}(R,R^+) + p^{n}R^{+}$. We also consider the subfunctor $\mathbb{Z}_{p,\epsilon}^{\times} \subset \mathbb{Z}_{p,\epsilon}$ representing $(R,R^+) \mapsto \ul{\zp^{\times}}\cdot (1+p^n R^+)$; this is an abelian group under multiplication.

For $\Spa (R,R^+)/ \Spa (\mathbb{Q}_p, \mathbb{Z}_p)$ affinoid perfectoid and a pair $(E, \eta)$ consisting of an elliptic curve $E/R^{+}$ and a non-vanishing differential $\eta$ on $E$, we say $(E,\eta)$ is $\epsilon$-ordinary if $\mathbb{Z}_{p,\epsilon} \cdot \HT(T_p E^\vee) = \mathbb{Z}_{p,\epsilon} \cdot \eta$. For a character $\kappa:\mathbb{Z}_{p,\epsilon}^{\times} \to \mathbb{G}_{m}^{\mathrm{an}}$, a $\mathbb{Z}_{p,\epsilon}$-overconvergent modular form $F$ of weight $\kappa$ is a rule assigning to any $\epsilon$-ordinary pair $(E, \eta)$ over $(R,R^+)$ an element $F(E,\eta) \in R$ such that $F(E,c^{-1}\eta)=\kappa(c)F(E,\eta)$ for $c \in \mathbb{Z}_{p,\epsilon}^\times(R)$. We obtain for any $\kappa$ as above a $\mathbb{Z}_{p,\epsilon}$-overconvergent $p$-adic modular form of weight $\kappa \cdot z^2$ by 
\[ (E, \eta) \mapsto \int_{T_p E^\vee} \kappa_! \circ \frac{\HT}{\eta} \cdot d\Eis^{(n)}(E,\eta),\]
where $\kappa_!$ is the extension by zero of $\kappa$ from $\mathbb{Z}_{p,\epsilon}^\times$ to $\mathbb{Z}_{p,\epsilon}$. This overconverges the Eisenstein series of weight $\kappa \cdot z^2$ over the closure of the ordinary locus ($\epsilon=0$) as obtained from the integral theory above.
\end{Eg}

\begin{Eg}\label{eg.quaternionic-eisenstein-series}
For $A$ an order in a quadratic extension $L/\mathbb{Q}_p$ and for $\epsilon=p^{-n}$, consider the $\epsilon$-neighborhood $A_\epsilon$ of $\ul{A}$ in $\mathbb{A}^1_L$ as above, and similarly consider $A_{\epsilon}^{\times}$. We say a pair $(E,\eta)$ over an affinoid perfectoid $\Spa(R,R^+)/\Spa (L, \mathcal{O}_L)$ is $A_\epsilon$-adapted if $A_\epsilon \cdot \HT(T_p E^\vee)= A_\epsilon \cdot \eta$. For $\kappa:A_\epsilon^\times \to \mathbb{G}_m^{\mathrm{an}}$ a character, an $A_\epsilon$-overconvergent modular form of weight $\kappa$ is a rule assigning to any $A_\epsilon$-adapted $(E,\eta)$ over $(R,R^+)$ an element $F(E,\eta) \in R$ such that $F(E, c^{-1}\eta)=\kappa(c)F(E,\eta)$ for $c \in A_{\epsilon}^\times(R)$. We obtain for any $\kappa$ an $A_\epsilon$-overconvergent $p$-adic modular form of weight $\kappa \cdot z^2$ by 
\[ (E, \eta) \mapsto \int_{T_p E^\vee} \kappa_! \circ \frac{\operatorname{HT}}{\eta}\cdot d\Eis^{(n)}{(E,\eta)},\]
where $\kappa_!$ is the extension by zero of $\kappa$ from $A_{\epsilon}^\times$ to $A_\epsilon$. This overconverges a modular form on the $A$-formal CM locus ($\epsilon=0$), a (twisted) profinite subset of the supersingular locus of the rigid analytic $p$-adic modular curve over $L$ isomorphic\footnote{This isomorphism follows, e.g., from \cite[Corollary 3.4.7, proof of Lemma 5.6.1]{HoweThesis}.} after an extension of scalars to
\[ D^\times(\mathbb{Q}) \backslash D^\times(\mathbb{A}_f) / (A^\times \times K^p),\]
where $D$ is a quaternion algebra over $\mathbb{Q}$ ramified at $p$ and $\infty$. Note that, by construction, the characters of $A^\times$ that extend to some $A_\epsilon^\times$ are precisely the $L$-analytic characters of $A^\times$. Note also that the existence of these analytic Eisenstein series on the quaternionic double quotient is already new. Specializing to genuine CM points inside the formal CM locus, one recovers the Schneider--Teitelbaum $p$-adic $L$-functions.
\end{Eg}

\begin{Q}
    Are the quaternionic Eisenstein series of Example \ref{eg.quaternionic-eisenstein-series} Hecke eigenvectors? If so, what are their Hecke eigenvalues, and what is the corresponding Galois representation (which will exist by \cite[Corollary B]{Howe.p-adic-JL})? 
\end{Q}

\appendix 
\section{Comparison with classical functional analysis}
\label{App:A}
In this appendix, we compare the constructions of Section \ref{Sec:Condensed} with those appearing in classical functional analysis. In particular, we extend (parts of) \cite{RJRC} to the setting where the base is a general $\mbb{Q}_p$-Banach algebra (not just a finite extension of $\qp$). These comparison results are used throughout the main body in the text (e.g. in the proof of Lemma \ref{Lem:TheAmiceCase}), and in \S \ref{Sub:Field},\ref{Sub:STComparison}. In \S \ref{sub:TheAmiceTransform}. we discuss the classical Amice transform as it is used in the proof of Theorem \ref{Thm:Main}.

\subsection{Comparison with classical functional analysis}
\label{subsec:ComparisonWithClassicalFA}

\subsubsection{} \label{subsub:ClassicalObjectDefs}
We first recall some objects which appear in classical functional analysis. Let $(R,R^+)$ be a uniform Huber pair over $(\qp,\zp)$. A topological $R$-module $M$ is \emph{Banach} if $M=M^+[\tfrac{1}{p}]$ where $M^+$ is a $p$-torsion free $p$-adically complete and separated $R^+$-module with the $p$-adic topology. A topological $R$-module $M$ is \emph{orthonomalizable Banach} (resp. \emph{free Smith}) if $M$ is isomorphic, as a topological $R$-module, to $\widehat{\bigoplus}_I R$ (resp. $\left( \prod_{I} R^+ \right)[\tfrac{1}{p}]$). Here the former module is equipped with the $p$-adic topology, and the topology on the latter module is induced from the product topology. We say such a module is \emph{countably orthonormalizable Banach} (resp. \emph{countably free Smith}) if the index set $I$ can be taken to be a countable set. 

\begin{Lem} \label{Lem:BanachSmithAEquiv}
    The functor $\operatorname{Mod}_R^{\operatorname{cond}} \to R\mathrm{-Mod}$ given by $M \mapsto M(*)_{\operatorname{top}}$ induces an equivalence (with inverse $V \mapsto \underline{V}$) between solid and classical $?$ modules, where $?$ can be:
    \begin{itemize}
        \item direct summands of orthonormalizable Banach;
        \item direct summands of countably free Smith.
    \end{itemize}
\end{Lem}
\begin{proof}
    It is enough to establish the proposition for orthonormalizable/free Banach/Smith modules. In this case, it follows from Lemma \ref{Lem:CompactlyGenerated} and Lemma \ref{Lem:Fullyfaithful}, together with the fact that the counit of the adjunction $(-)(*)_{\operatorname{top}} \dashv \underline{(-)}$ is a natural isomorphism on compactly-generated topological spaces.  
\end{proof}

\subsubsection{} We call a topological $R$-module $M$ \emph{strongly (countably) Fr\'echet} if it is isomorphic, as a topological $R$-module, to an inverse limit $\varprojlim_{n \in \mathbb{Z}_{\ge 0}} M_n$ of topological direct summands of (countably) orthonomalizable Banach modules $M_n$ with dense transition maps.  We have the following analogue of Lemma \ref{Lem:BanachSmithAEquiv} for strongly countably Fr\'echet spaces. 
\begin{Prop} \label{Prop:FrechetVSClassicalFrechet}
    The functor $\operatorname{Mod}_R^{\operatorname{cond}} \to R\mathrm{-Mod}$ given by $M \mapsto M(*)_{\operatorname{top}}$ induces an equivalence (with inverse $V \mapsto \underline{V}$) between solid and classical strongly (countably) Fr\'echet modules. 
\end{Prop}
\begin{proof}
    Since the index sets are countable, strongly countably Fr\'echet spaces are compactly-generated topological spaces. By Lemma \ref{Lem:Fullyfaithful}, it therefore suffices to prove the essential image of the functor $V \mapsto \underline{V}$ from (classical) strongly countably Fr\'echet spaces coincides with the category of condensed strongly countably Fr\'echet spaces. But this follows from Lemma \ref{Lem:BanachSmithAEquiv} and the fact that $\underline{(-)}$ commutes with inverse limits. 
\end{proof}

\begin{Rem}
    The reason we do not define LB spaces over a general $R$ is because of subtleties with direct limits of topological $R$-modules. Over a field, one works with the locally convex inductive limit, and developing this theory over an arbitrary Banach $\qp$-algebra $R$ is beyond the scope of this appendix.
\end{Rem}

\subsubsection{} If $M=M^+[\tfrac{1}{p}],N=N^+[\tfrac{1}{p}]$ are Banach $R$-modules, then we define their completed projective tensor product to be
\begin{align}
    M \widehat{\otimes}_{R} N := \left( M^+ \widehat{\otimes}_{R^+} N^+ \right)[\tfrac{1}{p}], 
\end{align}
where $\widehat{\otimes}_{R^+}$ denotes the $p$-adically completed tensor product. This does not depend on $M^+$ and $N^+$.

\begin{Lem} \label{Lem:DenseImageTensorTopological}
Let $M_1, M_2$ and $N$ be Banach $R$-modules. If $f:M_1\to M_2$ is a continuous $R$-linear morphism with dense image, then the map
    \begin{align}
       M_1 \widehat{\otimes}_{R} N \to M_2 \widehat{\otimes}_{R} N
    \end{align}
    has dense image.
\end{Lem}
\begin{proof}
Write $N=N^+[\tfrac{1}{p}]$ and $M_i=M_i^+[\tfrac{1}{p}]$. The density of $f$ is equivalent to the property that for all $k \in \mathbb{Z}$, the induced map
\[
g \colon M_1 \to M_2/p^k M_2^+
\]
is surjective. We therefore see that 
\begin{align*}
M_1 \otimes_{R^+} N^+ &\to \left( M_2/p^k M_2^+ \right) \otimes_{R^+} N^+ \\ 
&= (M_2 \otimes_{R^+} N^+)/p^k (M_2^+ \otimes_{R^+} N^+) \\ 
&= (M_2 \otimes_{R} N)/p^k (M_2^+ \otimes_{R^+} N^+) \\ 
&= (M_2 \widehat{\otimes}_R N)/p^k(M_2^+ \widehat{\otimes}_{R^+} N^+)
\end{align*}
is surjective, where the first equality uses that $M_2^+$ is $p$-torsion free (here, by abuse of notation $(M_2 \otimes_{R^+} N^+)/p^k (M_2^+ \otimes_{R^+} N^+)$ means quotient by the image of $p^k (M_2^+ \otimes_{R^+} N^+)$ in $M_2 \otimes_{R^+} N^+$), and the second equality uses that (the image of) $M_2^+ \widehat{\otimes}_{R^+} N^+$ is a lattice in $M_2 \widehat{\otimes}_R N$. This map is the natural one; in particular it factors through $M_1 \widehat{\otimes}_R N$ and the natural map
\[
M_1 \widehat{\otimes}_R N \to (M_2 \widehat{\otimes}_R N)/p^k(M_2^+ \widehat{\otimes}_{R^+} N^+)
\]
is surjective for all $k \in \mathbb{Z}$. This implies that $g \otimes 1 : M_1 \widehat{\otimes}_R N \to M_2 \widehat{\otimes}_R N$ has dense image. 
\end{proof}

\subsubsection{} \label{subsub:weakdual} Now assume that $R=K$ is a field and let $V$ denote a locally convex topological $K$-vector space in the sense of \cite[Section 4]{SchneiderFA}. We will write $V'_{w}$ for the $K$-vector space of continuous linear functions $M \to K$ equipped with the weak topology, and $V'_{b}$ for the same $K$-vector space equipped with the strong topology (see \cite[Section 6]{SchneiderFA}). 
\begin{Lem} \label{Lem:TopologicalDuals}
    If $M$ is a topological direct summand of a orthonormalizable Banach $K$-module, then there is a natural isomorphism
    \begin{align}
        \ul{M'_{w}} \simeq \Hom_{\ul{K}}(\ul{M}, \ul{K}).
    \end{align}
    If $N$ is a topological direct summand of a countably free Smith $K$-module, then there is a natural isomorphism
    \begin{align}
        \ul{N'_{b}} \simeq \Hom_{\ul{K}}(\ul{N}, \ul{K}).
    \end{align}
\end{Lem}
\begin{proof}
Using Proposition \ref{Prop:DualityCondensed} and the exactness of idempotent projection, it is enough to verify that the weak dual of $\widehat{\bigoplus}_I K$ (resp. the strong dual of $(\prod_I K^+)[1/p]$) coincides with $(\prod_I K^+)[1/p]$ (resp. $\widehat{\bigoplus}_I K$). The former is a direct computation, and the latter follows from the former and \cite[Corollary 13.8]{SchneiderFA}.
\end{proof}

\subsection{The Amice transform}  \label{sub:TheAmiceTransform}
For $V$ a $\mathbb{Q}_p$-Banach space, we write $\Cont(\zp, V)$ for the space of continuous functions from $\zp$ to $V$, equipped with the compact-open topology, so that $\Cont(\zp, V_0)$ is an open bounded lattice for any open bounded lattice in $V_0$, see Lemma \ref{Lem:CompactOpen} (if we put a norm on $V$, this is the topology induced by  the sup norm). For $h \geq 1$, we write $\mathrm{LA}_h(\zp, V)$ for the functions that are  analytic on each closed disk of radius $1/p^h$. We may write $\mathrm{LA}_h(\zp, V)=\bigoplus \mathrm{LA}_h(x + p^h\mathbb{Z}_p^h, V)$ as $x$ runs over residue classes for $\mathbb{Z}/p^h\mathbb{Z}$, and on each of these an open bounded lattice is given by formal sums with coefficients in $V_0$ going to zero of monomials in $(\frac{(z-x)}{p^h})$, for $V_0$ an open bounded lattice in $V$ (if we equip $V$ with the gauge norm associated to $V_0$, then this is the unit ball for the sup norm on coefficients).  

\begin{Lem}\label{lemma.completed-tensor-product-cont-and-la}
\[ \Cont(\zp, V)= \Cont(\zp, \mathbb{Q}_p) \widehat{\otimes}_{\mathbb{Q}_p} V \textrm { and } \mathrm{LA}_h(\zp, V) = \mathrm{LA}_h(\zp, \mathbb{Q}_p) \widehat{\otimes}_{\mathbb{Q}_p} V \]
\end{Lem}
\begin{proof}
We fix an orthonormal basis $e_i$ for $\Cont(\zp, \mathbb{Q}_p)$ by lifting a basis for $\Cont(\zp, \mathbb{F}_p)$ to locally constant functions on $\mathbb{Z}_p$. We fix a bounded open lattice $V_0$ in $V$ and an orthonormal basis $v_j$ for $V$ by lifting an $\mathbb{F}_p$-basis of $V_0/p$. Then, the functions $e_i \cdot v_j: z \mapsto e_i(z)v_j$ are an orthonormal basis for $\Cont(\zp, V)$ since they lie in the bounded open lattice $\Cont(\zp, V_0)$ and reduce mod $p$ to a basis of $\Cont(\zp, V_0/p)$. Thus the natural map 
\[ \Cont(\zp, \mathbb{Q}_p) \widehat{\otimes}_{\mathbb{Q}_p} V \rightarrow \Cont(\zp, V) \]
induced by the pairing $(f,v) \mapsto f \cdot v: z \mapsto f(z) v$ is an isomorphism. 

For $\mathrm{LA}_h(\mathbb{Z}_p, \bullet)$, it suffices to treat each $\mathrm{LA}_h(x+p^h \mathbb{Z}_p, \bullet)$ separately, where we find immediately that the $\frac{(z-x)}{p^h}^m v_j$ are an orthonormal basis and we conclude as above. 
\end{proof}  

\begin{Lem}[Mahler--Amice]\label{lemma.mahler-amice-bases}
The binomial coefficients $c_n: \zp \rightarrow \zp$, $c_n(z)=\binom{z}{n}$ are an orthonormal basis for $\Cont(\zp, \qp)$. The multiples $c_{n,h}=[\frac{n}{p^h}]! c_n$ are an orthonormal basis for $\LA_h(\zp, \qp)$. 
\end{Lem}
It follows, in particular, that the natural map $\mathrm{LA}(\zp, \qp) \rightarrow \Cont(\zp, \qp)$ is injective, and any $f \in \mathrm{LA}(\zp, \qp)$ has a unique M\"{a}hler expansion 
\[ f = \sum_{n \in \mathbb{Z}_{\geq 0}} a_n(f) \cdot c_{n} \textrm{ where } c_{n}(z)=\binom{z}{n}.  \] 

\subsubsection{} Let $\widehat{\mathbb{G}}_m$ be the formal multiplicative group over $\spf \zp$ and let $\gmhateta$ be its rigid generic fiber. We note that any $f \in \mathcal{O}(\gmhateta)$ has a unique power series expansion
\[ f=\sum_{n \in \mathbb{Z}_{\geq 0}}  a_n(f) \cdot t^{n}. \]
We have the universal character 
\[ \kappa=(1+t)^{z}=1+\binom{z}{1} t + \binom{z}{2} t^2 + \ldots = \sum_{n \in \mathbb{Z}_{\geq 0}} c_n \cdot t^{n} \in \Cont(\zp, \zp[[t]]). \]

\subsubsection{} For $r<1$ a rational power of $p$, we write $D_{r} \subseteq \gmhateta$ for the affinoid disk $|t| \leq p^{1/r}$. By restriction, we obtain an element $\kappa^{\locan}$ of
\[ \varprojlim_r \varinjlim_h \mathrm{LA}_h(\zp, \mathcal{O}(D_{r})). \]
Using Lemma \ref{lemma.completed-tensor-product-cont-and-la}, we may view $\kappa^\locan$ as an element of 

\begin{align} \varprojlim_r \varinjlim_h \mathrm{LA}_h(\zp, \mathcal{O}(D_{r})) & = \varprojlim_r \varinjlim_h \mathrm{LA}_h(\zp, \qp)\widehat{\otimes}_{\qp} \mathcal{O}(D_{r}).
\end{align}

\subsubsection{} Write $\mathrm{LA}(\zp, \qp)'_{b}$ for the continuous $\qp$-linear dual of $\mathrm{LA}(\zp, \qp)$ equipped with the strong topology. We let 
\[ \mathbb{F}_{\gmhateta}: \mathrm{LA}(\zp, \qp)'_{b} \rightarrow \varprojlim_r \mathcal{O}(\mathbb{D}_r) \]
denote the map obtained by contracting $\kappa^\locan$ with $\mu \in \mathrm{LA}(\zp, \qp)'$ (which induces a compatible family of elements $\mu_h$ in  $\mathrm{LA}_h(\zp, \mathbb{Q}_p)$). 

\begin{Cor}\label{cor.amice-bijections}
The map $\mathbb{F}_{\gmhateta}$ is a bijection.
\end{Cor}
\begin{proof}
The map can be described explicitly by 
\[ \mathbb{F}_{\gmhateta}(\mu)= \sum_{n \in \mathbb{Z}_{\geq 0}} \mu(c_n) t^n, \]
where we recall that $c_n$ is the binomial function. Comparing the growth conditions allowed on each side (using the Amice basis for locally analytic functions), one finds it is a bijection, cf. \cite[Theorem 1.8.4]{Colmeznotes}. 
\end{proof}


\end{document}